%% file: main.tex
\documentclass[twoside,11pt]{article}

\input{packages}
\input{commands}

\jmlrheading{26}{2026}{1-\pageref{LastPage}}{9/26}{-}{}{Florian Heinrichs, Luis A. Rodríguez}

\ShortHeadings{An FCLT for Locally Stationary Time Series in Banach Spaces}{F. Heinrichs, L. Rodríguez}
\firstpageno{1}

\begin{document}
	
\title{A Functional Central Limit Theorem for Locally Stationary Time Series in Banach Spaces}

\author{\name Florian Heinrichs \email f.heinrichs@fh-aachen.de \\
	\addr FH Aachen - University of Applied Sciences\\
	Department of Medical Engineering and Technomathematics
	\AND
	\name Luis A. Rodríguez \email luisalberto.rodriguez@uam.es \\
    \addr Universidad Autónoma de Madrid\\
    Department of Mathematics}

\editor{-}

\maketitle

\begin{abstract}%
    We establish a functional central limit theorem for locally stationary time series in a separable Banach space $B$. The result requires no type-2 or cotype assumptions and thus covers spaces central to functional data analysis, such as $C([0,1])$ and $L^p([0,1])$.
    
    Under moment conditions, summable physical dependence coefficients and a bracketing entropy condition controlling the infinite-dimensional tails, the centered partial sum process converges weakly in $D([0,1],B)$ to a centered Gaussian process, whose covariance is the integrated local long-run covariance, an element of the projective tensor product. We also obtain a stochastic integral representation with respect to a cylindrical Brownian motion.

    As the limit depends on the unknown long-run covariance, we prove consistency of a block multiplier bootstrap for the partial sum process. This yields CUSUM tests for a constant mean of the full process, based on arbitrary continuous seminorms and valid under a time-varying long-run covariance. For linear projections, we derive a self-normalized test with a pivotal limit. In simulations, the supremum norm test detects localized changes that its $L^2$ counterpart largely misses, and applications to EEG and temperature data illustrate the methodology. Examples in $C([0,1])$ and $W^{1,p}_0([0,1])$ show how the entropy condition can be verified.
\end{abstract}

\begin{keywords}
	Locally stationary time series; Banach-valued time series; functional central limit theorem; functional data analysis; block multiplier bootstrap; change point detection.
\end{keywords}

% MSC2020 classes
% Primary: 60F17, 60B12
% Secondary: 62R10, 62M10, 60F05

\maketitle 

\input{sec_intro}

\input{sec_methodology}

\input{app_example}

\input{sec_aux_results}

\input{sec_empirical_results}

\section*{Declaration of generative AI and AI-assisted technologies in the manuscript preparation process}

During the preparation of this work the authors used GPT-5.5 in order to check grammar and spelling, refine the wording, and double-check mathematical proofs. After using this tool, the authors reviewed and edited the content as needed and take full responsibility for the content of the published article.

\section*{Conflicts of Interest}

The authors declare no conflicts of interest.

\section*{Data Availability Statement}

The EEG data from Section \ref{sec:eeg} are openly available in Zenodo at \url{https://doi.org/10.5281/zenodo.14860668}.
The temperature data from Section \ref{sec:temperature} are available in GitHub at \url{https://github.com/FlorianHeinrichs/Banach_Space_FCLT}. These data were derived from the following data provided by the Deutscher Wetterdienst (DWD) Climate Data Center, under the CC BY 4.0 license: \url{https://opendata.dwd.de/climate_environment/CDC/observations_germany/climate/10_minutes/air_temperature/}.

\bibliography{bibliography}

\newpage

\appendix

\input{sec_proofs}

\input{app_proofs}

\input{app_empirical_results}

\end{document}

%% file: packages.tex
\usepackage{blindtext}

\usepackage[preprint]{jmlr2e}

\usepackage[utf8]{inputenc}
\usepackage[T1]{fontenc}
\usepackage{lmodern}
\usepackage{amsmath}
\usepackage{paralist}
\usepackage{float}
\usepackage{accents}
\usepackage{color}
\usepackage{bm}
\usepackage{lscape}
\usepackage{subfig}
\usepackage{longtable}
\usepackage{multibib}
\usepackage{booktabs}
\newcites{suppl}{References} 
\usepackage{tikz}
\usetikzlibrary{positioning}
\usepackage{dsfont}
\usepackage{pgfplotstable}
\usepackage{hyperref}
\usepackage{lastpage}
\usepackage{multirow}
\usepackage{setspace}

%% file: commands.tex
\newcommand{\eps}{{\varepsilon}}
\renewcommand{\phi}{\varphi}
\newcommand{\R}{\mathbb{R}}
\newcommand{\Z}{\mathbb{Z}}
\newcommand{\N}{\mathbb{N}}

\newcommand{\pr}{\mathbb{P}}       
\newcommand{\ex}{\mathbb{E}}       
\newcommand{\var}{\textnormal{Var}} 
\newcommand{\cov}{\textnormal{Cov}}

\newcommand{\Gb}{\mathbb{G}}

\newcommand{\Ac}{\mathcal{A}}

\newcommand{\Cc}{\mathcal{C}}
\newcommand{\Dc}{\mathcal{D}}
\newcommand{\Ec}{\mathcal{E}}
\newcommand{\Fc}{\mathcal{F}}
\newcommand{\Gc}{\mathcal{G}}

\newcommand{\Kc}{\mathcal{K}}
\newcommand{\Lc}{\mathcal{L}}
\newcommand{\Nc}{\mathcal{N}}
\newcommand{\Rc}{\mathcal{R}}
\newcommand{\Tc}{\mathcal{T}}

\newcommand{\Oc}{\mathcal{O}}
\newcommand{\diff}{{\,\mathrm{d}}}

\newcommand{\convw}{\rightsquigarrow}

\newcommand{\id}{\mathds{1}}

\DeclareMathOperator*{\sgn}{sgn}

\newtheorem{assumption}[theorem]{Assumption}

%% file: sec_intro.tex
% !TeX spellcheck = en_US
%% ====================
\section{Introduction} \label{sec:intro}
%% ====================

The last two decades have seen substantial developments along two largely parallel lines. On the one hand, the theory of non-stationary time series with time-varying distributions, and on the other hand, the analysis of function-valued stochastic processes, in particular in Banach spaces. Locally stationary models, formulated as triangular arrays whose distributions vary smoothly over rescaled time $t\in[0,1]$, provide a flexible framework for time-varying dependence structures in econometrics, meteorology and related fields. At the same time, functional data analysis treats observations as curves or more general infinite-dimensional objects, and a substantial limit theory has been developed for Hilbert- and Banach-space valued stochastic processes.  

Despite this progress, general invariance principles for locally stationary Banach-space valued time series remain limited. Existing functional central limit theorems (FCLTs) for locally stationary time series are mainly available for real-valued or finite-dimensional processes \citep{wu2011,dahlhaus2019}, empirical processes \citep{phandoidaen2022}, or processes taking values in separable Hilbert spaces such as $L^2([0,1])$ \citep{bucher2020,kurisu2021}. Conversely, central limit theorems for Banach-space valued time series typically assume independence \citep{hoffmann1976,araujo1978,dudley1983,mies2026} or stationarity \citep{dette2020,bigot2024}, and often impose geometric conditions such as type 2, which exclude important spaces used in functional data analysis, in particular $C([0,1])$ equipped with the supremum norm. Recent exceptions include the locally stationary $C([0,1])$-valued FCLT of \cite{kutta2025} and the uniform relative CLT for non-stationary empirical processes of \cite{palm2025}, but these results do not provide a general invariance principle for the Banach-valued partial sum process itself in $D([0,1],B)$. In particular, the bracketing entropy condition of the latter does not apply to the unit ball of the dual space $B^*$, so that tightness cannot be established. Convergence of all scalar projections does not by itself imply tightness of the original Banach-valued process, and additional control of the infinite-dimensional tail is required.  

The aim of this paper is to close this gap. We consider triangular arrays $\{(X_{i,n})_{1\le i\le n}\}_{n\in\N}$ taking values in a separable Banach space $B$, where $X_{i,n}=H(i/n,\Fc_i)$, for a measurable map $H$ and the past of an i.i.d. innovation sequence $\Fc_i$. Dependence is measured by the physical dependence measure of \cite{wu2005}, which quantifies the effect of replacing one innovation by an independent copy and is often easier to verify than classical mixing conditions.  
For the partial sum process  
\begin{equation*}
    S_n(t)=\frac{1}{n}\sum_{i=1}^{\lfloor nt\rfloor}X_{i,n},\qquad t\in[0,1],    
\end{equation*}
we study the centered and rescaled process $G_n(t)=\sqrt{n}(S_n(t)-\ex[S_n(t)])$ as a random element of $D([0,1],B)$.  

Under moment conditions of order $2+\kappa$, $\kappa>0$, Lipschitz continuity in rescaled time, summable physical dependence coefficients, and an entropy condition controlling the infinite-dimensional tails, we prove that $G_n$ converges weakly in $D([0,1],B)$ to a centered Gaussian process $G$.  
The covariance structure of the limit is given by  
\begin{equation*}
    \cov(G(t),G(s))=\int_0^{t\wedge s}\sigma^2(u)\diff u,    
\end{equation*}
where $\sigma^2(u)$ is the local long-run covariance, interpreted as an element of the projective tensor product $B\hat{\otimes} B$. We do not impose non-degeneracy of $\sigma^2(u)$, so the limiting Gaussian process may be degenerate in some directions of the dual space.  

The Banach space setting is deliberately general. We assume that $B$ is separable and admits a Schauder basis, but we do not impose type-2 or cotype assumptions. This includes classical spaces such as $L^p([0,1])$, $1\le p<\infty$, and in particular covers $C([0,1])$, which is not of type 2 and is central in functional data analysis.

The proof combines finite-dimensional convergence of scalar projections with a tightness argument for the Banach-valued process. The finite-dimensional convergence is obtained by a big-blocks-small-blocks argument and moment bounds derived from the physical dependence measure. Tightness is established by combining finite-rank approximation with a dependence-adjusted maximal inequality for the infinite-dimensional tails. This approach avoids geometric type assumptions and requires only moments of order $2+\kappa$.  

The FCLT alone does not yet allow for statistical inference, since the limiting distribution depends on the local long-run covariance $\sigma^2(u)$, a time-varying element of $B\hat{\otimes} B$ whose estimation is delicate, in particular without an inner product structure. We therefore consider a block multiplier bootstrap in the spirit of \cite{bucher2020}, and show that the bootstrap approximations converge weakly in $D([0,1],B)$, jointly with $G_n$, to independent copies of $G$. In contrast to \cite{bucher2020}, who work in $L^2([0,1])$, our result requires neither a Hilbert space structure nor cumulant conditions. Tightness of the bootstrap approximations follows from a weighted version of the maximal inequality used for the FCLT, in which the bootstrap multipliers enter as time-dependent weights. As an application, we obtain CUSUM tests for a constant mean of the full $B$-valued time series, which allow for a time-varying long-run covariance and arbitrary continuous seminorms, such as the supremum and the $L^2$-norm in $C([0,1])$. Our simulations show that the choice of the norm matters. The test based on the supremum norm detects changes of moderate height that are localized in a small part of the domain, whereas the $L^2$-based test largely misses them. This illustrates the practical relevance of going beyond Hilbert spaces.

The remainder of the paper is organized as follows. Section \ref{sec:method} introduces the Banach space framework, the locally stationary model, the main FCLT, the change point application, and a bootstrap approximation of the partial sum process. Section \ref{sec:aux_results} collects auxiliary results on weak convergence, autocovariances and the local long-run covariance, and a maximal inequality. Section \ref{sec:empirical} contains simulations and empirical illustrations. %, while Section \ref{sec:proof} gives the proofs of the main results. Additional proofs of examples and auxiliary results, as well as additional simulation results, are deferred to the appendices.
Proofs of the main and auxiliary results and examples, as well as additional simulation results, are deferred to the appendices.

%% file: sec_methodology.tex
% !TeX spellcheck = en_US
%% ====================
\section{Main Results} \label{sec:method}
%% ====================

\subsection{Mathematical Preliminaries}

\textit{Moments in Banach Spaces.} Let $B$ be a separable Banach space with norm $\| \cdot \|_B$. The central limit theorem is a statement about the weak convergence of the sum of \textit{centered} random variables. The Bochner integral is a generalization of the Lebesgue integral to Banach spaces, and allows a natural definition of centered random variables in Banach spaces.

The Bochner integral is defined, similar to the Lebesgue integral, for simple functions and extended to a broader class of functions through approximating sequences of simple functions \citep[see, e.\,g.,][]{ryan2002}. Importantly, a random variable $X : \Omega \to B$ is Bochner integrable if and only if $X$ is Bochner measurable and $\ex[\|X\|_B] < \infty$, by Theorem 5.1 of \cite{janson2015}. Moreover, since $B$ is separable, Bochner measurability is equivalent to (Borel) measurability, so that every random variable $X$ in $B$ is Bochner measurable, by Remark 3.2 of \cite{janson2015}. 

Different approaches exist to define higher moments of random variables in a Banach space $B$ in terms of tensor products \citep[see][]{janson2015}. In the following, we use a definition based on the \textit{projective tensor product}. Recall that the tensor product $B \otimes B$ is the unique (up to an isomorphism) vector space, such that there is a bilinear map $\otimes: B \times B \to B \otimes B$ with the universal property that for every vector space $W$ and every bilinear map $b : B\times B \to W$ there is a unique linear map $L_b: B \otimes B \to W$ satisfying $L_b(x \otimes y) = b(x, y)$ for all $x, y \in B$. The \textit{projective tensor norm} on $B\otimes B$ is defined by
\[
	\| u \|_\pi = \inf\Big\{ \sum_{j} \| x_{1j}\|_B \|x_{2j}\|_B: u = \sum_{j} x_{1j} \otimes x_{2j} \Big\},
\] 
where the infimum is taken over all different representations of $u$. Let $B \hat{\otimes} B$ denote the completion of $B \otimes B$ with respect to $\| \cdot \|_\pi$. For a random variable $X$ in $B$ with $\ex[\| X \|_B^2] < \infty$, $X \otimes X$ is Bochner integrable with $\ex[X \otimes X] \in B \hat{\otimes} B$ by Theorem 6.7 of \cite{janson2015}. 

For $b_1,b_2\in B^*$, the algebraic tensor $b_1\otimes b_2$ defines a continuous linear functional on $B\otimes B$ by  
\begin{equation*}
    (b_1\otimes b_2)\Big(\sum_{j=1}^m x_j\otimes y_j\Big)
=\sum_{j=1}^m b_1(x_j)b_2(y_j).    
\end{equation*}
Since  
\begin{equation*}
    |(b_1\otimes b_2)u|\le \|b_1\|_{B^*}\|b_2\|_{B^*}\|u\|_\pi,    
\end{equation*}
this functional extends uniquely and continuously to $B\hat{\otimes}B$.  Consequently, if $X,Y$ are square-integrable $B$-valued random variables, then  
\begin{equation*}
    (b_1\otimes b_2)\ex[X\otimes Y] 
    =\ex[(b_1\otimes b_2)(X\otimes Y)]
    =\ex[b_1(X)b_2(Y)].    
\end{equation*}
This identity will be used repeatedly to pass from tensor-valued covariances to scalar covariances of linear projections.  

Similarly, one may define mixed moments of random variables $X, Y \in B$ by taking the expectation of $X \otimes Y$ in $B \hat{\otimes} B$ \citep[see Remark 6.23 in][]{janson2015}. For centered random variables $X$ and $Y$, we can define $\var(X) := \ex[X \otimes X]$ and $\cov(X, Y) := \ex[X \otimes Y]$.

Based on this definition of expectation, we can define the conditional expectation of a random variable given a $\sigma$-algebra. More specifically, let $(\Omega, \Ac, \pr)$ denote a complete probability space and let $\ex[\| X \|_B] < \infty$. Further, let $\Cc \subset \Ac$ denote a sub-$\sigma$-algebra. Then, a unique $\Cc$-measurable random variable $Z$ exists such that $\ex[X \id_C] = \ex[Z \id_C]$ for all $C\in \Cc$ \citep[see, e.\,g., Proposition 1.10 in][]{daprato2014}. As in the real-valued case, $\ex[X | \Cc] := Z$ is the conditional expectation of $X$ given $\Cc$.

\textit{Local Stationarity.} Recall the framework of local stationarity, as proposed by \cite {zhou2009}. 
Let $(E,\Ec)$ be a measurable state space, and let $(\eta_i)_{i\in\Z}$ be an i.i.d. sequence of $E$-valued random elements. Let $(\eta_i^*)_{i\in\Z}$ be an independent copy, and write $\Fc_i = (\eta_i,\eta_{i-1},\dots)$ for the past at time $i$. For $i \ge 0$, let $\Fc_i^*=(\eta_i,\dots,\eta_1,\eta_0^*,\eta_{-1},\eta_{-2},\dots)$, that is, the past in which the innovation at time $0$ is replaced by an independent copy. Let $H:[0, 1] \times E^{\N_0} \to B$ denote a (possibly non-linear) map, such that $H(t, \Fc_i)$ is measurable for all $t\in [0, 1], i\in\N$. 

The \textit{physical dependence measure} offers an alternative to classic mixing coefficients, that is usually simpler to verify. For a map $H$ with $\sup_{t\in[0, 1]} \ex[\|H(t, \Fc_i)\|_B^\rho] < \infty$, the physical dependence measure is defined by 
\[ \delta_\rho(H, i) = \sup_{t\in[0, 1]} \ex\big[\| H(t, \Fc_i) - H(t, \Fc_i^*)\|^\rho\big]^{1/\rho}. \]
For real-valued time series, bounds on $\delta_\rho(H, i)$ follow from strong mixing under moment assumptions \citep{hill2025}. Conversely, strong mixing follows from bounds on $\delta_1(H, i)$ under certain regularity conditions on $\eta$ \citep{heinrichs2026}. Hence, weak physical dependence and mixing conditions are equivalent in many situations.

Finally, a triangular array $\{(X_{i, n})_{1\le i\le n}\}_{n\in\N}$ is called \textit{locally stationary}, if there exists a map $H$, which is continuous in its first argument, such that $X_{i, n} = H(i/n, \Fc_i)$, for all $i = 1, \dots, n$ and $n\in\N$. The map $H$ is \textit{Lipschitz continuous with respect to the $L^2$-norm,} if
\[ \sup_{0 \le s < t \le 1} \ex \big[ \| H(t, \Fc_i) - H(s, \Fc_i) \|^2\big]^{1/2} / |t - s| < \infty. \]
If the expectation of $X_{i, n} = H(i/n, \Fc_i)$ exists, the additive decomposition
\begin{equation*}
    X_{i, n} = \mu_{i, n} + \eps_{i, n},
\end{equation*}
holds, where $\mu_{i, n} = \ex[X_{i, n}]$ and $\eps_{i, n} = X_{i, n} - \ex[X_{i, n}]$. Under local stationarity, $\mu_{i, n} = \mu(i/n)$, where $\mu(t) = \ex[H(t, \Fc_0)]$. Lipschitz continuity of $H$ with respect to the $L^2$-norm implies Lipschitz continuity of $\mu$. 

For Banach spaces $X$ and $Y$, we write $\Lc(X, Y)$ for the Banach space of bounded linear operators from $X$ to $Y$, equipped with the operator norm.

\subsection{A Functional Central Limit Theorem} \label{sec:fclt}

Let $B$ be a separable Banach space and $X_{i,n}$ a locally stationary time series. Further, let 
\[ S_n(t) = \frac{1}{n} \sum_{i=1}^{\lfloor tn \rfloor} X_{i, n} \]
denote the partial sum process, for $t\in [0, 1]$. Note that $S_n(t) \in B$, and $S_n \in D([0, 1], B)$, where the latter denotes the Skorokhod space of càdlàg functions from $[0, 1]$ to $B$. Define the process $G_n = \{G_n(t)\}_{t\in[0, 1]}$ in terms of $G_n(t) = \sqrt{n} (S_n(t) - \ex[S_n(t)])$. 

The following assumptions separate three issues that are usually intertwined in Banach-space invariance principles: finite-dimensional approximation, temporal dependence, and tightness of the infinite-dimensional tail. The first two assumptions ensure that finite-dimensional projections of the partial sums have the expected Gaussian limits, while the tail condition below replaces geometric assumptions on $B$ such as type 2.

We first impose a mild approximation structure on the state space. The Schauder basis is used only to construct finite-rank projections and to formulate a condition ensuring that the part of the process not captured by these projections is asymptotically negligible.

\begin{assumption} \label{assump:space}
    There exists a Schauder basis $(e_k)_{k \in \N}$ in $B$. Let  $(e_k^*)_{k \in \N} \subset B^*$ denote the corresponding coordinate functionals, i.\,e., $e_k^*(e_i) = \id(k=i)$, and $\Pi_k$ the partial sum projections defined by $\Pi_k(x) = \sum_{i=1}^k e_i^*(x) e_i$.
\end{assumption}

\begin{assumption} \label{assump:local_stationarity}
	For the triangular array $\{(X_{i, n})_{1\le i\le n}\}_{n\in\N}$, it holds $X_{i, n}= \mu(i/n) + \eps_{i, n}$, for a deterministic mean function $\mu:[0, 1]\to B$ and a centered, locally stationary error process $\{(\eps_{i, n})_{1\le i\le n}\}_{n\in\N}$ with map $H$, such that the following conditions are satisfied: 
	\begin{enumerate}
        \item The map $H$ is Lipschitz continuous with respect to the $L^2$-norm, and moments of order $2+\kappa$ are uniformly bounded, for some $\kappa > 0$, i.\,e., $\sup_{t\in[0, 1]} \ex[\|H(t, \Fc_0)\|^{2+\kappa}] < \infty$.
		\item $\sum_{k \in \Z} \Theta_{|k|} < \infty$ for $\Theta_k = \sum_{i=k}^\infty \delta_{2+\kappa}(H, i)$, for $\kappa > 0$ as in part 1 and $\delta_{2+\kappa}(H, i)$ is non-increasing.
    \end{enumerate}
\end{assumption}

Assumption \ref{assump:local_stationarity} contains the stochastic regularity conditions used for the scalar central limit arguments. The Lipschitz condition controls the variation of the local approximating stationary process in rescaled time, while the summability of the physical dependence coefficients controls serial dependence uniformly over $t\in[0,1]$.

\begin{remark}
    The (local) long-run variance of $H$ is defined as 
    \[ \sigma^2(t) = \sum_{i=-\infty}^\infty \cov\big(H(t, \Fc_i), H(t, \Fc_0)\big), \]
    for $t\in [0, 1]$. By Lemma \ref{lem:lrv_lipschitz}, sequences $(y_m)_{m\in\N}, (z_m)_{m\in\N}$ of Lipschitz continuous maps $y_m, z_m:[0, 1]\to B$ exist, with $\sup_{t\in[0, 1]} \sum_{m=1}^\infty \|y_m(t)\|_B^2 < \infty$, $\sup_{t\in[0, 1]} \sum_{m=1}^\infty \|z_m(t)\|_B^2 < \infty$
    and $\sigma^2(t) = \sum_{m=1}^\infty y_m(t) \otimes z_m(t)$.
    In particular, $\sigma^2: [0, 1] \to B \hat{\otimes} B$ is continuous and the series converges with respect to the projective norm.

    Note that we do not assume that $\sigma^2(t)$ is positive definite. This is a strictly weaker assumption than the usual requirement $\inf_{t\in[0, 1]}\sigma^2(t) > 0$ in the univariate case. In particular, $\sigma^2(t)$, and as such the limiting distribution in Theorem \ref{thm:fclt} may be degenerate in certain directions $b\in B^*$.
\end{remark}

If $B$ is a Hilbert space or a Banach space of type 2, the type-2 inequality, combined with symmetrization, provides $L^2$-bounds for sums of independent centered random elements. Such bounds can then be used, together with an appropriate tightness criterion, to establish tightness of the corresponding partial sum processes \citep[see][Theorem 10.5]{ledoux2013}. Since the present results are intended to cover Banach spaces beyond those of (co-)type 2, we instead impose a condition on the empirical complexity of the projected tails. This condition says that, after applying sufficiently high finite-rank projections, the remaining stochastic fluctuation is small uniformly in the partial sum parameter.

Let $B_1^*$ denote the unit ball in $B^*$, $I_B: B\to B$ the identity operator and $T_k = I_B - \Pi_k$. Moreover, define
\begin{equation} \label{eq:def_set_entropy}
    \Gc_k = \{f: B\times[0, 1] \to \R | f(x, u) = (b\circ T_k)(x) \cdot \id(u \le t), b\in B_1^*, t\in[0, 1]\},
\end{equation}
and
\begin{equation*}
    V_{n, k} = \|f\|_{2, n} + \sum_{i=1}^\infty \min\{\|f\|_{2, n}, D\, \delta_2(H, i)\}, \qquad \|f\|_{2, n} = \bigg(\frac{1}{n} \sum_{i=1}^n \ex[|f(X_{i, n}, i/n)|^2]\bigg)^{1/2},
\end{equation*}
where $D >0$ is a constant satisfying $D\ge \sup_{f\in\Gc_k} \|f\|_{2, n}$ uniformly in $k$ and $n$. 
Throughout the paper, entropy is used only in the bracketing sense. More precisely, $N_{[]}(\eps, \Gc_k, V_{n,k})$ denotes the smallest number of brackets $[\ell,u]=\{f:\ell\le f\le u\}$, with $V_{n, k}(u-\ell)\le \eps$, needed to cover $\Gc_k$. To keep notation light, we write $N(\eps, \Gc_k, V_{n,k})$ for this bracketing number; ordinary covering entropy (without brackets) is not used below \citep[see][Chapter 2.1 for further details]{vandervaart1996}.
Finally, introduce the notation 
\begin{equation*} 
    \beta(q) = \sum_{j=q}^\infty \delta_2(H, j), \quad q^*(x) = \min\{q\in \N: \beta(q) \le qx\}, \quad r(\delta) = \max\{r > 0 : q^*(r)r\le \delta\},    
\end{equation*}
and $m(n, \delta, k) = r(\delta/D) D \sqrt{n}/\sqrt{\max\{1, \log(k)\}}$ from \cite{phandoidaen2022}. Based on this notation, we can state a regularity condition on the tails $T_k(X_{i, n})$. 

\begin{assumption} \label{assump:entropy}
    It holds $\sup_{f\in\Gc_k} V_{n,k}(f) \le \sigma_k$, 
    \begin{equation*}
        I_k = \limsup_{n\to\infty} \int_0^{\sigma_k} \sqrt{1 \vee \log(N(\eps, \Gc_k, V_{n,k}))} \diff \eps < \infty,
    \end{equation*}
    for all $k\in \N$, and $I_k \to 0$, as $k\to\infty$. Moreover,
    \begin{equation*}
        \frac{1}{\sqrt{n}} \sum_{i=1}^n \ex\Big[\|T_k X_{i, n}\|_B \id\Big\{\|T_k X_{i,n}\|_B > \tfrac{1}{4} m(n, \sigma_k, N(\sigma_k/2, \Gc_k, V_{n,k}))\Big\}\Big] \le c_k,
    \end{equation*}
    uniformly in $n\in\N$, for some sequence $(c_k)_{k\in\N}$ with $c_k \to 0$, as $k\to\infty$.
\end{assumption}

The first part of Assumption \ref{assump:entropy} is a bracketing entropy integral for the class of linear functionals applied to the tail $T_kx$, and the requirement $I_k\to0$ ensures that this class becomes asymptotically negligible as $k\to\infty$. The second part is a truncation condition similar to Lindeberg's \citep[see][Section 3.13.1]{vandervaart1996}. It takes out rare large tail observations, which is the analog of the large-value term in the dependence-adjusted maximal inequality used below.

The next assumption fixes the block lengths used in the big-blocks-small-blocks argument for finite-dimensional convergence. It is formulated separately because it depends only on the decay rate of the physical dependence coefficients and not on the geometry of $B$.

\begin{assumption} \label{assump:sequences}
    Sequences $(\ell_n)_{n\in\N}$ and $(s_n)_{n\in\N}$ exist, with $\ell_n \to\infty$ and $s_n\to\infty$, such that $s_n = o(\ell_n)$ and $\sqrt{n} \Theta_{s_n} \to 0$, as $n\to\infty$.
    Moreover, with $\kappa >0$, as in Assumption \ref{assump:local_stationarity} (2), it holds $\ell_n^{2+2/\kappa} = o(n)$, as $n\to\infty$.
\end{assumption}

\begin{remark}
    The assumptions are rather mild. Assumption \ref{assump:space} only requires existence of a Schauder basis, but unlike most of the literature, does not make any further assumption about the geometry of the space, such as type 2 or cotype 2 assumptions. This setting includes $\big(C([0, 1]), \|\cdot\|_\infty)$ as a special case, which is of particular interest in the field of functional data analysis, but not of type 2. 

    Assumption \ref{assump:local_stationarity} is weaker than usual regularity conditions for non-stationary error processes \citep[see, e.\,g.,][]{zhou2013,bucher2021}. In contrast to the literature, $\delta_{2+\kappa}(H, i)$ is bounded instead of $\delta_4(H, i)$, and it only needs to decay sufficiently fast, not exponentially. Further, $H$ must be Lipschitz continuous with respect to the $L^2$- rather than the $L^4$-norm, and moments of order $2+\kappa$ must be uniformly bounded instead of fourth- or eighth-order moments. One could replace the bound on $\delta_{2+\kappa}(h)$ by a bound on $\delta_2(h)$, at the cost of a moment condition of order $4$; see, for example, the proof of (20) in \cite{heinrichs2025}.  
    
    Assumption \ref{assump:entropy} ensures that $H$ only has negligible mass in the infinite-dimensional tail, such that all relevant variability is asymptotically captured by finite-rank approximations. The (bracketing) entropy condition could be replaced by an assumption on the geometry of $B$, for example, that $B$ is of type $2$.

    Assumption \ref{assump:sequences} is standard \citep[see, e.\,g.,][]{heinrichs2025}, and satisfied for most weakly dependent time series. When $\kappa = 2$, $\ell_n$ must satisfy $\ell_n^3/n \to 0$. For $\kappa = 1$, $\ell_n^4$ must be of order $o(n)$, and for $\kappa = \frac{1}{2}$, $\ell_n^6/n \to 0$.
\end{remark}

While entropy conditions are standard in the literature, they are often difficult to verify. The time series in the following example satisfies Assumptions \ref{assump:local_stationarity} and, more interestingly, \ref{assump:entropy}.

\begin{example} \label{ex:continuous}
    Let $\psi(t) = \id(0 \le t < 1/2) - \id(1/2 \le t < 1)$ and define the Haar functions $\psi_{\ell,j}(t) = 2^{\ell/2} \psi(2^\ell t - j)$. Moreover, define $s_0(t) = 1$, $s_1(t) = t$ and $s_{\ell,j}(t) = 2^{1+\ell/2} \int_0^t \psi_{\ell,j}(u) \diff u$. Then $\{s_0, s_1\} \cup \{s_{\ell,j}\}_{0\le j < 2^\ell, \ell\in\N}$ is the Faber-Schauder system \citep{faber1910}. Let $\Pi_k$ denote the partial sum projection keeping all basis functions up to $\ell = k, 0 \le j < 2^\ell$, and $T_k = I_{C([0, 1])} - \Pi_k$. 

    Fix $\alpha \in (1/2, 1]$, and assume that $\{(X_{i,n})_{1\le i\le n}\}_{n\in\N}$ satisfies
    \begin{equation*}
        X_{i, n} = \sum_{\ell=0}^\infty \sum_{j=0}^{2^\ell - 1} 2^{-(\alpha + 1)\ell} \xi_{i,n,\ell, j} s_{\ell, j},
    \end{equation*}
    where the centered, real-valued random variables $\{\xi_{i, n, \ell, j}: 1\le i\le n, n\in\N, 0 \le j \le 2^\ell -1, \ell \in \N\}$ are (i) uniformly bounded and $m$-dependent in $i$, i.\,e., $\sup_{i, j,\ell, n} |\xi_{i, n, \ell, j}|\le 1$ and the sequence $(\xi_{i, n, \ell, j})_{1\le i\le n}$ is $m$-dependent, and (ii) locally stationary, for fixed $j, \ell, n$.
    
    Then, $\{(X_{i,n})_{1\le i\le n}\}_{n\in\N}$ satisfies Assumptions \ref{assump:local_stationarity} and \ref{assump:entropy}. For a verification of the assumptions, we refer to Appendix \ref{app:proof_example}.
\end{example}

An additional example in the Sobolev space $W^{1,p}_0([0,1])$ based on $p$-Laplace elliptic problems is given in Section \ref{sec:pde_example}.

We can now state the functional central limit theorem. The result combines convergence of all finite-dimensional linear projections with the tail control in Assumption \ref{assump:entropy} to obtain weak convergence of the full $B$-valued partial sum process.

\begin{theorem}\label{thm:fclt}
	Let $B$ be a separable Banach space satisfying Assumption \ref{assump:space} and $\{(X_{i, n})_{1\le i\le n}\}_{n\in\N}$ a locally stationary time series, satisfying Assumptions \ref{assump:local_stationarity}, \ref{assump:entropy} and \ref{assump:sequences}. Then, $G_n \convw G$ as a process in $D([0, 1], B)$, where $G=\{G(t)\}_{t\in[0, 1]}$ is a Gaussian process in $B$ with $\ex[G(t)]=0$ and $\cov(G(t), G(s)) = \int_0^{t \wedge s} \sigma^2(x) \diff x$, for $s, t\in [0, 1]$.
\end{theorem}

The covariance formula is to be understood in the projective tensor sense introduced above. Equivalently, for $b_1,b_2\in B^*$, the covariance of the scalar projections $b_1(G(t))$ and $b_2(G(s))$ is obtained by applying $b_1\otimes b_2$ to $\int_0^{t\wedge s}\sigma^2(x)\diff x$. No non-degeneracy assumption is imposed on $\sigma^2(u)$, so the limit may be degenerate in some directions of $B^*$.

Whenever $u\in B \hat{\otimes} B$ is interpreted as an operator, we use the canonical map  
\begin{equation*}
    \Phi:B\hat{\otimes} B\to \Lc(B^*,B),\qquad \Phi(u\otimes v)(b)=b(v)u,    
\end{equation*}
extended linearly and continuously. With this convention, $\ex[X\otimes Y]$ corresponds to the covariance operator $b\mapsto \ex[b(Y)X]$, while applying $b_1\otimes b_2$ to $\ex[X\otimes Y]$ gives $\ex[b_1(X)b_2(Y)]$. The FCLT itself only requires the tensor interpretation, but the operator notation is useful when discussing the stochastic integral representation of the limiting Gaussian process. This representation is not needed for the weak convergence statement itself, but it connects the limit to the usual Brownian motion representation familiar from scalar locally stationary time series.

\begin{corollary}\label{cor:representation_limit}
    Under the assumptions of Theorem \ref{thm:fclt}, the Gaussian process $G$ can be written as
    \[
        G(t) = \int_0^t \sigma(u) \diff W(u),
    \]
    for $t\in [0, 1]$, a cylindrical Brownian motion $W=\{W(t)\}_{t\in[0,1]}$ in a Hilbert space $H$ and a family of linear bounded operators $\{\sigma(u)\}_{u\in[0, 1]}$, such that $\sigma(u): H \to B$ with $\sigma^2(u) = \sigma(u) \sigma^*(u)$ and $\sigma^*(u): B^* \to H$, for $u\in [0, 1]$.
\end{corollary}

By Assumption \ref{assump:local_stationarity}, $\sigma^2(t) \in B\hat{\otimes} B$. According to the corollary, $\sigma(t): H \to B$ and $\sigma^*(t): B^* \to H$. Equality of the operator $\sigma(t) \sigma^*(t): B^* \to B$ and the tensor $\sigma^2(t)$ is to be understood in the canonical sense, using the identification $\Phi$. Hence, the equality $\sigma^2(t) = \sigma(t) \sigma^*(t)$ means that $\Phi(\sigma^2(t)) = \sigma(t) \sigma^*(t)$. 

\subsection{Change Point Detection} \label{sec:self_normalization}

A natural application of functional central limit theorems for dependent data is change point detection, where one tests whether the mean of a time series remains constant over time. By Theorem \ref{thm:fclt}, the limiting distribution of CUSUM-type statistics under the null hypothesis depends on the (operator-valued) long-run variance, whose estimation is delicate. This is particularly true if the long-run variance is allowed to vary over time. To avoid an estimation of the long-run variance, we use self-normalization as  proposed by \cite{shao2010}, which yields a pivotal limiting distribution. Accordingly, we focus on the setting of a constant long-run variance and the ``at most one change'' alternative, and conduct inference for changes in the mean of projected observations $b(X_{i, n})$. 

More specifically, let $b \in B^*$ be a linear functional with $\|b\|_{B^*} = 1$ and let $\mu_i = \ex[b(X_{i, n})]$. We are interested in the hypotheses
\begin{equation*}
    H_0: \mu_1 = \mu_2 = \dots = \mu_n \quad \mathrm{vs.} \quad H_1: \mu_1 = \dots = \mu_{\lfloor \theta n\rfloor } \neq \mu_{\lfloor \theta n\rfloor+1} = \dots = \mu_n,
\end{equation*}
for some $\theta \in (0, 1)$. Let $T_n(t) = b(S_n(t)) = \tfrac{1}{n}\sum_{i=1}^{\lfloor nt\rfloor} b(X_{i, n})$, and define the test statistic
\begin{equation*}
    \mathbf{T}_n = \max_{k=1}^{n-1} \frac{|D_n(k/n)|}{V_n(k/n)},
\end{equation*}
where
\begin{equation*}
    D_n(t) = \sqrt{n}\big( T_n(t) - \tfrac{\lfloor nt\rfloor}{n} T_n(1) \big)
\end{equation*}
and 
\begin{equation*}
    V_n^2(t) = \sum_{i=1}^{\lfloor nt\rfloor} \big( T_n(\tfrac{i}{n}) - \tfrac{i}{\lfloor nt\rfloor} T_n(t) \big)^2 + \sum_{i=\lfloor nt\rfloor + 1}^{n} \Big( T_n(1) - T_n(\tfrac{i}{n}) - \tfrac{n-i}{n-\lfloor nt\rfloor} \big(T_n(1) - T_n(t)\big) \Big)^2.
\end{equation*}
Finally, let $q_{1-\alpha}$ denote the $(1-\alpha)$-quantile of 
\begin{equation} \label{eq:test_limit}
    \sup_{t\in [0, 1]} \frac{|W(t) - t W(1)|}{\Big( \int_0^t \big(W(x) - \tfrac{x}{t}W(t)\big)^2 \diff x + \int_t^1 \big(W(1) - W(x) - \tfrac{1-x}{1-t}[W(1) - W(t)]\big)^2 \diff x  \Big)^{1/2}},
\end{equation}
where $W = \{W(t)\}_{t\in(0, 1)}$ denotes a standard Brownian motion in $\R$.

\begin{corollary} \label{cor:test_properties}
    Let $B$ be a Banach space satisfying Assumption \ref{assump:space}, and $\eps = \{(\eps_{i, n})_{1\le i\le n}\}_{n\in\N}$ be a centered error process satisfying Assumptions \ref{assump:local_stationarity}, \ref{assump:entropy} and \ref{assump:sequences}, and $b\in B^*$ with $\|b\|_{B^*} = 1$, such that $(b \otimes b)\sigma^2(u) = \sigma_b^2 \in (0, \infty)$ for all $u\in[0,1]$. Finally, let $X_{i, n} = \mu + \delta\, \id(i > \lfloor \theta n \rfloor) + \eps_{i, n}$, for some $\delta \in B$. Then, the decision rule
    \begin{equation*}
        \mathbf{T}_n > q_{1-\alpha}
    \end{equation*}
    defines a test that has asymptotically level $\alpha$ under $H_0$ and is consistent against $H_1$.
\end{corollary}

In the Banach-valued model $X_{i, n} = \mu + \delta\, \id(i > \lfloor \theta n \rfloor) + \eps_{i, n}$, the projected alternative $H_{1}$ holds if and only if $b(\delta) \neq 0$. If $b(\delta)=0$, the Banach-valued mean may change in a direction not detected by $b$, but the projected observations $b(X_{i,n})$ have a constant mean and satisfy the null hypothesis. Accordingly, the test has asymptotic level $\alpha$ in this case and is consistent only against changes visible through the chosen functional $b$.  

In the context of locally stationary time series, the ``at most one change'' setting is not the most appropriate. Instead, we are often interested in detecting gradual changes of a smoothly varying mean $\mu(t) = \ex[H(t, \Fc_0)]$, where the long-run variance might vary over time. One approach is based on the joint convergence of a bootstrap approximation and the partial sum process, as suggested by \cite{bucher2020} in the space $L^2([0, 1])$. A corresponding result for separable Banach spaces is established in Section \ref{sec:bootstrap}. Another option would be to extend self-normalization procedures, based on a suitable permutation of the sample as proposed by \cite{heinrichs2021}, from $\R$ to arbitrary Banach spaces. This approach requires its own central limit theorem, which can presumably be derived with the same arguments as Theorem \ref{thm:fclt}.

%% ====================
\subsection{Bootstrap Approximation} \label{sec:bootstrap}
%% ====================

The limiting process $G$ in Theorem \ref{thm:fclt} depends on the local long-run covariance $\sigma^2(u)$, which is unknown in practice. The self-normalization approach from Section \ref{sec:self_normalization} avoids its estimation for statistics based on a fixed linear projection, but does not directly extend to functionals of the full $B$-valued process, such as $\sup_{t\in[0, 1]}\|G_n(t) - t G_n(1)\|_B$. In the following, we therefore consider a block multiplier bootstrap in the spirit of \cite{bucher2020}.

Let $(r_n)_{n\in\N}$ and $(m_n)_{n\in\N}$ be sequences of positive integers, where $r_n$ denotes the block length of the bootstrap and $m_n$ the bandwidth of a local mean estimator. For $j=1, \dots, n$, define the local mean estimator
\begin{equation*}
	\hat{\mu}_{j, n} = \frac{1}{|B_j|} \sum_{h \in B_j} X_{j+h, n}, \qquad B_j = \{h \in \Z: -m_n \vee (1-j) \le h \le m_n \wedge (n-j)\},
\end{equation*}
and, for $i = 1, \dots, n$, the normalized block sums
\begin{equation*}
	\hat{Z}_{i, n} = \frac{1}{\sqrt{r_n}} \sum_{j\in I_i} (X_{j, n} - \hat{\mu}_{j, n}), \qquad I_i = \{i, \dots, (i+r_n-1) \wedge n\}.
\end{equation*}
Let $\{R_i^{(k)}\}_{i, k\in\N}$ be i.i.d. standard normally distributed random variables, which are independent of $(\eta_i)_{i\in\Z}$. For $k\in\N$, the bootstrap approximations of $G_n$ are defined by
\begin{equation*}
	\hat{G}_n^{(k)}(t) = \frac{1}{\sqrt{n}} \sum_{i=1}^{\lfloor nt\rfloor} R_i^{(k)} \hat{Z}_{i, n} = \frac{1}{\sqrt{n}} \sum_{i=1}^{\lfloor nt\rfloor} \frac{R_i^{(k)}}{\sqrt{r_n}} \sum_{j=i}^{(i+r_n-1) \wedge n} (X_{j, n} - \hat{\mu}_{j, n}),
\end{equation*}
for $t\in[0, 1]$. Conditionally on the data, $\hat{G}_n^{(1)}, \hat{G}_n^{(2)}, \dots$ are independent, centered Gaussian processes in $D([0, 1], B)$.

\begin{assumption} \label{assump:bootstrap}
	\begin{enumerate}
		\item The sequences $(r_n)_{n\in\N}$ and $(m_n)_{n\in\N}$ satisfy $r_n \to\infty$, $r_n = o(m_n)$ and $m_n = o(\sqrt{n})$, as $n\to\infty$.
		\item The mean function $\mu$ is piecewise Lipschitz continuous, i.\,e., there exist $J\in\N_0$, points $0 < \theta_1 < \dots < \theta_J < 1$ and a constant $L_\mu>0$, such that $\|\mu(t) - \mu(s)\|_B \le L_\mu |t-s|$, for all $s, t\in[0, 1]$ for which the interval $(s\wedge t, s\vee t]$ contains none of the points $\theta_1, \dots, \theta_J$.
	\end{enumerate}
\end{assumption}

Assumption \ref{assump:bootstrap} (1) implies $r_n m_n = o(n)$. Part (2) holds for Lipschitz continuous mean functions ($J=0$), and allows for finitely many jumps, as in the ``at most one change'' alternative of Section \ref{sec:self_normalization}. In particular, $\mu$ is bounded, and we write $\|\mu\|_\infty = \sup_{t\in[0, 1]}\|\mu(t)\|_B$.

For tightness of the bootstrap approximations, we need a variant of the entropy condition from Assumption \ref{assump:entropy}, which is uniform with respect to the observation times. Let $K_B = \sup_{k\in\N} \|\Pi_k\|_{op}$ denote the basis constant of the Schauder basis $(e_k)_{k\in\N}$, which is finite by the uniform boundedness principle, such that $\|T_k\|_{op} \le 1 + K_B$. Define the classes
\begin{equation*}
	\bar{\Gc}_k = \{f: B\times[0, 1] \to \R | f(x, u) = (b\circ T_k)(x), b\in B_1^*\},
\end{equation*}
which correspond to $\Gc_k$ from \eqref{eq:def_set_entropy} with $t=1$, and
\begin{equation*}
	\tilde{V}_{n, k}(f) = \|f\|_{2, \infty, n} + \sum_{i=1}^\infty \min\{\|f\|_{2, \infty, n}, D\, \delta_2(H, i)\},
\end{equation*}
where $\|f\|_{2, \infty, n} = \max_{1\le i\le n} \ex[|f(X_{i, n}, i/n)|^2]^{1/2}$.
In addition to the requirement in Section \ref{sec:fclt}, the constant $D$ is chosen such that $D \ge 1 + K_B$ and $D \ge \sup_{f\in\bar{\Gc}_k} \|f\|_{2, \infty, n}$, uniformly in $k$ and $n$, which is possible under Assumptions \ref{assump:local_stationarity} and \ref{assump:bootstrap}.

\begin{assumption} \label{assump:entropy_uniform}
	It holds $\sup_{f\in\bar{\Gc}_k} \tilde{V}_{n, k}(f) \le \tilde{\sigma}_k$, for all $n\in\N$,
	\begin{equation*}
		\tilde{I}_k = \limsup_{n\to\infty} \int_0^{\tilde{\sigma}_k} \sqrt{1 \vee \log(N(\eps, \bar{\Gc}_k, \tilde{V}_{n, k}))} \diff \eps < \infty,
	\end{equation*}
	for all $k\in\N$, and $\tilde{I}_k\to 0$, as $k\to\infty$.
\end{assumption}

Since $\|f\|_{2, n}\le \|f\|_{2, \infty, n}$, the bracketing numbers in Assumption \ref{assump:entropy_uniform} are taken with respect to a stronger norm than in Assumption \ref{assump:entropy}. The time-uniform norm is needed, since the bootstrap multipliers enter the maximal inequality of Section \ref{sec:max_inequality} as random weights of the observations, which might concentrate on observation times where brackets with respect to $\|\cdot\|_{2,n}$ are wide. In contrast to Assumption \ref{assump:entropy}, no truncation condition is required.

\begin{remark}
	The examples from Sections \ref{sec:fclt} and \ref{sec:pde_example} satisfy Assumption \ref{assump:entropy_uniform}. Indeed, the brackets $[\tilde{f}_{\tilde{b}, \tilde{t}} - \eps/2, \tilde{f}_{\tilde{b}, \tilde{t}} + \eps/2]$ constructed in Appendix \ref{app:proof_example} have width $\eps$, uniformly on $\Kc\times[0, 1]$, and the time discretization is not needed for the classes $\bar{\Gc}_k$. Moreover, $\tilde{V}_{n,k}(f) \le (m+1)\|f\|_{2, \infty, n}$ and $\|f\|_{2, \infty, n} \le A_k$, for $f\in\bar{\Gc}_k$, by $m$-dependence and \eqref{Eqn:ProjectionBound}. Hence, the arguments in Appendix \ref{app:proof_example} yield the same bound for $\log N(\eps, \bar{\Gc}_k, \tilde{V}_{n,k})$ as for $\log N(\eps, \Gc_k, V_{n,k})$. The example from Section \ref{sec:pde_example} follows analogously.
\end{remark}

The following theorem establishes the validity of the bootstrap, formulated as joint weak convergence of the partial sum process and the bootstrap approximations.

\begin{theorem} \label{thm:bootstrap}
	Let $B$ be a separable Banach space satisfying Assumption \ref{assump:space} and $\{(X_{i, n})_{1\le i\le n}\}_{n\in\N}$ a locally stationary time series, satisfying Assumptions \ref{assump:local_stationarity}, \ref{assump:entropy}, \ref{assump:sequences}, \ref{assump:bootstrap} and \ref{assump:entropy_uniform}. Then, for any $K\in\N$,
	\begin{equation*}
		\big(G_n, \hat{G}_n^{(1)}, \dots, \hat{G}_n^{(K)}\big) \convw \big(G, G^{(1)}, \dots, G^{(K)}\big)
	\end{equation*}
	in $D([0, 1], B)^{K+1}$, where $G$ is the Gaussian process from Theorem \ref{thm:fclt}, and $G^{(1)}, \dots, G^{(K)}$ are independent copies of $G$, which are independent of $G$.
\end{theorem}

\begin{remark}
	\begin{enumerate}
		% \item By Lemma 2.2 of \cite{bucher2019}, Theorem \ref{thm:bootstrap} implies that the conditional distribution of $\hat{G}_n^{(1)}$ given the data converges to the distribution of $G$ in probability, with respect to the bounded Lipschitz metric on $D([0, 1], B)$. Hence, the distribution of continuous functionals of $G$ can be approximated by the conditional distribution of the corresponding functionals of $\hat{G}_n^{(1)}$. For instance, under the null hypothesis of a constant mean, $\sqrt{n}\big(S_n(t) - \tfrac{\lfloor nt\rfloor}{n}S_n(1)\big) = G_n(t) - \tfrac{\lfloor nt\rfloor}{n} G_n(1)$, so that the CUSUM statistic $\sup_{t\in[0,1]}\big\|\sqrt{n}\big(S_n(t) - \tfrac{\lfloor nt\rfloor}{n}S_n(1)\big)\big\|_B$ can be compared with the bootstrap replicates $\sup_{t\in[0,1]}\|\hat{G}_n^{(k)}(t) - \tfrac{\lfloor nt\rfloor}{n}\hat{G}_n^{(k)}(1)\|_B$, for $k=1, \dots, K$. If the distribution function of the limiting statistic is continuous, the corresponding bootstrap $p$-values are asymptotically uniformly distributed under the null hypothesis \citep[see Corollary 4.3 in][]{bucher2019}.
        \item By Lemma 2.2 of \cite{bucher2019}, Theorem \ref{thm:bootstrap} implies that the conditional distribution of $\hat{G}_n^{(1)}$ given the data converges to the distribution of $G$ in probability, with respect to the bounded Lipschitz metric on $D([0, 1], B)$. Hence, the distribution of continuous functionals of $G$ can be approximated by the conditional distribution of the corresponding functionals of $\hat{G}_n^{(1)}$. An application to change point detection is given in Corollary \ref{cor:bootstrap_test} below.        
		\item If $\mu$ is constant, the local mean estimators $\hat{\mu}_{j,n}$ are unbiased, and the condition $m_n = o(\sqrt{n})$ in Assumption \ref{assump:bootstrap} can be replaced by $m_n = o(n)$. Otherwise, the bias of $\hat{\mu}_{j, n}$ is controlled in the norm of $B$ by the triangle inequality, which requires $m_n = o(\sqrt{n})$, since no type-2 inequality is available.
		\item In contrast to \cite{bucher2020}, neither cumulant conditions nor a Hilbert space structure are required. In the proof of tightness, the bootstrap multipliers enter as random, time-dependent weights in a weighted version of the maximal inequality from Section \ref{sec:max_inequality}, which requires the time-uniform entropy condition in Assumption \ref{assump:entropy_uniform}.
	\end{enumerate}
\end{remark}

As an application of Theorem \ref{thm:bootstrap}, we consider a CUSUM test for a constant mean of the $B$-valued time series, i.\,e., for the hypotheses
\begin{equation} \label{eq:hypotheses_bootstrap}
    H_0: \mu(t) = \mu(0) \text{ for all } t\in[0, 1] \quad \mathrm{vs.} \quad H_1: \mu(t) \neq \mu(0) \text{ for some } t\in[0, 1].
\end{equation}
In contrast to Section \ref{sec:self_normalization}, neither a fixed linear functional $b\in B^*$, nor a time-constant long-run variance, nor an ``at most one change'' alternative is required. To cover different norms simultaneously, let $\zeta: B\to[0, \infty)$ be a continuous seminorm, i.\,e., $\zeta$ is subadditive, absolutely homogeneous and satisfies $\zeta(x)\le C_\zeta \|x\|_B$, for some constant $C_\zeta > 0$ and all $x\in B$. Define the CUSUM statistic
\begin{equation*}
    \mathbf{C}_n = \sup_{t\in[0, 1]} \zeta\Big(\sqrt{n}\big(S_n(t) - \tfrac{\lfloor nt\rfloor}{n} S_n(1)\big)\Big) = \max_{1\le j\le n} \zeta\bigg(\frac{1}{\sqrt{n}} \sum_{i=1}^j \big(X_{i, n} - \bar{X}_n\big)\bigg),
\end{equation*}
where $\bar{X}_n = \frac{1}{n}\sum_{i=1}^n X_{i, n}$, and the corresponding bootstrap replicates
\begin{equation*}
    \hat{\mathbf{C}}_n^{(k)} = \sup_{t\in[0, 1]} \zeta\big(\hat{G}_n^{(k)}(t) - \tfrac{\lfloor nt\rfloor}{n} \hat{G}_n^{(k)}(1)\big),
\end{equation*}
for $k\in\N$. For a sequence $(K_n)_{n\in\N}$ of positive integers, the bootstrap $p$-value is given by
\begin{equation*}
    \hat{p}_n = \frac{1}{K_n} \sum_{k=1}^{K_n} \id\big(\hat{\mathbf{C}}_n^{(k)} > \mathbf{C}_n\big),
\end{equation*}
and $\mathcal{H}_0$ is rejected at level $\alpha\in(0, 1)$, whenever $\hat{p}_n \le \alpha$. Finally, let
\begin{equation*}
    \Lambda_\mu(t) = \int_0^t \mu(u) \diff u - t \int_0^1 \mu(u) \diff u,
\end{equation*}
for $t\in[0, 1]$, where the integrals are Bochner integrals. Note that $\Lambda_\mu\equiv 0$ under $H_0$.

\begin{corollary} \label{cor:bootstrap_test}
    Let $B$ be a separable Banach space satisfying Assumption \ref{assump:space}, $\{(X_{i, n})_{1\le i\le n}\}_{n\in\N}$ a locally stationary time series, satisfying Assumptions \ref{assump:local_stationarity}, \ref{assump:entropy}, \ref{assump:sequences}, \ref{assump:bootstrap} and \ref{assump:entropy_uniform}, $\zeta$ a continuous seminorm on $B$, and $(K_n)_{n\in\N}$ a sequence of positive integers with $K_n\to\infty$.
    \begin{enumerate}
        \item If $H_0$ holds, and there exists $b\in B^*$ with $|b(x)|\le \zeta(x)$, for all $x\in B$, and $\int_0^1 (b\otimes b)\sigma^2(u)\diff u > 0$, then $\hat{p}_n$ converges weakly to the uniform distribution on $[0, 1]$. In particular, $\lim_{n\to\infty} \pr(\hat{p}_n\le\alpha) = \alpha$, for all $\alpha\in(0, 1)$.
        \item If $\sup_{t\in[0, 1]}\zeta(\Lambda_\mu(t)) > 0$, then $\lim_{n\to\infty} \pr(\hat{p}_n\le\alpha) = 1$, for all $\alpha\in(0, 1)$.
    \end{enumerate}
\end{corollary}

\begin{remark}
    \begin{enumerate}
        \item Corollary \ref{cor:bootstrap_test} covers several tests at once. For $\zeta = \|\cdot\|_B$, the condition in part (1) only requires that $G$ is not degenerate, i.\,e., $\var(b(G(1))) > 0$, for some $b\in B^*$. Moreover, the condition in part (2) holds whenever $\mu$ is not almost everywhere constant, since $\Lambda_\mu\equiv 0$ implies $\mu(t) = \int_0^1 \mu(u)\diff u$, for almost all $t\in[0, 1]$, by Lebesgue's differentiation theorem. In particular, the test is consistent against multiple change points and gradual changes. For $B = C([0, 1])$, the choices $\zeta(f) = \|f\|_{L^p}$, for $p\in[1, \infty)$, and $\zeta(f) = \sup_{x\in A}|f(x)|$, for a subset $A\subset[0, 1]$, define continuous seminorms, such that $L^2$-type statistics and statistics restricted to a part of the domain are covered by the same result. For $\zeta = |b(\cdot)|$, with $b\in B^*$, we obtain a test for a constant projected mean $\ex[b(X_{i, n})]$, which, in contrast to Corollary \ref{cor:test_properties}, allows for a time-varying long-run variance.
        \item The local mean estimators $\hat{\mu}_{j, n}$ ensure that Theorem \ref{thm:bootstrap} remains valid under the alternative, such that the bootstrap replicates $\hat{\mathbf{C}}_n^{(k)}$ are stochastically bounded, whereas $\mathbf{C}_n$ diverges. If $\hat{\mu}_{j, n}$ was replaced by the global mean $\bar{X}_n$, the bootstrap replicates would be inflated by a factor of order $\sqrt{r_n}$ under the alternative, which reduces the power in finite samples.
        \item Equivalently, by Lemma 4.2 of \cite{bucher2019}, $H_0$ can be rejected if $\mathbf{C}_n$ exceeds the empirical $(1-\alpha)$-quantile of $\hat{\mathbf{C}}_n^{(1)}, \dots, \hat{\mathbf{C}}_n^{(K_n)}$.
    \end{enumerate}
\end{remark}

%% file: app_example.tex
% !TeX spellcheck = en_US
%% ====================
\subsection{A Sobolev-space example satisfying the entropy condition} \label{sec:pde_example}
%% ====================

In the following, we provide another example of a time series satisfying Assumption \ref{assump:local_stationarity} and \ref{assump:entropy}. For the sake of completeness, we recall some well known definitions in the context of partial differential equations. Denote by $W^{1,p}(D)$ the $p$-Sobolev space of order $1$, $p\in[1,\infty]$, on a bounded domain $D\subset\R^{d}$ \citep[see][Chapters 8 and 9]{Brezis2010}. The norm of this space is
\[
    \|u\|_{W^{1,p}(D)}=\Big(\|u\|_{L^{p}(D)}^{p}+\sum_{i=1}^{d}\left\|\partial_{i}\,u(x)\right\|_{L^{p}(D)}^{p}\Big)^{\frac{1}{p}},\quad u\in W^{1,p}(D),
\]
where the partial derivatives $\left\{\partial_{i}\right\}_{i\in\{1,\ldots,d\}}$ (and all the differential operators seen in this subsection) are understood in the weak or distributional sense. For $p=2$, the space is also known as the Hilbert-Sobolev space, and is a Hilbert space.

Define the $p$-Laplace operator $-\Delta_{p}:W^{1,p}(D)\to\left(W^{1,p}(D)\right)^{\ast}$ by
\[
    (-\Delta_{p}\,u)(v)=\int_{D}\|\nabla u(x)\|^{p-2}\langle\nabla u(x),\nabla v(x)\rangle\,\diff x.
\]
The $p$-Laplace operator is the G\^{a}teaux derivative of the non-linear functional
\[
    \Phi_{0}(u)=\frac{1}{p}\,\int_{D}\|\nabla u(x)\|^{p}\,\diff x,
\]
so it naturally appears in partial differential equations associated with optimization problems where the functional $\Phi_{0}$ is involved. Consider the family of elliptic problems $-\Delta_{p}u=F_{t}$ indexed by $t\in[0,1]$, where $\Delta_{p}$ is the $p$-Laplace operator, $u\in W_{0}^{1,p}(D)$ with $W_{0}^{1,p}(D)=\left\{u\in W^{1,p}(D)\ \text{s.t.}\ u|_{\partial D}\equiv0\right\}$ and $F_{t}\in (W_{0}^{1,p}(D))^{\ast}$, whose weak formulation is
\begin{equation}\label{Eqn:WeakProblem}
    \int_{D}\|\nabla u(x)\|^{p-2}\langle\nabla u(x),\nabla v(x)\rangle\,\diff x-F_{t}(v)=0,\quad v\in W_{0}^{1,p}(D).
\end{equation}
For fixed $t\in[0,1]$, solving the elliptic problem \eqref{Eqn:WeakProblem} is equivalent to finding the critical points of the functional
\begin{equation}\label{Eqn:DirichletEnergy}
    \Phi_{F_{t}}(u)=\frac{1}{p}\,\int_{D}\|\nabla u(x)\|^{p}\,\diff x-F_{t}(u).
\end{equation}
For the sake of simplicity, let us take $d=1$, so $\Delta_{p}u=(|u^{\prime}|^{p-2}\,u^{\prime})$, and $D=[0,1]$. Fix the Banach space $B=W_{0}^{1,p}([0,1])$, where we can use the equivalent norm $\|u\|_{B}=\left\|u^{\prime}\right\|_{L^{p}(D)}$, due to Poincar\'{e}'s inequality \citep[see][Corollary 8.13]{Brezis2010}. Recall that by Proposition 8.14 of \cite{Brezis2010}, we have that for every $b\in B^{\ast}$, there exist $f_{b,0},f_{b,1}\in L^{p^{\prime}}([0,1])$, such that for all $u\in B$, $b(u)=\int_{[0,1]} f_{b,0}(x)\,u(x)+f_{b,1}(x)\,u^{\prime}(x) \,\diff x$, where $p^{\prime}=p/(p-1)$ is the H\"{o}lder dual exponent. We will formally denote $b=f_{b,0}\oplus(-f_{b,1}^{\prime})$.

As in Example \ref{ex:continuous}, define $\psi(x)=\id(0 \le x < 1/2 )-\id(1/2 \le x < 1)$ and the set $\{e_{\ell,j}\}_{\ell,j}$, $\ell\in\N$, $j\in\{0,\ldots,2^{\ell}-1\}$ as $e_{\ell,j}(x)=2^{\ell/p}\,\int_{[0,x]}\psi(2^{\ell}\,s-j)\diff s$. Note that we are taking the elements of the Faber-Schauder basis that lie in $B$ appropriately normalized so $\|e_{\ell,j}\|_{B}=1$. Furthermore, it can be shown that the dual basis is given by $e_{\ell,j}^{\ast}=0\oplus(-2^{\ell\,(p-2)/p}\,e_{\ell,j}^{\prime\prime})$.

Let $(\eta_{j})_{j\in\Z}$ be independent and identically distributed random variables with $|\eta_{j}|\leq M$, and define
\[
    F_{i,t}=a(t)\,\sum_{k=0}^{\infty}\lambda^{k}\,\eta_{i-k}\,e_{k}^{\ast},
\]
where $a(t)=\exp(t/(t-1))\,\id_{[0,1)}(t)$, $|\lambda|<1$, and $\left\{e_{k}\right\}_{k\in\N}$ is the relabeling given by $k=2^{\ell}+j-1$.  The family of elliptic problems $\Delta_{p}u=F_{i,t}$, or equivalently in integral form
\begin{equation}\label{Eqn:WeakProblem1D}
    \int_{[0,1]}\left(\left|u^{\prime}(x)\right|^{p-2}\,u^{\prime}(x)-a(t)\,\sum_{k=0}^{\infty}\lambda^{k}\,\eta_{i-k}\,e_{k}^{\prime}(x)\right)\,v^{\prime}(x)\,\diff x=0,\quad v\in B,
\end{equation}
has a unique solution $u_{i, t}$. Define the model through $H(t,\Fc_{i})=u_{i,t}$, where $\Fc_{i}=(\eta_k)_{k \le i}$ satisfies the definition of Section \ref{sec:intro}. Then, $X_{i,n}=u_{i,i/n}$ satisfies Assumptions \ref{assump:local_stationarity} and \ref{assump:entropy}. For a verification of the assumptions, we refer to Appendix \ref{app:proof_example_pde}.

\begin{remark}
    \begin{enumerate}
        \item The $p$-Laplacian, $\Delta_pu=\operatorname{div}(|\nabla u|^{p-2}\nabla u)$, is a classical nonlinear counterpart of the Laplace operator, which is recovered for $p=2$. It arises in several physical models involving nonlinear constitutive laws, including non-Newtonian power-law fluids, flow through porous media, and torsional creep. In particular, its derivation from nonlinear extensions of Darcy's law is discussed in \cite{BenediktEtAl2018} and \cite{GirgEtAl2025}. Of particular interest here are its applications in image and data processing. Nonlinear $p$-Laplacian diffusion, as well as its discrete, nonlocal and graph-based counterparts, has been applied to image filtering, segmentation, inpainting and clustering, providing nonlinear alternatives to standard Laplacian diffusion and allowing different degrees of smoothing and edge preservation \citep[see, e.\,g.,][]{ElmoatazEtAl2015}. Moreover, a functional central limit theorem for random solutions can be transferred, via the continuous mapping or delta method, to relevant quantities derived from the solution, such as spatial averages or the $p$-Dirichlet energy \eqref{Eqn:DirichletEnergy}.
        \item Given the subset of the Faber-Schauder basis of the previous example $\{e_{\ell,j}\}_{\ell,j}$, $\ell\in\N$, $j\in\{0,\ldots,2^{\ell}-1\}$, it is worth adding a comment on the interpretation of the dual basis $\{e_{\ell,j}^{\ast}\}_{\ell,j}$. Recall that for every $\varphi\in B$,
        \[
            e_{\ell,j}^{\ast}(\varphi)=\int_{[0,1]}\varphi^{\prime}(x)\,e_{\ell,j}^{\prime}(x)\,\diff x=2^{\frac{\ell\,(p-2)}{p}}\,\left(-\varphi\left(\frac{j}{2^{\ell}}\right)+2\,\varphi\left(\frac{j+1}{2^{\ell+1}}\right)-\varphi\left(\frac{j+1}{2^{\ell}}\right)\right),
        \]
        and by definition of $B$ also that $\varphi(0)=\varphi(1)=0$. In other words, the dual basis $\{e_{\ell,j}^{\ast}\}_{\ell,j}$ is formally a combination of pointwise evaluations. This fact links directly to the foundations of functional data analysis \citep[see][Chapter 1]{Ramsay&Silverman2005}, where in this context our ambient space is $B^{\ast}$.
        \item It is worth mentioning that Sobolev spaces provide a natural framework for functional data whenever the information carried by the derivatives is relevant, since their topology simultaneously controls a function and its (weak) derivatives in $L^{p}$-norm. This perspective is closely related to classical functional data analysis, where smoothness, derivative estimation, roughness penalties and differential operators have played a central role since its early development \citep{Ramsay&Silverman2005}. More explicitly, Sobolev spaces have been used to incorporate jointly curves and their derivatives into functional statistical procedures \citep[see, e.\,g.,][]{Aletti2016}. This techniques remain active in recent functional classification and clustering methods based on derivative information \citep{Ahpine&Yao2025}.
        Sobolev spaces also provide a flexible hierarchy of state spaces through their classical embedding properties. On a bounded regular domain $D\subset\mathbb{R}^{d}$, $W^{1,p_{2}}(D)\subset W^{1,p_{1}}(D)$ for $p_{1}<p_{2}$, while $W^{1,2}(D)$ is a Hilbert space. Moreover, the Gagliardo--Nirenberg--Sobolev inequality yields $W^{1,p}(D)\hookrightarrow L^{p^{\ast}}(D)$, $p^{\ast}=d\,p/(d-p)$, for $p<d$, whereas Morrey's theorem yields $W^{1,p}(D)\hookrightarrow C^{0,1-d/p}(\overline{D})$ for $p>d$ \citep[see][Sections 5.6.1--5.6.2]{Evans2022}. These embeddings connect the Sobolev framework with the usual $L^{p}$- and continuous-function settings of functional data analysis. Furthermore, these inclusions are strict and show that the framework cannot be reduced to the classical Hilbertian scenario. For instance, if $p<\min(2, 2\,d/(d+2))$, $W^{1,p}(D)$ does not embed into $L^{2}(D)$ or $W^{1,2}(D)$; while even in the Morrey regime the supremum-norm topology does not control derivative-dependent functionals. Thus, Sobolev-valued functional data provide a natural setting when both the topology of ambient space and the differential structure of data are statistically relevant. 
    \end{enumerate}
\end{remark}

%% file: sec_aux_results.tex
% !TeX spellcheck = en_US
%% ====================
\section{Auxiliary Results} \label{sec:aux_results}
%% ====================

\subsection{General Banach-Space and Weak-Convergence Tools} \label{sec:general_banach_results}

The proof of the main theorem uses several tools from probability theory and weak convergence in Banach spaces, collected here for convenience. Some of these facts are classical in spirit, but we include short proofs in the appendix to make the tensor-valued covariance calculations self-contained. The first result identifies the finite-dimensional linear projections that determine the law of a tight random element in $D([0,1],B)$, while the following propositions record elementary facts about tensor products, symmetrization and finite-rank approximation. 

Weak convergence in $D([0,1],B)$ will be verified through scalar projections $b(X(t))$, with $b\in B^*$. The following proposition states that, for tight random elements, these projected finite-dimensional distributions determine the full law, which generalizes Theorem 1.5.4 of \cite{vandervaart1996} to the Banach space $D([0, 1], B)$.

\begin{proposition} \label{prop:measure_equality}
	Let $X$ and $Y$ be tight random elements in $D([0, 1], B)$. Then $X$ and $Y$ have the same distribution if and only if 
	\begin{equation}\label{eq:prop_fdd}
		\big(b_1(X(t_1)), \dots, b_p(X(t_p)) \big) \stackrel{\Dc}{=} \big(b_1(Y(t_1)), \dots, b_p(Y(t_p)) \big),
	\end{equation}
	for any $b_1, \dots, b_p \in B^*, t_1, \dots, t_p \in [0, 1]$ and $p\in \N$.
\end{proposition}

The next two elementary facts justify the covariance calculations in the projective tensor product. They are used repeatedly to separate independent terms and to bound tensor-valued covariances by ordinary second moments.  

\begin{proposition} \label{prop:independence}
    Let $X$ and $Y$ be random variables in $B$ with $\ex[\|X\|_B^2] < \infty$ and $\ex[\|Y\|_B^2] < \infty$. If $X$ and $Y$ are independent, then $\ex[X \otimes Y]= \ex[X] \otimes \ex[Y]$.
\end{proposition}

\begin{proposition} \label{prop:bound_proj_norm}
    Let $X$ and $Y$ be random variables in $B$ with $\ex[\|X\|_B^2] < \infty$ and $\ex[\|Y\|_B^2] < \infty$. Then, $\|\ex[X\otimes Y]\|_\pi \le \ex[\|X\|_B \|Y \|_B]$.
\end{proposition}

The following maximal inequality is a symmetrized Banach space version of the Lévy–Ottaviani inequality \citep[see][Appendix A.1]{vandervaart1996}. It bounds the expected maximum of partial sums of independent centered Banach-valued random variables by the expected norm of a Rademacher sum. This form is convenient because the randomized sum can later be controlled using finite-dimensional or entropy arguments.  

\begin{proposition} \label{prop:ottaviani}
    Let $X_1, \dots, X_n$ be centered, independent random variables in $B$. Then,
    \begin{equation*}
        \ex\bigg[ \max_{i=1}^n \bigg\| \sum_{j=1}^i X_j \bigg\|_B \bigg]     
        \le 8 \ex\bigg[ \bigg\| \sum_{i=1}^n R_i X_i \bigg\|_B \bigg],     
    \end{equation*}
    for independent Rademacher random variables $R_1, \dots, R_n$.
\end{proposition}

Finally, we record a deterministic approximation lemma for Lipschitz paths in the projective tensor product. It allows tensor-valued covariance functions to be approximated uniformly by Lipschitz finite-rank tensor paths, which is useful when passing between finite-dimensional projections and the full Banach-space covariance structure.  

\begin{proposition} \label{prop:lipschitz_finite_rank}
    Let $x: [0, 1] \to B \hat{\otimes} B$ be a Lipschitz continuous map and $\eps > 0$. Then, it exists a Lipschitz continuous map $\tilde{x}: [0, 1] \to B\otimes B$, such that $\tilde{x}(t)$ is a finite rank operator, for each $t\in[0, 1]$ and $\sup_{t\in[0, 1]} \| x(t) - \tilde{x}(t)\|_\pi \le \eps$.
\end{proposition}

\subsection{A Banach Space-Valued Maximal Inequality} \label{sec:max_inequality}

The tightness proof requires uniform control of the empirical process obtained by applying linear functionals to the high-dimensional tail $T_kX_{i,n}$. Since the Banach space is not assumed to be of type 2, this control is not derived from the geometry of $B$, but from a dependence-adjusted entropy bound for real-valued function classes on $B\times[0,1]$. The following maximal inequality is an adaptation of the empirical process inequality of \cite{phandoidaen2022}. It will be applied with $f(x,u)=b(T_kx)\id(u\le t)$, so that the supremum over $f$ corresponds to the norm of the projected tail of the partial sum process. For the bootstrap approximations of Section \ref{sec:bootstrap}, it will be applied to weighted functions $f(x,u)=\omega(u)\, b(T_kx)$, where the weights $\omega(i/n)$ are determined by the bootstrap multipliers.
All suprema over function classes are assumed to be measurable; alternatively, all expectations may be read as outer expectations.

Let $\Gc$ denote a set of measurable functions $f: B \times [0, 1] \to\R$ and define 
\begin{equation} \label{eq:def_emp_process}
	\Gb_n(f) = \frac{1}{\sqrt{n}} \sum_{i=1}^n f(X_{i, n}, i/n) - \ex[f(X_{i, n}, i/n)],
\end{equation}
for $f\in \Gc$. For $f\in \Gc$, let $W_i(f) := f(X_{i,n}, i/n)$. Further, recall that $\Fc_i = (\eta_k)_{k\le i}$ and $X_{i, n}= H(i/n, \Fc_i)$. Denote by $\Fc_i^{(k)}$ the filtration $\Fc_i$, where $\eta_{i-k}$ is replaced by an i.i.d. copy, and let $X_{i, n}^{(i-k)} = H(i/n, \Fc_i^{(k)})$ and $W_i^{(k)}(f) = f(X_{i,n}^{(i-k)}, i/n)$.

\begin{assumption} \label{assump:functions}
	There exist a nonincreasing sequence $\Delta(k)$ with $\sum_{k\in\N} \Delta(k) < \infty$, a constant $D > 0$, a measurable weight function $\omega:[0, 1]\to\R$ and a constant $D_{\omega, n}$ with
	\begin{equation} \label{eq:def_weight_constant}
		D_{\omega, n} \ge D \bigg(\frac{1}{n} \sum_{i=1}^n \omega(i/n)^2\bigg)^{1/2},
	\end{equation}
	such that
	\begin{equation*}
		\ex[\sup_{f\in\Gc} | W_i(f) - W_i^{(k)}(f)|^2]^{1/2} \le |\omega(i/n)|\, D\, \Delta(k),
	\end{equation*}
	for all $1\le i\le n$ and $k\in \N$, and
	\begin{equation} \label{eq:def_seminorm0}
		\sup_{f\in\Gc} \| f \|_{2, n} \le D_{\omega, n}, \quad \| f \|_{2, n} := \bigg(\frac{1}{n} \sum_{i=1}^n \ex[|W_i(f)|^2] \bigg)^{1/2}.
	\end{equation}
	The class $\Gc$, the weight function $\omega$ and the constant $D_{\omega, n}$ may depend on $n$.
\end{assumption}

In the unweighted case $\omega\equiv 1$, we choose $D_{\omega, n} = D$, such that Assumption \ref{assump:functions} requires $\sup_{1\le i\le n} \ex[\sup_{f\in\Gc} | W_i(f) - W_i^{(k)}(f)|^2]^{1/2} \le D \Delta(k)$, for all $k\in\N$, and $\sup_{f\in\Gc} \| f \|_{2, n} \le D$. Weights are only needed for the bootstrap approximations in Section \ref{sec:bootstrap}, where they are given by the bootstrap multipliers. Importantly, $\omega$ enters the subsequent results only through the quadratic mean in \eqref{eq:def_weight_constant}, and not through $\sup_{u\in[0, 1]}|\omega(u)|$.

Under Assumption \ref{assump:local_stationarity}, we may choose $\Delta(k) = \delta_2(H, k)$ whenever $|f(x, u) - f(y, u)| \le |\omega(u)| D \|x-y\|_B$, for all $f\in\Gc$, $x, y\in B$ and $u\in[0, 1]$. For $\omega\equiv 1$, this is the case, for example, whenever
\begin{equation} \label{eq:func_class}
	\Gc \subset \{f: B\times[0, 1] \to \R | f(x, u) = b(x) \id(u \le t), b\in B_1^*, t\in[0, 1]\}.
\end{equation}

Analogously to (2.4) of \cite{phandoidaen2022}, define the dependence-adapted seminorm
\begin{equation} \label{eq:def_seminorm}
	V_n(f) = \| f\|_{2, n} + \sum_{k=1}^\infty \min\{\|f\|_{2, n}, D_{\omega, n} \Delta(k)\}.
\end{equation}
Moreover, for $\delta>0$ and $k\in\N$, let $m_\omega(n, \delta, k) = r(\delta/D_{\omega, n}) D_{\omega, n} \sqrt{n}/\sqrt{\max\{1, \log(k)\}}$, where $r(\cdot)$ is defined as in Section \ref{sec:fclt}, with $\delta_2(H, j)$ replaced by $\Delta(j)$. In particular, $m_\omega = m$ for $\omega\equiv 1$, $D_{\omega, n} = D$ and $\Delta(j) = \delta_2(H, j)$.

The following theorem adapts the maximal inequality of \cite{phandoidaen2022} to the present setting of Banach-space valued observations. The empirical process remains real-valued, but the indexing functions are now defined on the Banach-space domain $B\times[0,1]$. The quantity $m_\omega(n, \delta, k)$ balances the finite-class truncation with the ``rare-event'' remainder in the theorem's proof.

\begin{theorem} \label{thm:max_inequality}
	Let $\{(X_{i, n})_{1\le i\le n}\}_{n\in\N}$ be a locally stationary time series satisfying Assumption \ref{assump:local_stationarity}. Moreover, let $\Gc$ be a class of measurable functions satisfying Assumption \ref{assump:functions}. Let $F:B\times [0, 1]\to [0, \infty)$ be an envelope of $\Gc$, i.\,e., $|f(x, t)|\le |F(x,t)|$ for all $x\in B, t \in [0, 1]$ and $f\in \Gc$. 
	
	If $\sup_{f\in\Gc} V_n(f) \le \sigma$, for some constant $\sigma > 0$, then 
	\begin{equation*}
		\ex[\sup_{f\in\Gc}|\Gb_n(f)|] \le c\bigg[\int_0^\sigma \sqrt{1 \vee \log(N(\eps, \Gc, V_n))} \diff \eps + \sqrt{n}\Big\|F \id\big(F > \tfrac{1}{4}m_\omega[n, \sigma, N(\sigma/2, \Gc, V_n)]\big)\Big\|_{1, n}\bigg],
	\end{equation*}
	where $c>0$ is a universal constant and $\|\cdot\|_{1, n}$ is defined analogously to $\|\cdot\|_{2, n}$ in \eqref{eq:def_seminorm0}.
\end{theorem}

Compared to Theorem 4.4 of \cite{phandoidaen2022}, the domain changed from $(\R^d)^{\N_0}\times[0, 1]$ to $B\times [0, 1]$. Besides the changed domain, the theorem is a special case of Theorem 4.4, where the factors $D_{f, n}(u)$ in their representation $f(z, u) = D_{f, n}(u)\bar{f}(z, u)$ are chosen as $D_{f, n} = \omega$, for all $f\in\Gc$, and their sequence $\Delta(k)$ is replaced by $D\Delta(k)$. In this case, both of their quantities $\mathbb{D}_n$ and $D_n^\infty$ can be chosen as $D_{\omega, n}/D$, such that their threshold $m(n, \delta, k)$ coincides with $m_\omega(n, \delta, k)$, and the factor $1 + D_n^\infty/\mathbb{D}_n + \mathbb{D}_n/D_n^\infty = 3$ in their bound is absorbed into the constant $c$. Since the bound is non-asymptotic, the class $\Gc$ and the weights $\omega$ may depend on $n$.

We will apply the theorem to families of functions as in \eqref{eq:func_class}. In this case, we may choose the envelope $F(x, t) = \|x\|_B$, and $F(x, t) = |\omega(t)|\|x\|_B$ for the corresponding weighted classes.

\subsection{Autocovariances and the Long-Run Variance} \label{sec:autocov_lrv}

This subsection collects covariance bounds implied by the physical dependence assumptions. The results have two roles in the proof of the FCLT. First, they allow $m$-dependent approximations in the big-blocks-small-blocks argument, and second, they show that the local long-run covariance $\sigma^2(t)$ is a well-defined continuous path in $B\hat{\otimes} B$.  

We first approximate the error process by conditioning on a finite block of recent innovations. The physical dependence coefficients control the error of this approximation uniformly over the triangular array and after applying arbitrary bounded linear functionals.  

\begin{proposition}  \label{prop:sn_dependence}
	Let Assumption \ref{assump:local_stationarity} be satisfied, $\eps_{i, n} := X_{i,n} - \ex[X_{i, n}]$, $\tilde{\eps}_{i, n} = \ex[\eps_{i, n} | \eta_i, \dots, \eta_{i-m}]$ and $b \in B^*$ with norm $\|b\|_{B^*}$. Then,
	\begin{equation*}  
		\sup_{1\le i \le n, n \in\N} \ex\big[ \big(b(\eps_{i, n}) - b(\tilde{\eps}_{i, n})\big)^2\big]^{1/2} \le \|b\|_{B^*} \Theta_{m}
	\end{equation*}
    and
    \begin{equation*}
        \sup_{1\le i \le n, n \in\N} \ex[ \|\eps_{i, n} - \tilde{\eps}_{i, n}\|_B^2]^{1/2} \le \Theta_{m}.
    \end{equation*}
    
\end{proposition}

The next result identifies the limiting covariance contribution of a long block of observations from a stationary approximation. It is the covariance calculation used in the big-blocks part of the proof of finite-dimensional convergence and generalizes Proposition 10 of \cite{heinrichs2025} to Banach spaces.

\begin{proposition}  \label{prop:lrv_approximation}
	Let Assumption \ref{assump:local_stationarity} be satisfied and $b_1, b_2 \in B^*$. Then,
	\begin{align*}
	    & \sup_{t\in [0, 1]} \bigg|\frac{1}{\ell_n}\sum_{j_1, j_2=1}^{\ell_n} \ex\big[ b_1\big(H(t, \Fc_{0})\big) b_2\big(H(t, \Fc_{j_2-j_1})\big)\big] - (b_1 \otimes b_2) \sigma^2(t)\bigg| \\
        & = \sup_{t\in[0, 1]} \ex\big[\big|b_1\big(H(t, \Fc_{0})\big)\big|^2\big]^{1/2} \sup_{t\in[0, 1]} \ex\big[\big|b_2\big(H(t, \Fc_{0})\big)\big|^2\big]^{1/2} \cdot o(1) = o(1),
	\end{align*}
    as $n\to\infty$.
\end{proposition}

To ensure that the long-run covariance is well-defined as a tensor, we also need summability of the tensor-valued autocovariances. The following proposition gives the required projective norm bound and its scalar counterpart for linear projections.  

\begin{proposition} \label{prop:cov_bound}
    Let Assumption \ref{assump:local_stationarity} be satisfied. Then
    \[ 
        \sup_{t\in[0, 1]}\| \cov\big(H(t, \Fc_i), H(t, \Fc_0)\big) \|_\pi \le \Theta_i \sup_{t\in[0, 1]} \ex[\| H(t, \Fc_0) \|_B^2]^{1/2}.
    \]
    Moreover, for any $b\in B^*$, it holds
    \[ 
        \sup_{t\in[0, 1]}\big| \ex\big[b\big(H(t, \Fc_i)\big) b\big(H(t, \Fc_0)\big)\big] \big| \le \|b\|_{B^*} \Theta_i \sup_{t\in[0, 1]}\ex\big[\big|b\big(H(t, \Fc_0)\big)\big|^2\big]^{1/2}. 
    \]
\end{proposition}

The final result of this subsection combines these covariance bounds with Lipschitz continuity of the map $H$. It shows that the local long-run covariance is itself continuous in $B\hat{\otimes} B$ and admits a representation by summable rank-one tensor paths.  

\begin{lemma}  \label{lem:lrv_lipschitz}
	Let Assumption \ref{assump:local_stationarity} be satisfied, then sequences $(y_m)_{m\in\N}, (z_m)_{m\in\N}$ of Lipschitz continuous maps $y_m, z_m:[0, 1]\to B$ exist, such that 
    \begin{equation} \label{eq:lrv_l2_summability}
        \sup_{t\in[0, 1]} \sum_{m=1}^\infty \|y_m(t)\|_B^2 < \infty \quad \text{and} \quad \sup_{t\in[0, 1]} \sum_{m=1}^\infty \|z_m(t)\|_B^2 < \infty,
    \end{equation}
    and 
    \begin{equation} \label{eq:lrv_lipschitz_decomposition}
        \sigma^2(t) = \sum_{m=1}^\infty y_m(t) \otimes z_m(t).
    \end{equation}
    In particular, $\sigma^2: [0, 1] \to B \hat{\otimes} B$ is continuous.
\end{lemma}

%% file: sec_empirical_results.tex
% !TeX spellcheck = en_US
%% ====================
\section{Empirical Results} \label{sec:empirical}
%% ====================

In the following, we study the finite sample properties of the self-normalized test from Section \ref{sec:self_normalization} and of the bootstrap test from Section \ref{sec:bootstrap}. In Section \ref{sec:simulation_study}, we generate synthetic functional data and in Section \ref{sec:real_data}, we apply the methodology to real data sets from meteorology and neuroscience. Additional simulation results are given in Appendix \ref{app:empirical_results}. The code for the experiments is available in the GitHub repository.\footnote{\url{https://github.com/FlorianHeinrichs/Banach_Space_FCLT}}

The purpose of this section is illustrative. The self-normalized statistic studied in Section \ref{sec:self_normalization} is theoretically justified for projected mean changes under a constant projected long-run variance and an ``at most one change'' alternative. The bootstrap test from Corollary \ref{cor:bootstrap_test} additionally allows for a time-varying long-run covariance and for arbitrary non-constant, piecewise Lipschitz continuous mean functions. Accordingly, the simulations below include models with time-varying scale, which serve as robustness experiments for the self-normalized test, but are covered by the theory for the bootstrap test. The data examples should likewise be interpreted as illustrations of the methodology rather than as a complete inferential analysis of the corresponding scientific questions.

\subsection{Monte Carlo Simulation Study} \label{sec:simulation_study}

We considered the model $X_{i, n} = \delta\, \mu(i/n) + \eps_{i, n}$, for $\delta \in [0, 1]$, a centered, locally stationary error process $\eps$ in $C([0, 1])$ and three mean functions $\mu: [0, 1]\to C([0, 1])$, namely an abrupt global change, an abrupt local change and a gradual change,
\begin{gather*}
    \mu^{\mathrm{glo}}(t) = \nu_h\, \id(t > 1/2), \qquad 
    \mu^{\mathrm{loc}}(t) = \nu^{\mathrm{loc}}\, \id(t > 1/2), \\
    \mu^{\mathrm{gra}}(t) = \nu_h \min\{1, \max\{0, 4t - 2\}\},
\end{gather*}
where $\nu_h(x) = \max\{-1, \min\{1, (1/2 - x)/h\}\}$ and $\nu^{\mathrm{loc}}(x) = \max\{0, 1 - |x - 1/4|/w\}$, for $x\in[0, 1]$, with $h = w = 0.05$. The function $\nu_h$ is a continuous version of the step function $x\mapsto\id(x \le 1/2) - \id(x > 1/2)$, such that the observations take values in $C([0, 1])$. While $\nu_h$ changes the whole curve, $\nu^{\mathrm{loc}}$ only affects a small neighborhood of $x = 1/4$, with $\|\nu^{\mathrm{loc}}\|_\infty = 1$ and $\|\nu^{\mathrm{loc}}\|_{L^2} = (2w/3)^{1/2} \approx 0.18$.

Let $(\eta_i^{(1)})_{i\in\N}$ and $(\eta_i^{(2)})_{i\in\N}$ denote i.i.d. sequences of Brownian motions and Brownian bridges, respectively. For the error process $\eps$, the choices 
\begin{equation*}
\begin{array}{llll}
	\text{(BM)} & \eps_{i,n} = \eta_i^{(1)}, 
	& \text{(BB)} & \eps_{i,n} = \eta_i^{(2)}, \\
	\text{(FAR-BM)} & \eps_{i,n} = \rho(\eps_{i-1,n}) + \eta_i^{(1)}, 
	& \text{(FAR-BB)} & \eps_{i,n} = \rho(\eps_{i-1,n}) + \eta_i^{(2)}, \\
	\text{(tvBM)} & \eps_{i,n} = \sigma(\tfrac{i}{n})\eta_i^{(1)}, & \\
	\text{(tvFAR1)} & \eps_{i,n} = \rho(\eps_{i-1,n}) + \sigma(\tfrac{i}{n})\eta_i^{(1)},
	& \text{(tvFAR2)} & \eps_{i,n} = \sigma(\tfrac{i}{n}) \rho(\eps_{i-1,n}) + \eta_i^{(1)},\\
\end{array}
\end{equation*}
as proposed by \cite{bucher2023}, were considered, where $\sigma(x) = x + \tfrac{1}{2}$ and $\rho$ denotes the integral operator $\rho(f) = \int_0^1 K(\cdot, x) f(x) \diff x$, for $K(y, x) = 0.3 \sqrt{6}\min\{x, y\}$. Due to the time-varying factor $\sigma(x)$, the error terms are non-stationary with time-varying long-run variance, which violates the self-normalization setting, but not the assumptions of Corollary \ref{cor:bootstrap_test}.

For the self-normalized test, we considered the functionals
\begin{equation*}
    b_0(f) = \int_0^1 f(x) \diff x \quad \mathrm{and} \quad b_1(f) = \int_0^{1/2}f(x) \diff x - \int_{1/2}^1 f(x) \diff x,
\end{equation*}
and denote the corresponding tests by SN-$b_0$ and SN-$b_1$. Note that $b_0(\nu_h) = 0$ and $b_1(\nu_h) = 1 - h$, so that $\ex[b_0(X_{i, n})]$ is constant under the global and the gradual change, whereas $\ex[b_1(X_{i, n})]$ changes, for $\delta\neq 0$. In contrast, the local change is hardly visible through both functionals, since $b_0(\nu^{\mathrm{loc}}) = b_1(\nu^{\mathrm{loc}}) = w$.

For the bootstrap test from Corollary \ref{cor:bootstrap_test}, we considered the seminorms $\zeta(f) = \|f\|_\infty$ and $\zeta(f) = \|f\|_{L^2}$, and denote the corresponding tests by BS-sup and BS-$L^2$, respectively. We chose the block length $r_n = \lceil n^{1/4}\rceil$ and the bandwidth $m_n = \lceil n^{0.45}\rceil$, which satisfy Assumption \ref{assump:bootstrap}, and $K_n = 200$ bootstrap replicates. In finite samples, the local centering reduces the variance of the block sums $\hat{Z}_{i, n}$ approximately by the factor $1 - r_n/(2m_n + 1)$, for weakly dependent data. Therefore, we multiplied $\hat{Z}_{i, n}$ by $(1 - r_n/(2m_n + 1))^{-1/2}$. Since this factor converges to $1$, Theorem \ref{thm:bootstrap} and Corollary \ref{cor:bootstrap_test} remain valid.

All functions were evaluated on an equidistant grid of $100$ points in $[0, 1]$, and integrals and norms were approximated accordingly. For each combination of $n\in\{100, 200, 500\}$, $\eps$, $\mu$ and $\delta\in\{0, 0.1, \dots, 1\}$, we simulated $1000$ functional time series and conducted all four tests at a level of $\alpha = 5\%$.

Table \ref{tab:level} displays the empirical rejection rates under the null hypothesis, i.\,e., for $\delta = 0$, for four representative error processes. As the model does not depend on $\mu$ for $\delta = 0$, we pooled the replications of the three mean functions, such that each rejection rate is based on $3000$ replications. The results for the remaining error processes are similar and are given in Table \ref{tab:level_additional} in Appendix \ref{app:empirical_results}.

% \begin{table}[ht]
%     \centering
%     \caption{Empirical rejection rates (in \%) under the null hypothesis ($\delta = 0$), based on $3000$ replications each. Values outside the $95\%$-range $[4.2, 5.8]$ of the corresponding binomial distribution are printed in bold.}
%     \label{tab:level}
%     \begin{tabular}{llrrrr}
%     \toprule
%     $n$ & $\eps$ & SN-$b_0$ & SN-$b_1$ & BS-sup & BS-$L^2$ \\
%     \midrule
%     100 & (BM) & 5.0 & 4.3 & 6.6 & 7.6 \\
%      & (tvBM) & 6.7 & 6.5 & 8.1 & 9.2 \\
%      & (FAR-BM) & 6.1 & 5.7 & 12.2 & 13.3 \\
%      & (tvFAR1) & 9.0 & 7.5 & 13.3 & 14.7 \\
%     \midrule
%     200 & (BM) & 5.5 & 5.2 & 7.6 & 7.8 \\
%      & (tvBM) & 6.9 & 7.6 & 6.9 & 6.9 \\
%      & (FAR-BM) & 5.9 & 5.8 & 10.9 & 11.5 \\
%      & (tvFAR1) & 7.9 & 7.7 & 12.0 & 13.3 \\
%     \midrule
%     500 & (BM) & 5.4 & 5.2 & 6.5 & 5.9 \\
%      & (tvBM) & 7.6 & 6.9 & 6.2 & 6.8 \\
%      & (FAR-BM) & 5.6 & 5.6 & 9.4 & 10.1 \\
%      & (tvFAR1) & 6.9 & 7.2 & 9.3 & 9.7 \\
%     \bottomrule
%     \end{tabular}
% \end{table}

{
\setlength{\tabcolsep}{3pt}
\begin{table}[ht]
    \centering
    \caption{Empirical rejection rates (in \%) under the null hypothesis ($\delta = 0$), based on $3000$ replications each.}
    \label{tab:level}
    \footnotesize
    \begin{tabular}{l|rrrr|rrrr|rrrr}
    \toprule
    $n$ & \multicolumn{4}{c|}{$100$} & \multicolumn{4}{c|}{$200$} & \multicolumn{4}{c}{$500$} \\
    $\eps$ & SN-$b_0$ & SN-$b_1$ & BS-sup & BS-$L^2$ & SN-$b_0$ & SN-$b_1$ & BS-sup & BS-$L^2$ & SN-$b_0$ & SN-$b_1$ & BS-sup & BS-$L^2$ \\
    \midrule
     (BM) & 5.0 & 4.3 & 6.6 & 7.6 & 5.5 & 5.2 & 7.6 & 7.8 & 5.4 & 5.2 & 6.5 & 5.9 \\
     (tvBM) & 6.7 & 6.5 & 8.1 & 9.2 & 6.9 & 7.6 & 6.9 & 6.9 & 7.6 & 6.9 & 6.2 & 6.8 \\
     (FAR-BM) & 6.1 & 5.7 & 12.2 & 13.3 & 5.9 & 5.8 & 10.9 & 11.5 & 5.6 & 5.6 & 9.4 & 10.1 \\
     (tvFAR1) & 9.0 & 7.5 & 13.3 & 14.7 & 7.9 & 7.7 & 12.0 & 13.3 & 6.9 & 7.2 & 9.3 & 9.7 \\
    \bottomrule
    \end{tabular}
\end{table}
}

For error processes with constant long-run variance, the self-normalized tests have approximately nominal level. Under a time-varying long-run variance, their rejection rates exceed the nominal level, and the distortion does not vanish as $n$ grows, e.\,g., the rejection rates for (tvBM) and $n = 500$ are $7.6\%$ (SN-$b_0$) and $6.9\%$ (SN-$b_1$). This is in line with the theory, since the limit distribution in \eqref{eq:test_limit} requires a constant long-run variance. The bootstrap tests are slightly liberal for (BM), with rejection rates between $5.9\%$ and $7.8\%$, and for (tvBM), their rejection rates approach the nominal level as $n$ grows. Under serial dependence, the bootstrap tests are clearly liberal for small samples, with rejection rates between $12.2\%$ and $14.7\%$ for $n = 100$, which decrease to between $9.3\%$ and $10.1\%$ for $n = 500$. This is due to the short blocks, which underestimate the long-run covariance, as is typical for block-based bootstrap procedures.

\begin{figure}[ht]
    \centering
    \includegraphics[width=\textwidth]{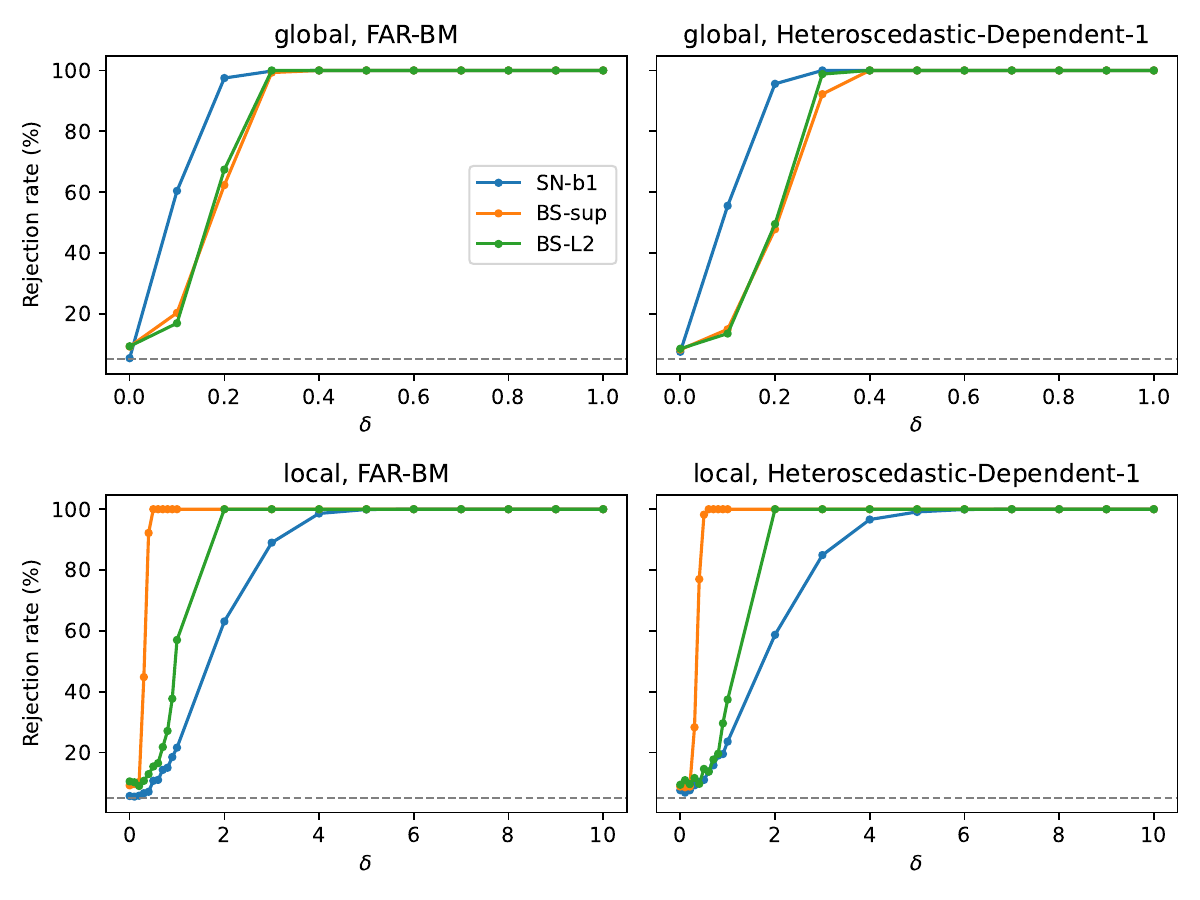}
    \caption{Empirical rejection rates (in \%) for $n = 500$ as functions of $\delta$, for the abrupt global change (top) and the abrupt local change (bottom), and the error processes (BM) (left) and (tvFAR1) (right). The dotted line marks the nominal level of $5\%$.}
    \label{fig:power}
\end{figure}

Figure \ref{fig:power} displays the empirical rejection rates for $n = 500$ as functions of $\delta$, for the abrupt global and local changes and the error processes (BM) and (tvFAR1). For the global change, the self-normalized test SN-$b_1$ is the most powerful, as $b_1$ is aligned with the direction of the change. For instance, for (BM) and $\delta = 0.2$, the rejection rates are $99.8\%$ (SN-$b_1$), $98.5\%$ (BS-$L^2$) and $87.7\%$ (BS-sup). Since the global change affects the whole curve, BS-$L^2$ is slightly more powerful than BS-sup. For the local change, this order is reversed. The test BS-sup detects the change already for small heights, whereas BS-$L^2$ and SN-$b_1$ have little power. For (BM) and $\delta = 0.4$, the rejection rates are $99.6\%$ (BS-sup), $9.4\%$ (BS-$L^2$) and $9.3\%$ (SN-$b_1$), and for (tvFAR1) and $\delta = 0.5$, they are $98.2\%$, $14.6\%$ and $11.0\%$, respectively. These differences in power by far exceed the level distortions reported in Table \ref{tab:level}. 

For smaller sample sizes, the same pattern emerges, and the results for the gradual change resemble those for the abrupt global change, with slightly lower power (see Appendix \ref{app:empirical_results}). The test SN-$b_0$ behaves as under the null hypothesis for the global and the gradual change, since $b_0(\nu_h) = 0$, and has little power against the local change, since $b_0(\nu^{\mathrm{loc}})$ is small.

\subsection{Case Study} \label{sec:real_data}

In the following, we apply the self-normalization procedure to detect changes in concentration levels and daily temperature curves.

\subsubsection{EEG-Data Analysis} \label{sec:eeg}

We employed the ``Consumer-Grade EEG and Eye-Tracking Dataset``, which contains EEG recordings with simultaneously measured eye movements from 116 sessions \cite{afonso2025}. The EEG was measured with 4 electrodes at a sampling rate of 256\,Hz, and we considered only the 102 recordings without technical difficulties. Each session includes four recordings with increasing degree of difficulty (level-1-saccades/-smooth, and level-2-saccades/-smooth).

As usual in neuroscience, we calculated a spectrogram of the raw time series with Hann windows of 2\,s and a step width of 0.5\,s, and applied the logarithm to the resulting spectra. The result was a time series in $\big(C([0, 128])\big)^4$, where the spectrum at each time step was considered a discretization of a continuous function and the frequency ranges between 0\,Hz and the Nyquist frequency 128\,Hz. As a measure of attention, we calculated the difference of average power in the $\beta$- and $\alpha$-frequency bands, i.\,e.,
\begin{equation*}
    b(f) = \frac{1}{4} \sum_{i=1}^4 \bigg( \frac{1}{30-13} \int_{13}^{30} f_i(x)  \diff x - \frac{1}{12 - 8} \int_8^{12} f_i(x) \diff x\bigg),
\end{equation*}
where $f_i$ denotes the spectrum of electrode $i \in\{1, \dots, 4\}$.

Since we are interested in a specific feature of the spectra, namely the attention measure $b$, we only applied the self-normalized test. The bootstrap tests from Section \ref{sec:bootstrap} detect changes of the full mean function and would also react to changes that are unrelated to attention, such as electrode drift or ocular artifacts, which mainly affect low frequencies.

The results are displayed in Table \ref{tab:eeg} for a nominal level of $5\%$. Interestingly, the number of EEG recordings with a detected change increases with an increasing difficulty of the task, which might be attributed to an attention drift.

\begin{table}
    \centering
	\caption{EEG recordings with attention drift out of 102 recordings per level.}	
	\label{tab:eeg}
	\begin{tabular}{rrrr}
        \toprule
		Level-1-Saccades & Level-1-Smooth & Level-2-Saccades & Level-2-Smooth \\
		\midrule
		8 & 18 & 24 & 22 \\
        \bottomrule
	\end{tabular}
\end{table}

\subsubsection{Temperature Curves} \label{sec:temperature}

We used temperature data from station measurements in Aachen (1993 to 2011), Bochum (2007 to 2024) and Göttingen (1993 to 2024), provided by the Deutscher Wetterdienst (DWD) Climate Data Center\footnote{\url{https://opendata.dwd.de/climate_environment/CDC/observations_germany/climate/10_minutes/air_temperature/}}. The air temperature was measured in 10-minute intervals, with time stamps in CET before 2000 and in UTC thereafter, which we converted to UTC. The data from each day was interpreted as discretization of a smooth daily temperature curve in $C([0, 24])$, and days with missing values were removed. Since the annual cycle corresponds to a rapidly oscillating mean function in rescaled time, we removed it by subtracting, separately for each time of the day, the least squares fit of $a_0 + a_1\cos(2\pi d/365.25) + a_2\sin(2\pi d/365.25)$, where $d$ denotes the day of the year.

To the resulting time series, we applied the self-normalized test based on
\begin{equation*}
    b_0(f) = \frac{1}{6}\bigg(\int_6^9 f(x) \diff x - \int_{16}^{19} f(x) \diff x\bigg) \quad \mathrm{and} \quad b_1(f) = \frac{1}{14}\int_6^{20} f(x) \diff x,
\end{equation*}
which compare temperatures in the morning and afternoon and calculate the mean temperature during the day, respectively, and the bootstrap tests based on the supremum and the $L^2$-norm, with $K_n = 1000$ and $r_n, m_n$ as in Section \ref{sec:simulation_study}. The tests were applied to the full time series and to the last year (365 days) of each series.

The $p$-values are displayed in Table \ref{tab:temperature}. For the full time series, both bootstrap tests and the self-normalized test based on $b_1$ reject the null hypothesis for Bochum and Göttingen, whereas $b_0$ does not indicate a change. Hence, the detected changes affect the level of the temperature curves rather than the difference between morning and afternoon. For Aachen, whose series ends in 2011, the $p$-values of the bootstrap tests are slightly above $5\%$. For the last year, only the bootstrap tests reject, again for Bochum and Göttingen. These results should be interpreted with care, since the bootstrap tests are liberal for serially dependent data and moderate sample sizes (see Section \ref{sec:simulation_study}).

\begin{table}[ht]
    \centering
    \caption{$p$-values of the tests for the temperature curves. $p$-values below $0.05$ are printed in bold. A bootstrap $p$-value of $0.000$ means that none of the $1000$ bootstrap replicates exceeded the test statistic.}
    \label{tab:temperature}
    \begin{tabular}{lrrrrrrrr}
    \toprule
    & \multicolumn{4}{c}{Full time series} & \multicolumn{4}{c}{Last year} \\
    \cmidrule(lr){2-5} \cmidrule(lr){6-9}
    Location & SN-$b_0$ & SN-$b_1$ & BS-sup & BS-$L^2$ & SN-$b_0$ & SN-$b_1$ & BS-sup & BS-$L^2$ \\
    \midrule
    Aachen & 0.822 & 0.508 & 0.081 & 0.089 & 0.425 & 0.968 & 0.063 & 0.086 \\
    Bochum & 0.904 & \textbf{0.014} & \textbf{0.000} & \textbf{0.000} & 0.605 & 0.192 & \textbf{0.031} & \textbf{0.011} \\
    Göttingen & 0.706 & \textbf{0.013} & \textbf{0.000} & \textbf{0.000} & 0.837 & 0.253 & \textbf{0.017} & \textbf{0.021} \\
    \bottomrule
    \end{tabular}
\end{table}

%% file: sec_proofs.tex
% !TeX spellcheck = en_US
%% ====================
\section{Proofs} \label{sec:proof}
%% ====================

\subsection{Proof of Theorem \ref{thm:fclt}}

By Lemma \ref{lem:fdd}, all finite-dimensional distributions obtained by applying finitely many linear functionals $b\in B^*$ converge to the corresponding finite-dimensional distributions of $G$. By Lemma \ref{lem:aldous}, the sequence $(G_n)_{n\in\N}$ satisfies Aldous' condition in $D([0,1],B)$. The proof of Proposition \ref{prop:norm_aldous}, in particular \eqref{eq:escaping_mass}, shows that the infinite-dimensional tails $T_k G_n$ are uniformly negligible in the sense that  
\begin{equation*}
    \lim_{k\to\infty} \limsup_{n\to\infty} \ex[ \sup_{t\in [0, 1]}\| T_k(G_n(t)) \|_B] = 0.
\end{equation*}
This tail estimate is the step where Assumption \ref{assump:entropy} enters and ensures that the projected tightness is lifted to the $B$-valued partial sum process.  
By Theorem 23.11 of \cite{kallenberg2021}, Aldous' condition from Lemma \ref{lem:aldous} implies the condition in Equation (11) of \cite{kallenberg2021}.  
Theorem 23.9 of \cite{kallenberg2021} and Lemma \ref{lem:fdd}, together with Proposition \ref{prop:measure_equality} for identification of subsequential limits by linear projections, therefore yields $G_n\convw G$ in $D([0,1],B)$.  

\begin{lemma}\label{lem:fdd}
	Let $B$ be a separable Banach space and $\{(X_{i, n})_{1\le i\le n}\}_{n\in\N}$ a locally stationary time series, satisfying Assumptions \ref{assump:local_stationarity} and \ref{assump:sequences}. Then, with $G$ as in Theorem \ref{thm:fclt},
	\begin{equation}\label{eq:fdd_conv}
		\big(b_1(G_n(t_1)), \dots, b_p(G_n(t_p)) \big) \convw \big(b_1(G(t_1)), \dots, b_p(G(t_p)) \big)
	\end{equation}
	in $\R^p$, for any $b_1, \dots, b_p \in B^*, t_1, \dots, t_p \in [0, 1]$ and $p\in \N$.
\end{lemma}

\begin{lemma}\label{lem:aldous}
	Under the assumptions of Theorem \ref{thm:fclt}, $G_n$ satisfies Aldous' condition, i.\,e., for all $\eps > 0$ and all positive sequences $(h_n)_{n\in\N}$ with $h_n \to 0$,
    \begin{equation} \label{eq:aldous_norm}
         \lim_{n\to\infty} \sup_{\tau} \pr\big(\| G_n(\tau + h_n) - G_n(\tau) \|_B > \eps\big) = 0,
    \end{equation}
    where the supremum is taken over all stopping times $\tau$ on the grid $\{\tfrac{1}{n}, \tfrac{2}{n}, \dots, \tfrac{n}{n}\}$ with respect to the natural filtration $\{\sigma(\Fc_i)\}_{i=1}^n$ satisfying $\tau + h_n \le 1$ almost surely.
\end{lemma}

\subsection{Proof of Lemma \ref{lem:fdd}}

Fix $p \in \N$ and consider functionals $b_1, \dots, b_p \in B^*$ and time instants $t_1, \dots, t_p \in [0, 1]$. By the Cramér-Wold device, the convergence in \eqref{eq:fdd_conv} is equivalent to
\begin{equation} \label{eq:cramer_wold}
		\sum_{i=1}^{p} a_i b_i(G_n(t_i)) \convw \sum_{i=1}^{p} a_i b_i(G(t_i)), 
\end{equation}
for all $a_1, \dots, a_p \in \R$. Let $\eps_{j, n} := X_{j, n} - \ex[X_{j, n}]$, for $j=1, \dots, n$ and $n\in\N$, denote the centered errors corresponding to $X_{j, n}$. For the left-hand side, we have by linearity of the $b_i$,
\begin{align*}
	\sum_{i=1}^{p} a_i b_i(G_n(t_i)) =  \sum_{j=1}^n \frac{1}{\sqrt{n}} \sum_{i=1}^{p} a_i  b_i(\eps_{j, n}) \id(j \le \lfloor t_i n \rfloor).
\end{align*}
In the following, we use the big-blocks-small-blocks method to prove the weak convergence in \eqref{eq:cramer_wold}. Let $(\ell_n)_{n\in\N}$ denote a sequence of block lengths for the big blocks, and $(s_n)_{n\in\N}$ accordingly the lengths of the small blocks, as in Assumption \ref{assump:sequences}. Further, let $d_n = \lfloor \tfrac{n}{\ell_n + s_n}\rfloor$ denote the number of blocks.

Define $s_n$-dependent random variables 
\[
	\tilde{\eps}_{j, n} = \ex[\eps_{j, n} | \eta_j, \dots, \eta_{j - s_n}].
\]
By Proposition \ref{prop:sn_dependence}, $\sup_{1\le j \le n, n \in\N} \ex\big[\big( b_i(\eps_{j, n}) - b_i(\tilde{\eps}_{j, n}) \big)^2\big]^{1/2} \le \max_{i=1}^p \|b_i\|_{B^*} \Theta_{s_n}$, such that
\begin{align*}
	& \ex\bigg[\bigg( \sum_{j=1}^n \frac{1}{\sqrt{n}} \sum_{i=1}^{p} a_i  [b_i(\eps_{j, n})  - b_i(\tilde{\eps}_{j, n})]\id(j \le \lfloor t_i n \rfloor) \bigg)^2\bigg]^{1/2} \\
	& \le \sum_{j=1}^n \frac{1}{\sqrt{n}} \sum_{i=1}^{p} a_i \ex\big[\big( b_i(\eps_{j, n})  - b_i(\tilde{\eps}_{j, n}) \big)^2\big]^{1/2} \id(j \le \lfloor t_i n \rfloor) = \Oc(\sqrt{n}\Theta_{s_n}).
\end{align*}
Hence, we can replace the random variables $\eps_{i, n}$ by their $s_n$-dependent counterpart and obtain
\[ \sum_{i=1}^{p} a_i b_i(G_n(t_i)) =  \sum_{j=1}^n \frac{1}{\sqrt{n}} \sum_{i=1}^{p} a_i  b_i(\tilde{\eps}_{j, n}) \id(j \le \lfloor t_i n \rfloor) + \Oc_\pr(\sqrt{n}\Theta_{s_n}). \]
Let 
\begin{equation} \label{eq:def_y}
    Y_j = \frac{1}{\sqrt{n}} \sum_{i=1}^{p} a_i  b_i(\tilde{\eps}_{j, n}) \id(j \le \lfloor t_i n \rfloor),
\end{equation}
and define the blocks
\begin{align*}
	B_k & = \{ i \in \N: (k-1)(\ell_n + s_n) + 1 \le i \le (k-1)(\ell_n + s_n) + \ell_n  \} \\
	S_k & = \{ i \in \N: (k-1)(\ell_n + s_n) + \ell_n + 1 \le i \le k (\ell_n + s_n)  \} \\
\end{align*}%
for $k=1, \dots, d_n$, and the remainder
\[
	R = \{ i \in \N: d_n (\ell_n + s_n)  < i \le n \}.
\]
Then, 
\[ \sum_{i=1}^{p} a_i b_i(G_n(t_i)) =  \sum_{j=1}^n Y_j + \Oc_\pr(\sqrt{n}\Theta_{s_n}). \]
In the following, we show that the sums over the small blocks and the remainder are asymptotically negligible, whereas the big blocks are independent by definition of the random variables $\tilde{\eps}_{j, n}$. Since $s_n = o(\ell_n)$, the small blocks have distances larger than $s_n$, for almost all $n\in\N$, and are therefore independent. Hence, 
\begin{equation}\label{eq:var_small_blocks}
	\ex\bigg[ \bigg( \sum_{k=1}^{d_n} \sum_{j \in S_k} Y_j \bigg)^2 \bigg]
	= \sum_{k=1}^{d_n} \sum_{j_1, j_2 \in S_k} \ex[ Y_{j_1}Y_{j_2}].
\end{equation}
Trivially, the $s_n$-dependent random variables $\tilde{\eps}_{j,n}$ satisfy Assumption \ref{assump:local_stationarity}. By the same arguments as in the proof of Proposition \ref{prop:cov_bound}, one may derive
\begin{equation*}
    \big| \ex[b_{i_1}(\tilde{\eps}_{j_1,n}) b_{i_2}(\tilde{\eps}_{j_2,n})]\big| \le C_{i_1, i_2} \Theta_{|j_1 - j_2|},
\end{equation*}
for a suitable constant $C_{i_1, i_2} > 0$, depending only on $b_{i_1}$ and $b_{i_2}$. Taking the maximum over all such constants, yields a constant $C_b > 0$, depending only on $b_1, \dots b_p$, for which the previous display holds true. In particular,
\begin{equation*}
    |\ex[Y_{j_1}Y_{j_2}]|
    % & =  \bigg|\frac{1}{n} \sum_{i_1, i_2=1}^{p} a_{i_1} a_{i_2}  \ex[b_{i_1}(\tilde{\eps}_{j_1, n}) b_{i_2}(\tilde{\eps}_{j_2, n})] \id(j_1 \le \lfloor t_{i_1} n \rfloor, j_2 \le \lfloor t_{i_2} n \rfloor)\bigg| \\
    \le \frac{1}{n} \sum_{i_1, i_2=1}^{p} a_{i_1} a_{i_2} \big|  \ex[b_{i_1}(\tilde{\eps}_{j_1, n}) b_{i_2}(\tilde{\eps}_{j_2, n})] \big|
    \le \frac{C}{n}  \Theta_{|j_1 - j_2|}, 
\end{equation*}
and further, by the summability condition in Assumption \ref{assump:local_stationarity},
 \begin{equation} \label{eq:var_small_blocks_2}
    \ex\bigg[ \bigg( \sum_{k=1}^{d_n} \sum_{j \in S_k} Y_j \bigg)^2 \bigg] 
    \le \frac{C}{n}  \sum_{k=1}^{d_n} \sum_{j_1, j_2 \in S_k} \Theta_{|j_1 - j_2|}
    \le C \frac{d_n s_n}{n}
    = \Oc\Big(\frac{s_n}{\ell_n}\Big).    
 \end{equation}
Hence, the small blocks, and by the same arguments the remainder $R$, are negligible. 

By independence of the big blocks, we have analogously to \eqref{eq:var_small_blocks},
\begin{align}\label{eq:var_big_blocks}
	R_n & := \ex\bigg[ \bigg( \sum_{k=1}^{d_n} \sum_{j \in B_k} Y_j \bigg)^2 \bigg] \\
	% & = \sum_{k=1}^{d_n} \sum_{j_1, j_2 \in B_k} \ex[ Y_{j_1}Y_{j_2}] \\
	& = \sum_{i_1, i_2=1}^{p} a_{i_1} a_{i_2} \frac{1}{n} \sum_{k=1}^{d_n} \sum_{j_1, j_2 \in B_k}  \ex[b_{i_1}(\tilde{\eps}_{j_1, n}) b_{i_2}(\tilde{\eps}_{j_2, n})] \id(j_1 \le \lfloor t_{i_1} n \rfloor, j_2 \le \lfloor t_{i_2} n \rfloor). \notag
\end{align}
Recall that $C_1$ uniformly bounds the operator norms of $b_1, \dots, b_p$. By the triangle inequality, Jensen's inequality, the Cauchy-Schwarz inequality and Proposition \ref{prop:sn_dependence}, it holds
\begin{align*}
    & \max_{\substack{i_1, i_2 = 1, \dots, p \\ j_1, j_2 = 1, \dots, n}} \big| \ex[ b_{i_1}(\tilde{\eps}_{j_1, n}) b_{i_2}(\tilde{\eps}_{j_2, n})] - \ex[ b_{i_1}(\eps_{j_1, n}) b_{i_2}(\eps_{j_2, n})] \big| \\
    & \le \max_{\substack{i_1, i_2 = 1, \dots, p \\ j_1, j_2 = 1, \dots, n}} \ex[ |b_{i_1}(\tilde{\eps}_{j_1, n})| |b_{i_2}(\tilde{\eps}_{j_2, n}) - b_{i_2}(\eps_{j_2, n})|] + \ex[ |b_{i_1}(\tilde{\eps}_{j_1, n}) - b_{i_1}(\eps_{j_1, n})| |b_{i_2}(\eps_{j_2, n})|] \\
    & \le C_1^2 \Theta_{s_n} \Big[ \max_{j = 1, \dots, n} \big( \ex[\|\eps_{j, n}\|_B^2] \big)^{1/2} + \max_{j = 1, \dots, n} \big( \ex[\|\tilde{\eps}_{j, n}\|_B^2] \big)^{1/2}\Big].
\end{align*}
By the moment bound in Assumption \ref{assump:local_stationarity}, the maxima on the right-hand side can be bounded by some constant, so that the entire expression is of order $\Oc(\Theta_{s_n})$. By the same arguments, though using that $H$ is Lipschitz continuity with respect to the $L^2$-norm instead of Proposition \ref{prop:sn_dependence}, we have
\begin{equation*}
    \max_{\substack{i_1, i_2 = 1, \dots, p \\ j_1, j_2 = 1, \dots, n}} \big| \ex[ b_{i_1}(\eps_{j_1, n}) b_{i_2}(\eps_{j_2, n})] -  \ex\big[ b_{i_1}\big(H(\tfrac{j_1}{n}, \Fc_{j_1})\big) b_{i_2}\big(H(\tfrac{j_2}{n}, \Fc_{j_2})\big)\big] \big| = \Oc(\tfrac{\ell_n}{n}).
\end{equation*}
Again, by the same arguments, we can replace $\tfrac{j_1}{n}$ and $\tfrac{j_2}{n}$ by $\tfrac{k}{d_n}$, for $j_1, j_2 \in B_k$. More specifically, we have by Lipschitz continuity of $H$, 
\begin{equation*}
    \max_{\substack{i_1, i_2 = 1, \dots, p \\ j_1, j_2 \in B_k}} \big|  \ex\big[ b_{i_1}\big(H(\tfrac{j_1}{n}, \Fc_{j_1})\big) b_{i_2}\big(H(\tfrac{j_2}{n}, \Fc_{j_2})\big)\big] - \ex\big[ b_{i_1}\big(H(\tfrac{k}{d_n}, \Fc_{j_1})\big) b_{i_2}\big(H(\tfrac{k}{d_n}, \Fc_{j_2})\big)\big] \big| = \Oc(\tfrac{\ell_n}{n}).
\end{equation*}
Combining these approximations, and using the stationarity of $H(t, \Fc_i)$, for fixed $t\in[0, 1]$, we have
\begin{equation*}
    \ex[ b_{i_1}(\tilde{\eps}_{j_1, n}) b_{i_2}(\tilde{\eps}_{j_2, n})] = \ex\big[ b_{i_1}\big(H(\tfrac{k}{d_n}, \Fc_{0})\big) b_{i_2}\big(H(\tfrac{k}{d_n}, \Fc_{j_2-j_1})\big)\big] + \Oc(\Theta_{s_n} + \tfrac{\ell_n}{n}),
\end{equation*}
uniformly for $i_1, i_2 = 1, \dots, p, j_1, j_2 \in B_k$. Plugging this into \eqref{eq:var_big_blocks}, yields
\begin{align} \label{eq:var_big_blocks_2}
    R_n = \sum_{i_1, i_2=1}^{p} a_{i_1} a_{i_2} \frac{1}{n} \sum_{k=1}^{d_n} \sum_{j_1, j_2 \in B_k} & \ex\big[ b_{i_1}\big(H(\tfrac{k}{d_n}, \Fc_{0})\big) b_{i_2}\big(H(\tfrac{k}{d_n}, \Fc_{j_2-j_1})\big)\big] \\
    & \times \id(j_1 \le \lfloor t_{i_1} n \rfloor, j_2 \le \lfloor t_{i_2} n \rfloor) + \Oc(\ell_n \Theta_{s_n} + \tfrac{\ell_n^2}{n}). \notag
\end{align}
For $j\in B_k$, it follows $j \le k(\ell_n + s_n) \le \lfloor t d_n \rfloor (\ell_n + s_n) \le nt$, and hence $j \le \lfloor nt\rfloor$, whenever $k \le \lfloor t d_n \rfloor$. Conversely, if $k - 2 \ge \lfloor t d_n \rfloor$, it holds 
\begin{equation}
    j \ge (k-1) (\ell_n + s_n) + 1 > (\lfloor t d_n \rfloor + 1) (\ell_n + s_n) + 1 > tn.
\end{equation}
Hence, for $k \in \{1, \dots, d_n\} \setminus\{ \lfloor td_n \rfloor + 1 \}$, it holds that 
\begin{equation*}
     \id(j_1 \le \lfloor t_{i_1} n \rfloor, j_2 \le \lfloor t_{i_2} n \rfloor) = \id(k \le \lfloor (t_{i_1} \wedge t_{i_2}) d_n \rfloor), 
\end{equation*}
for $j_1, j_2 \in B_k$. By the same arguments leading to \eqref{eq:var_small_blocks_2}, the term corresponding to $k = \lfloor td_n \rfloor + 1$ is negligible. More specifically, 
\begin{align*}
    \frac{1}{n} \sum_{j_1, j_2 \in B_k} \ex\big[ b_{i_1}\big(H(\tfrac{k}{d_n}, \Fc_{0})\big) b_{i_2}\big(H(\tfrac{k}{d_n}, \Fc_{j_2-j_1})\big)\big] 
    \id(j_1 \le \lfloor t_{i_1} n \rfloor, j_2 \le \lfloor t_{i_2} n \rfloor)
    = \Oc\big(\tfrac{\ell_n^2}{n}\big),
\end{align*} 
by the moment bound in Assumption \ref{assump:local_stationarity} and boundedness of $b_{i_1}$ and $b_{i_2}$. Hence, we can replace the indicator in \eqref{eq:var_big_blocks_2} at the cost of an error of order $\Oc(\ell_n^2/n)$, which yields
\begin{align*}
    R_n & = \sum_{i_1, i_2=1}^{p} a_{i_1} a_{i_2} \frac{1}{n} \sum_{k=1}^{\lfloor (t_{i_1} \wedge t_{i_2}) d_n \rfloor} \sum_{j_1, j_2 \in B_k} \ex\big[ b_{i_1}\big(H(\tfrac{k}{d_n}, \Fc_{0})\big) b_{i_2}\big(H(\tfrac{k}{d_n}, \Fc_{j_2-j_1})\big)\big] + \Oc(\ell_n \Theta_{s_n} + \tfrac{\ell_n^2}{n})  \\
    & = \sum_{i_1, i_2=1}^{p} a_{i_1} a_{i_2} \frac{\ell_n}{n} \sum_{k=1}^{\lfloor (t_{i_1} \wedge t_{i_2}) d_n \rfloor}   (b_{i_1} \otimes b_{i_2}) \sigma^2(\tfrac{k}{d_n}) + \Oc(\ell_n \Theta_{s_n} + \tfrac{\ell_n^2}{n}) + o(1),
\end{align*}
where the last equality follows by Proposition \ref{prop:lrv_approximation}. By  continuity of $\sigma^2(t)$, $(b_1 \otimes b_2) \sigma^2(t)$ is continuous for $b_1, b_2 \in B^*$. Since the Riemann sum of a continuous function over a bounded interval converges to the corresponding integral, it follows that 
\begin{equation} \label{eq:var_big_blocks_3}
    R_n = \sum_{i_1, i_2=1}^{p} a_{i_1} a_{i_2} \int_0^{t_{i_1} \wedge t_{i_2}} (b_{i_1} \otimes b_{i_2}) \sigma^2(x) \diff x + \Oc(\ell_n \Theta_{s_n} + \tfrac{\ell_n^2}{n}) + o(1).
\end{equation}
By Assumption \ref{assump:sequences}, $\Oc(\ell_n \Theta_{s_n} + \ell_n^2/n)=o(1)$. By Hille's theorem, the Bochner integral commutes with bounded linear operators \citep[see Theorem 1.2.4 in][]{hytonen2016}, such that 
\begin{equation*}
    \int_0^{t_{i_1} \wedge t_{i_2}} (b_{i_1} \otimes b_{i_2}) \sigma^2(x) \diff x 
    = (b_{i_1} \otimes b_{i_2}) \int_0^{t_{i_1} \wedge t_{i_2}}  \sigma^2(x) \diff x 
    = (b_{i_1} \otimes b_{i_2}) \cov\big(G(t_{i_1}), G(t_{i_2})\big). 
\end{equation*}
In particular, it follows from \eqref{eq:var_big_blocks_3} that $R_n$ converges to  $\var\big(\sum_{i=1}^{p} a_i b_i G(t_i)\big)$, as $n\to\infty$.

For the random variables $Y_j$, as defined in \eqref{eq:def_y}, 
\begin{align*} 
	\sup_{1\le j \le n, n \in\N} \ex[Y_j^{2+\kappa}]^{1/(2+\kappa)}
	& \le \sup_{1\le j \le n, n \in\N} \frac{1}{\sqrt{n}} \sum_{i=1}^p |a_i| \ex\big[\big(b_i(\tilde{\eps}_{j, n})\big)^{2+\kappa}\big]^{1/(2+\kappa)} \id(j \le \lfloor t_i n \rfloor)\\
	& \le \frac{C_1}{\sqrt{n}} \max_{i=1}^p \|b_i\|_{B^*} \sup_{1\le j \le n, n \in\N}  \ex[\|\tilde{\eps}_{j, n}\|_B^{2+\kappa}]^{1/(2+\kappa)}
	\le \frac{C_2}{\sqrt{n}}, \notag
\end{align*}
for some constants $C_1, C_2 \ge 0$, by Assumption \ref{assump:local_stationarity}. 
Therefore, 
\begin{align*}
    \sum_{k=1}^{d_n} \ex\bigg[\bigg(\sum_{j \in B_k} Y_j\bigg)^{2+\kappa}\bigg] 
    \le \sum_{k=1}^{d_n} \bigg(\sum_{j \in B_k} \ex[|Y_j|^{2+\kappa}]^{\frac{1}{2+\kappa}}\bigg)^{2+\kappa} 
    \le C d_n \Big(\frac{\ell_n}{\sqrt{n}}\Big)^{2+\kappa} 
    % = C \frac{\ell_n^{1+\kappa}}{n^{\kappa/2}}
    = C \Big(\frac{\ell_n^{2/\kappa+2}}{n}\Big)^{\kappa/2},
\end{align*}
which converges to $0$ by Assumption \ref{assump:sequences}. By Lyapunov's central limit theorem, 
\begin{equation*}
    \sum_{i=1}^{p} a_i b_i(G_n(t_i)) \convw \Nc\bigg(0, \var\bigg(\sum_{i=1}^{p} a_i b_i G(t_i)\bigg) \bigg) \stackrel{\Dc}{=} \sum_{i=1}^{p} a_i b_i G(t_i),
\end{equation*}
such that \eqref{eq:cramer_wold} holds, and the lemma follows by the Cramér-Wold device.

\subsection{Proof of Lemma \ref{lem:aldous}}

The lemma is a direct consequence of the following proposition, where the condition in \eqref{eq:functional_aldous} is established by Proposition \ref{prop:functional_aldous}.

\begin{proposition} \label{prop:norm_aldous}
    Let the assumptions of Theorem \ref{thm:fclt} be satisfied. If
    \begin{equation} \label{eq:functional_aldous}
         \lim_{n\to\infty} \sup_{\tau} \pr(| b(G_n(\tau + h_n) - G_n(\tau)) | > \eps) = 0,
    \end{equation}
    for all $b\in B^*$, $\eps > 0$ and all positive sequences $(h_n)_{n\in\N}$ with $h_n \to 0$, where the supremum is taken over all stopping times $\tau$ on the grid $\{\tfrac{1}{n}, \tfrac{2}{n}, \dots, \tfrac{n}{n}\}$ with respect to the natural filtration $\{\sigma(\Fc_i)\}_{i=1}^n$ satisfying $\tau + h_n \le 1$ almost surely, then Aldous' condition in \eqref{eq:aldous_norm} holds true.
\end{proposition}

\begin{proof}
    Recall the set $\Gc_k$ from \eqref{eq:def_set_entropy} and the definition of the empirical process from \eqref{eq:def_emp_process}. Then, $\sup_{t\in [0, 1]}\| T_k(G_n(t)) \|_B = \sup_{f\in\Gc_k}|\Gb_n(f)|$. By Theorem \ref{thm:max_inequality},
    $\ex[\sup_{f\in\Gc}|\Gb_n(f)|] \le c[I_k + c_k]$, which does not depend on $n$ and vanishes as $k\to\infty$, by Assumption \ref{assump:entropy}. In particular,
    \begin{equation} \label{eq:escaping_mass}
        \lim_{k\to\infty} \limsup_{n\to\infty} \ex[ \sup_{t\in [0, 1]}\| T_k(G_n(t)) \|_B] = 0.
    \end{equation}
    By the triangle inequality, it holds
    \begin{equation} \label{eq:partition_of_identity}
        \pr(\| G_n(\tau + h_n) - G_n(\tau) \|_B > \eps) \le P_{n, k}^{(1)} + P_{n, k}^{(2)},
    \end{equation}
    for $k, n\in \N$, where
    \begin{align*}
        P_{n, k}^{(1)} & := \pr(\| \Pi_k[G_n(\tau + h_n) - G_n(\tau)] \|_B > \tfrac{\eps}{2}), \\
        P_{n, k}^{(2)} & := \pr(\| T_k[G_n(\tau + h_n) - G_n(\tau)] \|_B > \tfrac{\eps}{2}),
    \end{align*} 
    and $T_k = I_b - \Pi_k$. In the following, we bound $P_{n, k}^{(1)}$ and $P_{n, k}^{(2)}$ separately.  
    First, note that the dimension of $Im(\Pi_k)$ is $k$. Consider $\tilde{b}_1, \dots, \tilde{b}_k \in B^*$ that span the dual of $Im(\Pi_k)$, i.\,e., $span(\{\tilde{b}_i|_{Im(\Pi_k)}\}_{i=1, \dots, k}) = (Im(\Pi_k))^*$. Further define $b_i = \tilde{b}_i \circ \Pi_k \in B^*$. Then, 
    \begin{multline*}
        \| \Pi_k y \|_B 
        = \sup_{\|b\|_{B^*} \le 1} |b(\Pi_k y)| 
        = \sup_{\substack{\alpha_1, \dots, \alpha_k\\\|\sum_{i=1}^k \alpha_i b_i \|_{B^*}\le 1}} \bigg| \sum_{i=1}^k \alpha_i b_i(y)  \bigg| \\
        % & \le \sup_{\substack{\alpha_1, \dots, \alpha_k\\\|\sum_{i=1}^k \alpha_i b_i \|_{B^*}\le 1}} \sum_{i=1}^k |\alpha_i|  |b_i(y)| \\
        \le \sum_{i=1}^k \sup_{\alpha_i: \| \alpha_i b_i \|_{B^*}\le 1}  |\alpha_i|  |b_i(y)|
        \le \max_{i=1}^k \frac{1}{\| b_i \|_{B^*}} \sum_{j=1}^k |b_j(y)|.
    \end{multline*}
    Hence, with $C_k = \max_{i=1}^k \| b_i \|_{B^*}^{-1}$, it holds
    \begin{equation*}
        P_{n, k}^{(1)}
        \le \pr\bigg( C_k \sum_{i=1}^k \big|b_i\big(G_n(\tau + h_n) - G_n(\tau)\big)\big| > \tfrac{\eps}{2}\bigg) 
        \le \sum_{i=1}^k  \pr\Big( \big|b_i\big(G_n(\tau + h_n) - G_n(\tau)\big)\big| > \tfrac{\eps}{2C_k k}\Big) 
    \end{equation*}
    by the union bound. By \eqref{eq:functional_aldous}, the right-hand side vanishes, as $n \to\infty$, for fixed $k\in \N$, such that
    $\lim_{n\to\infty} P_{n, k}^{(1)} = 0$.

    Concerning $P_{n, k}^{(2)}$, note that      
    \begin{equation*}
        \ex[\| T_k(G_n(\tau)) \|_B]
        \le \ex[\sup_{t\in [0, 1]} \| T_k(G_n(t)) \|_B ].
    \end{equation*}
    Consequently, by Markov's inequality and the triangle inequality,
    \begin{equation*}
        P_{n, k}^{(2)}
        \le C \ex[\sup_{t\in [0, 1]} \| T_k(G_n(\tau)) \|_B],
    \end{equation*}
    for some constant $C\ge 0$. 
    By \eqref{eq:escaping_mass}, for any $\delta > 0$, some $k(\delta)\in \N$ exists, such that 
    $\lim_{n\to\infty} P_{n, k}^{(2)} < \delta/3$, for $k \ge k(\delta)$. In particular, some $n_1(k(\delta))\in\N$ exists with 
    $P_{n, k(\delta)}^{(2)} < 2\delta/3$,  for all $n \ge n_1(k(\delta))$. Finally, since $\lim_{n\to\infty} P_{n, k(\delta)}^{(1)} = 0$, some $n_2(k(\delta))$ exists, such that $P_{n, k}^{(1)} < \delta/3$,
    for $n \ge n_2(k(\delta))$. Setting $n_0(\delta) = \max\{n_1(k(\delta)), n_2(k(\delta)) \}$,  it holds
    \begin{equation*}
        \pr(\| G_n(\tau + h_n) - G_n(\tau) \|_B > \eps) < \delta
    \end{equation*}
    for any $n \ge n_0(\delta)$.
    
\end{proof}

\begin{proposition} \label{prop:functional_aldous}
    Let Assumptions \ref{assump:space}, \ref{assump:local_stationarity} and \ref{assump:sequences} be satisfied. Then, \eqref{eq:functional_aldous} holds true, i.\,e., Aldous' condition for functionals.
\end{proposition}

\begin{proof}
    Let $b\in B^*$. By Markov's inequality, it is sufficient to show that
    \begin{equation*}
        \lim_{n\to\infty} \sup_\tau \ex[ | b(G_n(\tau + h_n) - G_n(\tau)) | ] = 0,
    \end{equation*}
    where the supremum is taken over all stopping times $\tau$ on the grid $\{\tfrac{1}{n}, \tfrac{2}{n}, \dots, \tfrac{n}{n}\}$ with respect to the natural filtration $\{\sigma(\Fc_i)\}_{i=1}^n$ satisfying $\tau + h_n \le 1$ almost surely. Let $\tau$ be an arbitrary stopping time on the grid. The subsequent proof consists of two steps. First, we replace $b(\eps_{i, n})$ by $m$-dependent random variables, and second, we bound the sum corresponding to these random variables. 
    
    Let $m\in \N$. Then, by an index shift and repeated use of the triangle inequality, it holds
    \begin{align} \notag
        & \ex\bigg[ \bigg| \frac{1}{\sqrt{n}} \sum_{i=\tau n + 1}^{\tau n + \lfloor h_n n\rfloor} b(\eps_{i, n}) - \ex[b(\eps_{i, n})|\eta_i, \dots, \eta_{i-m}] \bigg|^{2+\kappa} \bigg]^{1/(2+\kappa)} \\
        & = \ex\bigg[ \bigg| \frac{1}{\sqrt{n}} \sum_{i=\tau n + 1}^{\tau n + \lfloor h_n n\rfloor} \sum_{\ell=-\infty}^{i-m} \ex[b(\eps_{i, n})|\eta_i, \dots, \eta_{\ell}] - \ex[b(\eps_{i, n})|\eta_i, \dots, \eta_{\ell-1}] \bigg|^{2+\kappa} \bigg]^{1/(2+\kappa)} \label{eq:functional_aldous1} \\
        % & \le \frac{1}{\sqrt{n}} \sum_{i=1}^{\lfloor h_n n\rfloor} \ex\bigg[ \bigg|  \sum_{\ell=-\infty}^{i-m} \ex[b(\eps_{i+\tau n, n})|\eta_{i+\tau n}, \dots, \eta_{\ell+\tau n}] - \ex[b(\eps_{i+\tau n, n})|\eta_{i + \tau n}, \dots, \eta_{\ell+\tau n-1}] \bigg|^{2+\kappa} \bigg]^{1/(2+\kappa)} \\
        & \le \frac{1}{\sqrt{n}} \sum_{i=1}^{\lfloor h_n n\rfloor} \sum_{\ell=-\infty}^{i-m} \ex\big[ \big|   \ex[b(\eps_{i+\tau n, n})|\eta_{i+\tau n}, \dots, \eta_{\ell+\tau n}] - \ex[b(\eps_{i+\tau n, n})|\eta_{i + \tau n}, \dots, \eta_{\ell+\tau n-1}] \big|^{2+\kappa} \big]^{1/(2+\kappa)}. \notag
    \end{align}
    Now, define $\eps_{i, n}^*$ by replacing $\eta_{\ell+\tau n}$ in the definition of $\eps_{i, n}$ with an independent copy $\eta_{\ell + \tau n}^*$. By Assumption \ref{assump:local_stationarity} and Jensen's inequality for the conditional expectation, it then holds
    \begin{align*}
        & \ex\big[ \big|   \ex[b(\eps_{i+\tau n, n})|\eta_{i+\tau n}, \dots, \eta_{\ell+\tau n}] - \ex[b(\eps_{i+\tau n, n})|\eta_{i + \tau n}, \dots, \eta_{\ell+\tau n-1}] \big|^{2+\kappa} \big]^{1/(2+\kappa)} \\
        & = \ex\big[ \big|   \ex[b(\eps_{i+\tau n, n} - \eps_{i+\tau n, n}^*)|\eta_{i + \tau n}, \dots, \eta_{\ell+\tau n}] \big|^{2+\kappa} \big]^{1/(2+\kappa)} \\
        & \le \|b\|_{B^*} \delta_{2+\kappa}(H, i - \ell),
    \end{align*}
    for $\ell < i-m$. Plugging this bound into the right-hand side of \eqref{eq:functional_aldous1}, yields
    \begin{align} \label{eq:functional_aldous2}
        & \ex\bigg[ \bigg| \frac{1}{\sqrt{n}} \sum_{i=\tau n + 1}^{\tau n + \lfloor h_n n\rfloor} b(\eps_{i, n}) - \ex[b(\eps_{i, n})|\eta_i, \dots, \eta_{i-m}] \bigg|^{2+\kappa} \bigg]^{1/(2+\kappa)}  \\
        & \le \frac{C}{\sqrt{n}} \sum_{i=1}^{\lfloor h_n n\rfloor} \sum_{\ell=-\infty}^{i-m} \delta_{2+\kappa}(H, i - \ell) \le C \sqrt{n} h_n\Theta_{2+\kappa, m}. \notag
    \end{align}
    Hence, we may replace $b(\eps_{i, n})$ by $\ex[b(\eps_{i, n})|\eta_i, \dots, \eta_{i-m}]$ at a cost of order $\sqrt{n} h_n\Theta_{2+\kappa, m}$.

    Now, using the triangle inequality, we split the remaining expression into terms depending on $\Fc_{\tau n}$ and those independent of it. More specifically, we have
    \begin{align*}
        & \ex\bigg[ \bigg| \frac{1}{\sqrt{n}} \sum_{i=\tau n + 1}^{\tau n + \lfloor h_n n\rfloor} \ex[b(\eps_{i, n})|\eta_i, \dots, \eta_{i-m}] \bigg|^{2+\kappa} \bigg]^{1/(2+\kappa)} \\
        & \le \frac{1}{\sqrt{n}}  \ex\bigg[ \bigg| \sum_{i=\tau n + 1}^{\tau n + m} \ex[b(\eps_{i, n})|\eta_i, \dots, \eta_{i-m}] \bigg|^{2+\kappa} \bigg]^{1/(2+\kappa)}  \\
        &\phantom{=}+ \frac{1}{\sqrt{n}}  \ex\bigg[ \bigg| \sum_{i=\tau n + m +1}^{\tau n + \lfloor h_n n\rfloor} \ex[b(\eps_{i, n})|\eta_i, \dots, \eta_{i-m}] \bigg|^{2+\kappa} \bigg]^{1/(2+\kappa)}
    \end{align*}
    By the same arguments as before, i.\,e., an index shift $i \to j+ \tau n$ and the triangle inequality, we can bound the first expression on the right-hand side by $Cm n^{-1/2}$, for some $C \ge 0$. Recall that $\tau$ is a stopping time with respect to $\{\sigma(\Fc_i)\}_{i=1}^n$, such that $\tau n = \sum_{\nu=1}^n \nu \id(\tau n = \nu)$. By the law of iterated expectation, we can write the second term as 
    \begin{align} \label{eq:functional_aldous3}\notag 
        & \ex\bigg[ \bigg| \sum_{i=\tau n + m +1}^{\tau n + \lfloor h_n n\rfloor} \ex[b(\eps_{i, n})|\eta_i, \dots, \eta_{i-m}] \bigg|^{2+\kappa} \bigg] \\
        & = \sum_{\nu=1}^n \ex\bigg[ \ex\bigg[ \Big| \sum_{i=\nu + m +1}^{\nu + \lfloor h_n n\rfloor} \ex[b(\eps_{i, n})|\eta_i, \dots, \eta_{i-m}] \Big|^{2+\kappa} \id(\tau n = \nu) \bigg| \Fc_\nu \Big]\bigg] \\
        & = \sum_{\nu=1}^n \pr(\tau n = \nu)\ex\bigg[ \bigg| \sum_{i=\nu + m +1}^{\nu + \lfloor h_n n\rfloor} \ex[b(\eps_{i, n})|\eta_i, \dots, \eta_{i-m}] \bigg|^{2+\kappa} \bigg], \notag
    \end{align}
    where we used in the last step that $\{\tau n = \nu\}$ is $\sigma(\Fc_\nu)$-measurable and the sum is independent of $\Fc_\nu$. Note that $\ex[b(\eps_{i, n})|\eta_i, \dots, \eta_{i-m}]$ satisfies Assumption \ref{assump:local_stationarity}. By Theorem 1 of \cite{liu2013}, obtain the bound
    \begin{align*}
        & \ex\bigg[ \bigg| \sum_{i=\nu + m +1}^{\nu + \lfloor h_n n\rfloor} \ex[b(\eps_{i, n})|\eta_i, \dots, \eta_{i-m}] \bigg|^{2+\kappa} \bigg]^{1/(2+\kappa)}\\
        & \le C_1 \sqrt{\lfloor h_n n \rfloor} (\Theta_{1, 2} + \ex\big[ \ex[b(\eps_{i, n})|\eta_i, \dots, \eta_{i-m}]^2 \big]^{1/2} )\\ 
        & \phantom{=}+ C_1 \lfloor h_n n \rfloor^{1/(2+\kappa)} \bigg(  \sum_{j=1}^\infty \min(j, n)^{\tfrac{\kappa}{2(2+\kappa)}} \delta_{2+\kappa}(j) + \ex\big[\ex[b(\eps_{i, n})|\eta_i, \dots, \eta_{i-m}]^{2+\kappa} \big]^{1/(2+\kappa)} \bigg) \\
        & \le C_2 h_n^{1/(2+\kappa)} n^{1/2},
    \end{align*}
    for constants $C_1, C_2 \ge 0$. Plugging this bound into \eqref{eq:functional_aldous3} gives
    \begin{equation}
        \ex\bigg[ \bigg| \frac{1}{\sqrt{n}} \sum_{i=\tau n + 1}^{\tau n + \lfloor h_n n\rfloor} \ex[b(\eps_{i, n})|\eta_i, \dots, \eta_{i-m}] \bigg|^{2+\kappa} \bigg]^{1/(2+\kappa)} \le C \big( \tfrac{m}{\sqrt{n}} + h_n^{1/(2+\kappa)}\big).
    \end{equation}
    Hence, it follows from \eqref{eq:functional_aldous2} that
    \begin{equation}
        \ex\bigg[ \bigg| \frac{1}{\sqrt{n}} \sum_{i=\tau n + 1}^{\tau n + \lfloor h_n n\rfloor} b(\eps_{i, n})\bigg|^{2+\kappa} \bigg]^{1/(2+\kappa)} \le C \big( \sqrt{n}h_n\Theta_{m, 2+\kappa} +\tfrac{m}{\sqrt{n}} + h_n^{1/(2+\kappa)}\big).
    \end{equation}
    With $m = s_n$, as in Assumption \ref{assump:sequences}, the right-hand side converges to $0$, as $n\to\infty$, uniformly in $\tau$. 
\end{proof}

\subsection{Proof of Corollary \ref{cor:representation_limit}}

Before we prove the corollary, we introduce a helpful inequality to bound the operator norm in terms of the projective tensor norm.

\begin{proposition} \label{prop:norm_bound}
    Let $R\in B\otimes B$ be a finite tensor. Let $\Phi$ denote the canonical identification $\Phi: B\hat{\otimes} B \to \Lc(B^*, B)$. Then, $\| \Phi(R) \|_{op} \le \| R\|_\pi$.
\end{proposition}

\begin{proof}
    Let $\sum_{j=1}^k u_j \otimes v_j$ be a representation of $R$. By definition, $\Phi(R)(b) = \sum_{j=1}^k b(v_j) u_j$, for any $b\in B^*$. Note that $\Phi(R)(b)$ is well-defined as it does not depend on the representation. Moreover,
    \[ \| \Phi(R)(b)\|_B \le \sum_{j=1}^k |b(v_j)| \| u_j\|_B \le \|b\|_{B^*} \sum_{j=1}^k \|v_j\|_B \| u_j\|_B.   \]
    Taking the infimum over all representations of $R$, yields $\| \Phi(R) \|_{op} \le \| R\|_\pi$.
\end{proof}

\noindent \textbf{Proof of Corollary \ref{cor:representation_limit}}: 

\noindent The corollary's proof essentially consists of six steps. First, we construct pointwise operators $S_t:H_t\to B$, for time-dependent Hilbert spaces $H_t$. Next, we show that the Hilbert spaces are uniformly separable. Then, we use a Gram-Schmidt process for the construction of a measurable, orthonormal basis of $H_t$. This basis will be used to construct operators $U_t:\ell^2 \to H_t$, such that $\sigma(t) = S_t U_t$ satisfies the corollary's statement. The final steps are proving measurability of $t\mapsto \sigma(t)$ and the stochastic integral representation.

\textbf{Pointwise operators:} Define $Q_t = \Phi(\sigma^2(t))\in  \Lc(B^*, B)$ and $q_t: B^* \times B^* \to \R$ in terms of $q_t(b, c) = b(Q_t c) = (b \otimes c) \sigma^2(t)$. Then, $q_t$ is symmetric and positive semidefinite. Symmetry follows from stationarity of $\{H(t, \Fc_i)\}_{i\in\Z}$ by
\begin{align*}
    q_t(b, c) & = (b \otimes c) \sigma^2(t)
    = \sum_{i\in\Z} \ex\big[b\big(H(t, \Fc_i)\big) c\big(H(t, \Fc_0) \big)\big] \\
    & = \sum_{i\in\Z} \ex\big[c\big(H(t, \Fc_{-i})\big) b\big(H(t, \Fc_0) \big)\big]
    = (c \otimes b) \sigma^2(t) = q_t(c, b)
\end{align*}
where we used that the expectation commutes with bounded linear operators \citep[see, e.\,g., Remark 6.21 in][]{janson2015}. Positive semidefiniteness follows by usual arguments for real-valued stationary time series. More specifically, 
\begin{align*}
    q_t(b, b) = \sum_{i\in\Z} \ex\big[b\big(H(t, \Fc_i)\big) b\big(H(t, \Fc_0) \big)\big]
    = \lim_{N\to\infty} \frac{1}{2N+1} \var\bigg(\sum_{i=-N}^{N} b\big(H(t, \Fc_i)\big)\bigg),
\end{align*}
where the sequence on the right-hand side is non-negative for any $N\in\N$, so that $q_x(b, b) \ge 0$ for any $b\in B^*$.

For each $t\in[0, 1]$, define $N_t = \{b\in B^*: q_t(b, b) = 0\}$ and let $H_t$ be the completion of $B^* / N_t$ with respect to the inner product $\langle [b]_t, [c]_t\rangle_{H_t} = q_t(b, c)$. Let $A_t : B^* \to H_t$ denote the canonical projection $A_tb = [b]_t$. On $A_t(B^*) \subset H_t$, define the map $S_t^0(A_tb) = Q_tb$. Note that, if $A_t {b} = 0$, then $q_t(b, b) = 0$, and by the Cauchy-Schwarz inequality, $q_t(c, b) = 0$ for all $c \in B^*$. Thus, $c(Q_t b) = q_t(c, b) = 0$, for all $c\in B^*$, such that $Q_t b = 0$. This shows, that $S_t^0$ is well-defined. 
Moreover, for $b\in B^*$,
\begin{equation*}
	\| S_t^0(A_t b) \|_B 
	= \| Q_t b\|_B
	= \sup_{\|c\|_{B^*} \le 1} |c(Q_t b)| 
	= \sup_{\|c\|_{B^*} \le 1} |q_t(c, b)| 
	\le \sup_{\|c\|_{B^*} \le 1} q_t^{1/2}(c, c) q_t^{1/2}(b, b).
\end{equation*}
By definition of the dual norm on $B^*$ and the operator norm, 
\begin{equation*}
	q_t(c, c) = c (Q_t c) \le \| c\|_{B^*} \| Q_t c\|_B \le \| c\|_{B^*}^2 \|Q_t\|_{op},
\end{equation*}
which implies that
\begin{equation*}
	\| S_t^0(A_t b) \|_B  \le  \|Q_t\|_{op}^{1/2} q_t^{1/2}(b, b) = \|Q_t\|_{op}^{1/2} \| A_t b\|_{H_t}.
\end{equation*}
Hence $S_t^0$ extends uniquely to a bounded linear operator $S_t:H_t \to B$ with $\| S_t\|_{op} \le \|Q_t\|_{op}^{1/2}$. For $b, c\in B^*$,
\begin{equation*}
	c(S_t A_t b) = c(Q_t b) = q_t(c, b) = q_t(b, c) = \langle A_t b, A_t c\rangle_{H_t},
\end{equation*}
such that $S_t^* c = A_t c$. Consequently, $S_t S_t^* c = S_t A_t c = Q_t c$ for all $c\in B^*$, hence, $S_t S_t^* = Q_t$.

\textbf{Uniform separability of $H_t$:} By Lemma \ref{lem:lrv_lipschitz}, $t\mapsto \sigma^2(t)$ is continuous. The canonical identification $\Phi: B \hat{\otimes} B \to \Lc(B^*, B)$ is continuous with $\| \Phi(x)\|_{op} \le \|x\|_\pi$ for $x \in B \hat{\otimes} B$, by Proposition \ref{prop:norm_bound}. Therefore, 
\begin{equation*}
	\| Q_t - Q_s \|_{op} = \| \Phi(\sigma^2(t) - \sigma^2(s)) \|_{op} \le \|\sigma^2(t) - \sigma^2(s)\|_\pi,
\end{equation*}
so $t\mapsto Q_t$ is continuous with respect to the operator norm. $\Phi$ is continuous, maps finite tensors to finite-rank operators and $\sigma^2(t) \in B\hat{\otimes} B$ is the projective-norm limit of finite tensors. Since finite-rank operators are compact and the operator-norm limit of compact operators is compact, each $Q_t$ is compact.

Define $K_0 = \{Q_t b: t\in[0, 1], \|b\|_{B^*}\le 1\}$, and $K = \overline{K_0}$ its closure in $B$. The set $K_0$ is relatively compact. Indeed, fix $\eps > 0$. Since $[0, 1]$ is compact, by continuity, $\{Q_t: t\in[0, 1]\}$ is compact with respect to the operator norm. Hence, there exist $t_1, \dots, t_N$ such that for every $t$, it exists $t_r$ with $\|Q_t - Q_{t_r}\|_{op} < \eps /2$. For each $t_r$, compactness of $Q_{t_r}$ implies that $Q_{t_r}(B_1^*)$ has a finite $\eps/2$-net in $B$. Combining these finitely many nets gives an $\eps$-net for $K_0$. Thus, $K_0$ is totally bounded, and since $B$ is complete, its closure $K$ is compact.

For each $m\in\N$, consider the set of restrictions $\Rc_m = \{b|_K: b\in B^*, \|b\|_{B^*}\le m\} \subset C(K)$. Since $K$ is a compact metric space, $C(K)$ is separable. Choose a countable set $D_m \subset m B_1^*$ whose restrictions to $K$ are dense in $\Rc_m$. Set
\begin{equation*}
	D = \bigcup_{m=1}^\infty D_m,
\end{equation*}
and enumerate $D = \{d_1, d_2, \dots\}$. We will show that $A_t(D)$ is dense in $H_t$ for every $t$. 

Fix $b\in B^*$. There exists $m\in\N$ such that $\|b\|_{B^*}\le m$. By construction of $D_m$, there exists a sequence $(b_n)_{n\in\N} \subset D_m$, such that $\sup_{x\in K} |(b-b_n)(x)| \to 0$. Define the remainder $r_n = b - b_n$. Then, $\|r_n\|_{B^*} \le 2m$, and further, $\|A_t r_n\|_{H_t}^2 = q_t(r_n, r_n) = r_n(Q_t r_n)$. If $r_n \neq 0$, $Q_t(r_n / \|r_n\|_{B^*}) \in K$, hence
\begin{equation*}
	0 \le r_n(Q_t r_n) = \|r_n\|_{B^*} r_n\Big(Q_t \frac{r_n}{\|r_n\|_{B^*}}\Big) \le 2m \sup_{x\in K}|r_n(x)|.
\end{equation*}
The right-hand side tends to $0$, hence $A_t b_n \to A_t b$ in $H_t$. Since $A_t(B^*)$ is dense in $H_t$ by definition, $A_t(D)$ is dense in $H_t$. Importantly, $D$ is countable and independent of $t$.

\textbf{Measurable Gram-Schmidt:} Define $v_n(t) = A_t d_n \in H_t$. For $i, j\in \N$, 
\begin{equation*}
	\langle v_i(t), v_j(t) \rangle_{H_t} = q_t(d_i, d_j) = d_i(Q_t d_j),
\end{equation*}
hence, each scalar function $t\mapsto \langle v_i(t), v_j(t) \rangle_{H_t}$ is continuous, and in particular Borel measurable, by operator-norm continuity of $Q_t$. We recursively apply Gram-Schmidt. 

Set $w_1(t) = v_1(t)$, and define $u_1(t) = w_1(t) / \|w_1(t)\|_{H_t}$, whenever the denominator is positive, and $u_1(t) = 0$ otherwise. Suppose, we have $u_1(t), \dots, u_{n-1}(t)$, and each $u_k(t)$ is of the form $u_k(t) = \sum_{j=1}^{k} a_{k,j}(t) v_j(t)$, with Borel measurable coefficients $a_{k, j}:[0, 1]\to \R$. Define 
\begin{equation*}
	w_n(t) = v_n(t) - \sum_{k=1}^{n-1} \langle v_n(t), u_k(t) \rangle_{H_t} u_k(t),
\end{equation*}
where the coefficients 
\begin{equation*}
	\langle v_n(t), u_k(t) \rangle_{H_t} = \sum_{j=1}^k a_{k,j}(t) \langle v_n(t), v_j(t) \rangle_{H_t} 
\end{equation*}
are Borel measurable as scalar functions of $t$. Therefore, $w_n(t)$ has a representation
\begin{equation*}
	w_n(t) = \sum_{j=1}^n \beta_{n,j}(t) v_j(t)
\end{equation*}
with Borel measurable coefficients $\beta_{n,j}$. Finally, the squared norm
\begin{equation*}
	\| w_n(t)\|_{H_t}^2 = \sum_{i, j=1}^{n} \beta_{n,i}(t) \beta_{n,j}(t) \langle v_i(t), v_j(t)\rangle_{H_t}
\end{equation*}
is Borel measurable. As for $n=1$, define $u_n(t) = w_n / \|w_n(t)\|_{H_t}$, whenever the denominator is positive, and $u_n(t) = 0$ otherwise. Since $A_t(D) = \{v_n(t): n\in\N\}$ is dense in $H_t$, for every $t\in[0, 1]$, the non-zero vectors among $(u_n(t))_{n\in\N}$ form an orthonormal basis of $H_t$. This is effectively the standard Gram-Schmidt process, with the additional property that all components are Borel measurable.

\textbf{A unique Hilbert space $H$:} Let $H = \ell^2$ with canonical orthonormal basis $(e_n)_{n\in\N}$, define $U_t: \ell^2 \to H_t$ by $U_t e_n = u_n(t)$, and extend linearly and continuously. Since the non-zero $u_n(t)$ are orthonormal, 
\begin{equation*}
	\| U_t h\|_{H_t}^2 = \sum_{n: u_n(t) \neq 0} |h_n|^2 \le \sum_{n=1}^{\infty} |h_n|^2 = \| h\|_{\ell^2}^2,
\end{equation*}
such that $U_t$ is a contraction. The adjoint $U_t^*: H_t \to \ell^2$ is given by $U_t^* x = (\langle x, u_n(t)\rangle_{H_t})_{n\in\N}$. By Parseval's identity,
\begin{equation*}
	U_t U_t^* x = \sum_{n=1}^{\infty} \langle x, u_n(t)\rangle_{H_t} u_n(t) = x,
\end{equation*}
for all $x\in H_t$, such that $U_t U_t^* = I_{H_t}$. In particular, $\|U_t\|_{op} = 1$, if $H_t\neq \{0\}$, and $\|U_t\|_{op} = 0$ otherwise.

Define $\sigma(t) = S_t U_t:\ell^2 \to B$ with $\|\sigma(t)\|_{op} \le \|S_t\|_{op} \|U_t\|_{op} \le \|Q_t\|^{1/2}_{op}$, where $\|\cdot\|_{op}$ denotes the operator norm of the respective spaces. By continuity of $Q_t$ and compactness of $[0, 1]$, $\Sigma := \sup_{t\in[0, 1]}\|\sigma(t)\|_{op} <\infty$. Moreover, for $b\in B^*$,
\begin{equation*}
	\sigma(t) \sigma^*(t) b= S_t U_t U_t^* S_t^* b = S_t S_t^* b = Q_t b.
\end{equation*}
Hence, $\sigma(t) \sigma^*(t) = Q_t = \Phi(\sigma^2(t))$.

\textbf{Measurability of $t\mapsto \sigma(t) h$:} For each $n$, using the finite coefficient representation of $u_n(t)$,
\begin{equation*}
	\sigma(t) e_n = S_t u_n(t) = \sum_{j=1}^n a_{n, j} (t) S_t A_t d_j = \sum_{j=1}^n a_{n, j}(t) Q_t d_j.
\end{equation*}
Since $t\mapsto Q_t$ is continuous with respect to the operator norm, $t\mapsto Q_t d_j$ is continuous as a $B$-valued map. Each coefficient $a_{n, j}(t)$ is Borel measurable, hence, $t\mapsto \sigma(t) e_n$ is strongly measurable as a $B$-valued map. 
For general $(h_n)_{n\in\N} \in \ell^2$, let $h^{(N)} = \sum_{n=1}^{N} h_n e_n$. Then 
\begin{equation*}
	t \mapsto \sigma(t) h^{(N)} = \sum_{n=1}^{N} h_n \sigma(t) e_n
\end{equation*}
is strongly measurable. Moreover, 
\begin{equation*}
	\sup_{t\in[0, 1]} \| \sigma(t) h - \sigma(t) h^{(N)} \|_B \le \Sigma \|h - h^{(N)}\|_{\ell^2} \xrightarrow{N \to\infty} 0. 
\end{equation*}
Thus, $t\mapsto \sigma(t) h$ is the uniform limit of strongly measurable $B$-valued maps, and as such strongly measurable. 

Similarly, for $b\in B^*$, $(\sigma^*(t) b)_n = \langle \sigma^*(t)b, e_n\rangle_{\ell^2} = b(\sigma(t) e_n)$ is Borel measurable for each $n$. Since $\ell^2$ is separable, coordinate measurability implies strong measurability of $t \mapsto \sigma^*(t) b$. Furthermore, 
\begin{equation*}
	\| \sigma^*(t) b \|_{\ell^2}^2 
	= \langle \sigma^*(t) b, \sigma^*(t) b\rangle_{\ell^2} 
	= b(\sigma(t)\sigma^*(t)b) = b(Q_t b) = q_t(b, b) \le \|Q_t\|_{op} \|b\|_{B^*}^2.
\end{equation*}
Since $t\mapsto Q_t$ is bounded on $[0, 1]$, $t\mapsto \sigma^*(t) b$ belongs to $L^2([0, 1], \ell^2)$ for every $b\in B^*$. 

% $B$ and $H$ are separable, so the maps are Bochner measurable by Pettis measurability theorem \citep[see, e.\,g., Theorem 1.1.6 in][]{hytonen2016}.

\textbf{Stochastic Integral Representation:}
For $b \in B^*$,
\begin{align*}
	\int_0^1 \| \sigma^*(t) b\|_{\ell^2}^2 \diff t
	= \int_0^1 b(Q_t b) \diff t
	= (b \otimes b) \int_0^1 \sigma^2(t) \diff t
	= b (Rb),
\end{align*}
where the operator $R = \Phi(\int_0^1 \sigma^2(t)\diff t) = \int_0^1 Q_t \diff t$ is the covariance operator of the $B$-valued Gaussian random variable $G(1)$, by Theorem \ref{thm:fclt}. By Theorem 4.2 of \cite{vanNeerven2005}, $\sigma$ is stochastically integrable with respect to a cylindrical Brownian motion $W$ on $\ell^2$.

The same applies on every interval $[0, t]$, because $\int_0^t Q_u \diff u = \Phi(\int_0^t \sigma^2(u) \diff u)$ is the covariance operator of the Gaussian random variable $G(t)$ from Theorem \ref{thm:fclt}. Define $\tilde{G}(t) = \int_0^t \sigma(u) \diff W(u)$. Then $\tilde{G}$ is a centered Gaussian process and, for $b, c\in B^*$,
\begin{equation*}
	\cov\big(b(\tilde{G}(t)), c(\tilde{G}(s))\big) = \int_0^{t \wedge s} \langle \sigma^*(u) b, \sigma^*(u) c\rangle_{\ell^2} \diff u = \int_0^{t\wedge s} (b\otimes c)\sigma^2(u) \diff u,
\end{equation*}
since $\sigma(u)\sigma^*(u) = Q_u = \Phi(\sigma^2(u))$, which is exactly the scalar covariance of the Gaussian limit $G$. Since $G$ and $\tilde{G}$ are both centered $B$-valued Gaussian processes and all scalar covariance functions agree, their finite-dimensional distributions coincide, and 
\begin{equation*}
	\big(G(t_1),\dots, G(t_p)\big) \stackrel{\Dc}{=} \bigg(\int_0^{t_1} \sigma(u) \diff W(u), \dots, \int_0^{t_p} \sigma(u) \diff W(u)\bigg),
\end{equation*}
for all $p\in \N$ and $t_1,\dots, t_p\in[0, 1]$, and $\Phi(\sigma^2(t)) = \sigma(t) \sigma^*(t)$. By Proposition \ref{prop:measure_equality}, the laws coincide.

\subsection{Proof of Corollary \ref{cor:test_properties}}

First, note that $b\big(G(t)\big)$ is a centered Gaussian process with covariance 
\begin{equation*}
    \cov\big(b(G(t)), b(G(s))\big) = (b \otimes b) \int_0^{t \wedge s} \sigma^2(x) \diff x = t \wedge s (b\otimes b) \sigma^2.
\end{equation*}
Since $(b\otimes b) \sigma^2$ is constant and strictly positive, $b\big(G(t)\big)$ has the same distribution as $\sqrt{(b\otimes b) \sigma^2} W(t)$, for a standard Brownian motion $W$. Under the null hypothesis,
\begin{equation*}
    D_n(t) = b\big(G_n(t) - t G_n(1)\big) \convw b\big(G(t) - t G(1)\big)  \stackrel{\Dc}{=} \sqrt{(b\otimes b) \sigma^2} \big(W(t) - t W(1)\big),
\end{equation*}
by Theorem \ref{thm:fclt}. Similarly, 
\begin{align*}
    V_n^2(t) & = \frac{1}{n}\sum_{i=1}^{\lfloor nt\rfloor} \big[b\big( G_n(\tfrac{i}{n}) - \tfrac{i}{\lfloor nt\rfloor} G_n(t) \big)\big]^2 + \frac{1}{n}\sum_{i=\lfloor nt\rfloor + 1}^{n} \Big[b\Big( G_n(1) - G_n(\tfrac{i}{n}) - \tfrac{n-i}{n-\lfloor nt\rfloor} \big(G_n(1) - G_n(t)\big) \Big)\Big]^2 \\
    & \convw V^2(t):=(b\otimes b)\sigma^2  \bigg(\int_0^t \big(W(x) - \tfrac{x}{t}W(t)\big)^2 \diff x + \int_t^1 \Big(W(1) - W(x) - \tfrac{1-x}{1-t}[W(1) - W(t)]\Big)^2 \diff x \bigg).
\end{align*}
Hence, under the null hypothesis, $\mathbf{T}_n$ converges weakly to the limiting distribution specified in \eqref{eq:test_limit}.
Under the alternative, $V_n^2(\theta)$ converges weakly to $V^2(\theta)$ as before. Though, for the numerator, it holds
\begin{equation*}
    D_n(\theta) = b\big(G_n(\theta) - \theta G_n(1)\big) + \sqrt{n} \ex[T_n(\theta) - \tfrac{\lfloor \theta n \rfloor}{n}T_n(1)].
\end{equation*}
The first term converges weakly to $b\big(G(\theta) - \theta G(1)\big)$. Explicit calculation of the expectation yields
\begin{equation*}
    \ex[T_n(\theta) - \tfrac{\lfloor \theta n \rfloor}{n}T_n(1)]
    = \tfrac{\lfloor \theta n \rfloor}{n} b(\mu) - \tfrac{\lfloor \theta n \rfloor}{n} b(\mu) - \tfrac{\lfloor \theta n \rfloor}{n} \cdot \tfrac{n - \lfloor \theta n \rfloor}{n}b(\delta) = \theta(1-\theta)b(\delta) + \Oc(n^{-1}).
\end{equation*}
In particular, $D_n(\theta) = \sqrt{n} \theta(1-\theta)b(\delta) + \Oc_\pr(1)$,
such that 
\begin{equation*}
    \mathbf{T}_n \ge \frac{|D_n(\theta_n)|}{V_n(\theta_n)} = \frac{|D_n(\theta)|}{V_n(\theta)} \xrightarrow{n\to\infty} \infty
\end{equation*}
for $\theta_n = \lfloor \theta n \rfloor / n$, under the alternative.

\subsection{Proof of Theorem \ref{thm:bootstrap}} \label{sec:proof_bootstrap}

The proof follows the same steps as the proof of Theorem \ref{thm:fclt}, namely convergence of the finite-dimensional distributions of linear projections, Aldous' condition and a tail estimate analogous to \eqref{eq:escaping_mass}.

\begin{lemma} \label{lem:bootstrap_fdd}
	Let Assumptions \ref{assump:local_stationarity}, \ref{assump:sequences} and \ref{assump:bootstrap} be satisfied. Then, for any $p, K\in\N$, $b_1, \dots, b_p\in B^*$ and $t_1, \dots, t_p\in[0, 1]$,
	\begin{equation*}
		\Big( \big(b_l(G_n(t_l))\big)_{l=1}^p, \big(b_l(\hat{G}_n^{(k)}(t_l))\big)_{l=1, \dots, p}^{k=1, \dots, K}\Big) \convw \Big( \big(b_l(G(t_l))\big)_{l=1}^p, \big(b_l(G^{(k)}(t_l))\big)_{l=1, \dots, p}^{k=1, \dots, K}\Big)
	\end{equation*}
	in $\R^{p(K+1)}$, with $G, G^{(1)}, \dots, G^{(K)}$ as in Theorem \ref{thm:bootstrap}.
\end{lemma}

In the following, we fix the replicate and write $\hat{G}_n = \hat{G}_n^{(1)}$ and $R_i = R_i^{(1)}$, such that the index $k$ again refers to the projections $\Pi_k$ and $T_k$. Let $\Fc_i^R = \sigma\big((\eta_l)_{l\in\Z}, R_1, \dots, R_i\big)$, for $i = 0, \dots, n$. Then, $\hat{G}_n(i/n)$ is $\Fc_i^R$-measurable, so that stopping times with respect to the natural filtration of $\hat{G}_n$ are stopping times with respect to $\{\Fc_i^R\}_{i=1}^n$.

\begin{lemma} \label{lem:bootstrap_aldous}
	Under the assumptions of Theorem \ref{thm:bootstrap}, $\hat{G}_n$ satisfies Aldous' condition, i.\,e., for all $\eps > 0$ and all positive sequences $(h_n)_{n\in\N}$ with $h_n\to 0$,
	\begin{equation} \label{eq:aldous_bootstrap}
		\lim_{n\to\infty} \sup_{\tau} \pr\big(\| \hat{G}_n(\tau + h_n) - \hat{G}_n(\tau) \|_B > \eps\big) = 0,
	\end{equation}
	where the supremum is taken over all stopping times $\tau$ on the grid $\{\tfrac{1}{n}, \tfrac{2}{n}, \dots, \tfrac{n}{n}\}$ with respect to $\{\Fc_i^R\}_{i=1}^n$ satisfying $\tau + h_n\le 1$ almost surely.
\end{lemma}

As in the proof of Theorem \ref{thm:fclt}, tightness of $\hat{G}_n$ in the infinite-dimensional directions relies on a tail estimate, which is the analog of \eqref{eq:escaping_mass} for the bootstrap approximations. In contrast to \eqref{eq:escaping_mass}, it relies on the weighted version of Theorem \ref{thm:max_inequality}.

\begin{lemma} \label{lem:bootstrap_tail}
	Let Assumptions \ref{assump:space}, \ref{assump:local_stationarity}, \ref{assump:bootstrap} and \ref{assump:entropy_uniform} be satisfied. Then,
	\begin{equation} \label{eq:escaping_mass_bootstrap}
		\lim_{k\to\infty} \limsup_{n\to\infty} \ex\big[ \sup_{t\in[0, 1]} \|T_k(\hat{G}_n(t))\|_B \big] = 0.
	\end{equation}
\end{lemma}

\noindent \textbf{Proof of Theorem \ref{thm:bootstrap}:} By Theorem \ref{thm:fclt}, $(G_n)_{n\in\N}$ converges weakly in the Polish space $D([0, 1], B)$, and is therefore tight. By Lemma \ref{lem:bootstrap_aldous}, $\hat{G}_n$ satisfies Aldous' condition. For fixed $t\in[0, 1]$, the random variables $\Pi_k \hat{G}_n(t)$ are tight in the finite-dimensional space $\Pi_k(B)$, by Lemma \ref{lem:bootstrap_fdd}, which, together with Lemma \ref{lem:bootstrap_tail}, implies tightness of $\hat{G}_n(t)$ in $B$. Hence, the same arguments as in the proof of Theorem \ref{thm:fclt}, based on Theorems 23.9 and 23.11 of \cite{kallenberg2021}, yield tightness of $(\hat{G}_n)_{n\in\N}$ in $D([0, 1], B)$. Since $\hat{G}_n^{(k)}$ has the same distribution as $\hat{G}_n$, for each $k\in\N$, the random vectors $(G_n, \hat{G}_n^{(1)}, \dots, \hat{G}_n^{(K)})$ are tight in $D([0, 1], B)^{K+1}$. The proof of Proposition \ref{prop:measure_equality} carries over verbatim to tight random elements of $D([0, 1], B)^{K+1}$, where the class $\Fc$ consists of the functions $(x^{(0)}, \dots, x^{(K)}) \mapsto g\big(b_1(x^{(k_1)}(t_1)), \dots, b_p(x^{(k_p)}(t_p))\big)$, with $k_1, \dots, k_p \in\{0, \dots, K\}$. By Lemma \ref{lem:bootstrap_fdd}, all subsequential limits therefore coincide with the distribution of $(G, G^{(1)}, \dots, G^{(K)})$, and the assertion follows by Prokhorov's theorem. \hfill $\square$

\subsection{Auxiliary Results for the Bootstrap} \label{sec:aux_bootstrap}

Throughout this and the following subsections, $C$ denotes a generic constant, which does not depend on $n$, but may change from line to line. Since $m_n = o(\sqrt{n})$, we may assume $n\ge 2m_n$. For $j = 1, \dots, n$, let
\begin{equation*}
	\bar{\eps}_{j, n} = \frac{1}{|B_j|}\sum_{h\in B_j} \eps_{j+h, n} \quad \text{and} \quad \bar{\mu}_{j, n} = \frac{1}{|B_j|}\sum_{h\in B_j} \mu\big(\tfrac{j+h}{n}\big),
\end{equation*}
such that $X_{j, n} - \hat{\mu}_{j, n} = \eps_{j, n} - \bar{\eps}_{j, n} + \mu(j/n) - \bar{\mu}_{j, n}$. Note that $|B_j| = (m_n \wedge (n-j)) + (m_n\wedge (j-1)) + 1$, such that $m_n + 1 \le |B_j| \le 2m_n+1$, for $n \ge 2m_n$. For $b\in B^*$, we decompose
\begin{equation} \label{eq:bootstrap_decomposition}
	b(\hat{Z}_{i, n}) = U_{i, n}(b) + V_{i, n}^{(1)}(b) + V_{i, n}^{(2)}(b),
\end{equation}
where
\begin{equation*}
	U_{i, n}(b) = \frac{1}{\sqrt{r_n}} \sum_{j\in I_i} b(\eps_{j, n}), \quad
	V_{i, n}^{(1)}(b) = -\frac{1}{\sqrt{r_n}} \sum_{j\in I_i} b(\bar{\eps}_{j, n})
\end{equation*}
and
\begin{equation*}
	V_{i, n}^{(2)}(b) = \frac{1}{\sqrt{r_n}} \sum_{j\in I_i} b\big(\mu(\tfrac{j}{n}) - \bar{\mu}_{j, n}\big).
\end{equation*}

\begin{lemma} \label{lem:bootstrap_blocks}
	Let Assumptions \ref{assump:local_stationarity} and \ref{assump:bootstrap} be satisfied. Then, there exists a constant $C > 0$, such that, for all $b\in B^*$ and $n \ge 2m_n$,
	\begin{enumerate}
		\item[(i)] $\max_{1\le i\le n} \ex[|U_{i, n}(b)|^{2+\kappa}]^{1/(2+\kappa)} \le C \|b\|_{B^*}$,
		\item[(ii)] $\max_{1\le i\le n} \ex[|V_{i, n}^{(1)}(b)|^{2+\kappa}]^{1/(2+\kappa)} \le C \|b\|_{B^*} (r_n/m_n)^{1/2}$,
		\item[(iii)] $\frac{1}{\sqrt{n}} \sum_{j=1}^n \|\mu(j/n) - \bar{\mu}_{j, n}\|_B \le C m_n/\sqrt{n}$ and
		\begin{equation*}
			\frac{1}{n} \sum_{i=1}^n V_{i, n}^{(2)}(b)^2 \le C \|b\|_{B^*}^2 \bigg(\frac{r_n m_n^2}{n^2} + \frac{r_n(m_n + r_n)}{n}\bigg).
		\end{equation*}
	\end{enumerate}
	In particular, all bounds in (ii) and (iii) converge to $0$, as $n\to\infty$.
\end{lemma}

\begin{proof}
	(i) By Theorem 1 of \cite{liu2013}, as in the proof of Proposition \ref{prop:functional_aldous}, there exists a constant $C > 0$, such that
	\begin{equation} \label{eq:moment_block_sums}
		\ex\bigg[\bigg|\sum_{j=j_0+1}^{j_0 + L} b(\eps_{j, n})\bigg|^{2+\kappa}\bigg]^{1/(2+\kappa)} \le C \|b\|_{B^*} \sqrt{L},
	\end{equation}
	for all $b \in B^*$, $n\in\N$ and $0 \le j_0 < j_0 + L \le n$, where we used that $\sum_{j=1}^\infty j\, \delta_{2+\kappa}(H, j) \le \sum_{k=0}^\infty \Theta_k < \infty$, by Assumption \ref{assump:local_stationarity}. Since $|I_i| \le r_n$, the first assertion follows.
	
	(ii) By the triangle inequality and \eqref{eq:moment_block_sums},
	\begin{align*}
		\ex[|V_{i, n}^{(1)}(b)|^{2+\kappa}]^{1/(2+\kappa)}
		& \le \frac{1}{\sqrt{r_n}} \sum_{j\in I_i} \frac{1}{|B_j|} \ex\bigg[\bigg|\sum_{h\in B_j} b(\eps_{j+h, n})\bigg|^{2+\kappa}\bigg]^{1/(2+\kappa)} \\
		& \le \frac{C\|b\|_{B^*}}{\sqrt{r_n}} \sum_{j\in I_i} |B_j|^{-1/2}
		\le C \|b\|_{B^*} \sqrt{\frac{r_n}{m_n}}.
	\end{align*}
	
	(iii) Let $\mathcal{J}_n = \{j\in\{1, \dots, n\}: |j - n\theta_l| \le m_n, \text{ for some } l\in\{1, \dots, J\}\}$, such that $|\mathcal{J}_n| \le J(2m_n + 1)$. For $j\notin\mathcal{J}_n$ and $h\in B_j$, the interval between $j/n$ and $(j+h)/n$ does not contain any of the points $\theta_1, \dots, \theta_J$, such that $\|\mu(j/n) - \mu((j+h)/n)\|_B \le L_\mu m_n/n$, by Assumption \ref{assump:bootstrap}, and therefore $\|\mu(j/n) - \bar{\mu}_{j, n}\|_B \le L_\mu m_n /n$. For $j \in \mathcal{J}_n$, it holds $\|\mu(j/n) - \bar{\mu}_{j, n}\|_B \le 2 \|\mu\|_\infty$. Hence,
	\begin{equation*}
		\frac{1}{\sqrt{n}} \sum_{j=1}^n \|\mu(j/n) - \bar{\mu}_{j, n}\|_B \le \frac{L_\mu m_n}{\sqrt{n}} + \frac{2\|\mu\|_\infty J(2m_n+1)}{\sqrt{n}}.
	\end{equation*}
	Similarly, $|V_{i, n}^{(2)}(b)| \le \|b\|_{B^*} \big( \sqrt{r_n} L_\mu m_n/n + 2\|\mu\|_\infty \sqrt{r_n} \id(I_i \cap \mathcal{J}_n \neq \emptyset)\big)$, and $I_i \cap\mathcal{J}_n \neq \emptyset$ holds for at most $J(2m_n + 1 + r_n)$ indices $i$. This yields the second bound in (3). Finally, Assumption \ref{assump:bootstrap} (1) implies $m_n/\sqrt{n}\to 0$ and $r_n/m_n \to 0$, as well as $r_n(m_n+r_n)/n \le 2 m_n^2/n \to 0$ and $r_n m_n^2/n^2 \le m_n^3/n^2 \to 0$, since $r_n\le m_n$, for sufficiently large $n$.
\end{proof}

\begin{lemma} \label{lem:bootstrap_cov}
	Let Assumptions \ref{assump:local_stationarity} and \ref{assump:bootstrap} be satisfied, $b\in B^*$ and $t\in[0, 1]$. Then,
	\begin{equation*}
		Q_n(t) := \frac{1}{n}\sum_{i=1}^{\lfloor nt\rfloor} b(\hat{Z}_{i, n})^2 \xrightarrow{\pr} \int_0^t (b\otimes b)\sigma^2(u) \diff u,
	\end{equation*}
	as $n\to\infty$.
\end{lemma}

\begin{proof}
	Abbreviate $U_i = U_{i, n}(b)$ and $V_i = V_{i, n}^{(1)}(b) + V_{i, n}^{(2)}(b)$. By Lemma \ref{lem:bootstrap_blocks}, $\max_{1 \le i\le n}\ex[U_i^2] \le C$ and $\frac{1}{n}\sum_{i=1}^n \ex[V_i^2] \to 0$. By the Cauchy-Schwarz inequality,
	\begin{equation*}
		\ex\bigg[\bigg| Q_n(t) - \frac{1}{n}\sum_{i=1}^{\lfloor nt\rfloor} U_i^2\bigg|\bigg] 
		\le 2 \bigg(\frac{1}{n}\sum_{i=1}^n \ex[U_i^2]\bigg)^{1/2}\bigg(\frac{1}{n}\sum_{i=1}^n \ex[V_i^2]\bigg)^{1/2} + \frac{1}{n}\sum_{i=1}^n \ex[V_i^2] \to 0.
	\end{equation*}
	Hence, it suffices to show
	\begin{equation} \label{eq:bootstrap_cov_mean}
		\frac{1}{n}\sum_{i=1}^{\lfloor nt\rfloor} \ex[U_i^2] \to \int_0^t (b\otimes b)\sigma^2(u) \diff u
	\end{equation}
	and
	\begin{equation} \label{eq:bootstrap_cov_lln}
		\ex\bigg[\bigg| \frac{1}{n}\sum_{i=1}^{\lfloor nt\rfloor} \big(U_i^2 - \ex[U_i^2]\big)\bigg|\bigg] \to 0.
	\end{equation}
	
	\textbf{Proof of \eqref{eq:bootstrap_cov_mean}:} Let $i \le n - r_n + 1$, such that $I_i = \{i, \dots, i+r_n-1\}$. For $j_1, j_2\in I_i$, Lipschitz continuity of $H$ with respect to the $L^2$-norm and the moment bound in Assumption \ref{assump:local_stationarity} imply
	\begin{equation*}
		\big|\ex[b(\eps_{j_1, n}) b(\eps_{j_2, n})] - \ex\big[b\big(H(\tfrac{i}{n}, \Fc_{j_1})\big) b\big(H(\tfrac{i}{n}, \Fc_{j_2})\big)\big]\big| \le C \frac{r_n}{n}.
	\end{equation*}
	By Proposition \ref{prop:cov_bound} and the same arguments as in its proof, both expectations are bounded by $C \Theta_{|j_1-j_2|}$ in absolute value. Hence, by stationarity of $\{H(i/n, \Fc_j)\}_{j\in\Z}$,
	\begin{align*}
		& \bigg|\ex[U_i^2] - \frac{1}{r_n} \sum_{j_1, j_2 = 1}^{r_n} \ex\big[b\big(H(\tfrac{i}{n}, \Fc_0)\big) b\big(H(\tfrac{i}{n}, \Fc_{j_2-j_1})\big)\big]\bigg| \\
		& \le \frac{C}{r_n} \sum_{j_1, j_2\in I_i} \min\Big\{\frac{r_n}{n}, \Theta_{|j_1-j_2|}\Big\}
		\le 2C \sum_{h=0}^\infty \min\Big\{\frac{r_n}{n}, \Theta_h\Big\},
	\end{align*}
	which converges to $0$, uniformly in $i$, by dominated convergence, since $r_n/n\to0$ and $\sum_{h=0}^\infty\Theta_h <\infty$. By Proposition \ref{prop:lrv_approximation}, whose proof only uses $\ell_n\to\infty$, applied with $\ell_n$ replaced by $r_n$, the second term inside the absolute value equals $(b\otimes b)\sigma^2(i/n) + o(1)$, uniformly in $i$. For the at most $r_n$ indices $i > n - r_n+1$, it holds $\ex[U_i^2]\le C$, by Lemma \ref{lem:bootstrap_blocks}. Consequently,
	\begin{equation*}
		\frac{1}{n}\sum_{i=1}^{\lfloor nt\rfloor} \ex[U_i^2] = \frac{1}{n}\sum_{i=1}^{\lfloor nt\rfloor} (b\otimes b)\sigma^2(\tfrac{i}{n}) + \Oc\Big(\frac{r_n}{n}\Big) + o(1),
	\end{equation*}
	and \eqref{eq:bootstrap_cov_mean} follows, since $u\mapsto (b\otimes b)\sigma^2(u)$ is continuous by Lemma \ref{lem:lrv_lipschitz}.
	
	\textbf{Proof of \eqref{eq:bootstrap_cov_lln}:} Fix $s\in\N$, let $\tilde{\eps}_{j, n} = \ex[\eps_{j, n}|\eta_j, \dots, \eta_{j-s}]$, as in Proposition \ref{prop:sn_dependence}, and $\tilde{U}_i = r_n^{-1/2}\sum_{j\in I_i} b(\tilde{\eps}_{j, n})$. For $h\in\N$, let
	\begin{equation*}
		d_{j, h} = \ex[b(\eps_{j, n})|\eta_j, \dots, \eta_{j-h}] - \ex[b(\eps_{j, n})|\eta_j, \dots, \eta_{j-h+1}].
	\end{equation*}
	As in the proof of Proposition \ref{prop:sn_dependence}, $b(\eps_{j, n}) - b(\tilde{\eps}_{j, n}) = \sum_{h=s+1}^\infty d_{j, h}$ in $L^2$, and $\ex[d_{j, h}^2]^{1/2}\le \|b\|_{B^*}\delta_2(H, h)$. Moreover, for $j_1 < j_2$, $d_{j_2, h}$ is measurable with respect to $\Ac_{j_1, h} = \sigma(\eta_l: l\ge j_1-h+1)$, whereas $\ex[d_{j_1, h}|\Ac_{j_1, h}] = 0$, since $\eta_{j_1-h}$ is independent of $\Ac_{j_1, h}$. Hence, for fixed $h$, the random variables $(d_{j, h})_{j\in I_i}$ are uncorrelated, and
	\begin{equation*}
		\ex[(U_i - \tilde{U}_i)^2]^{1/2}
		\le \frac{1}{\sqrt{r_n}} \sum_{h=s+1}^\infty \bigg(\sum_{j\in I_i}\ex[d_{j, h}^2]\bigg)^{1/2}
		\le \|b\|_{B^*}\sum_{h=s+1}^\infty \delta_2(H, h)
		\le \|b\|_{B^*} \Theta_{s+1}.
	\end{equation*}
	The dependence coefficients and moments of $\tilde{\eps}_{j, n}$ are bounded by those of $\eps_{j, n}$, so that \eqref{eq:moment_block_sums} holds for $\tilde{\eps}_{j,n}$ with the same constant, and $\max_i \ex[|\tilde{U}_i|^{2+\kappa}] \le C$, uniformly in $s$. In particular, $\ex[|U_i^2 - \tilde{U}_i^2|] \le \ex[(U_i - \tilde{U}_i)^2]^{1/2}(\ex[U_i^2]^{1/2} + \ex[\tilde{U}_i^2]^{1/2}) \le C \Theta_{s+1}$.
	
	Since $\tilde{U}_i$ is measurable with respect to $\sigma(\eta_{i-s}, \dots, \eta_{i+r_n-1})$, the random variables $\tilde{U}_{i_1}$ and $\tilde{U}_{i_2}$ are independent, whenever $|i_1-i_2|\ge r_n + s$. For $M > 0$, let $\tilde{U}_i^M = \tilde{U}_i^2 \wedge M$. Then,
	\begin{equation*}
		\var\bigg(\frac{1}{n}\sum_{i=1}^{\lfloor nt\rfloor}\tilde{U}_i^M\bigg) \le \frac{1}{n^2} \sum_{\substack{i_1, i_2=1\\ |i_1-i_2|<r_n+s}}^n M^2 \le \frac{2(r_n+s)M^2}{n},
	\end{equation*}
	and $\ex[\tilde{U}_i^2 - \tilde{U}_i^M] \le \ex[\tilde{U}_i^2 \id(\tilde{U}_i^2 > M)] \le C M^{-\kappa/2}$. Combining these bounds yields
	\begin{equation*}
		\ex\bigg[\bigg| \frac{1}{n}\sum_{i=1}^{\lfloor nt\rfloor} \big(U_i^2 - \ex[U_i^2]\big)\bigg|\bigg]
		\le C\Theta_{s+1} + C M^{-\kappa/2} + \Big(\frac{2(r_n+s)M^2}{n}\Big)^{1/2}.
	\end{equation*}
	Since $r_n/n\to 0$, letting first $n\to\infty$, then $M\to\infty$ and finally $s\to\infty$ yields \eqref{eq:bootstrap_cov_lln}.
\end{proof}

\subsection{Proof of Lemma \ref{lem:bootstrap_fdd}}

Let $Y_n = (b_l(G_n(t_l)))_{l=1}^p$ and $\hat{Y}_n^{(k)} = (b_l(\hat{G}_n^{(k)}(t_l)))_{l=1}^p$, for $k = 1, \dots, K$. Since the bootstrap approximations are linear in the independent standard normal multipliers, the vectors $\hat{Y}_n^{(1)}, \dots, \hat{Y}_n^{(K)}$ are, conditionally on $(\eta_l)_{l\in\Z}$, independent and centered Gaussian with covariance matrix $\hat{\Sigma}_n = (\hat{\Sigma}_{n, l, l'})_{l, l'=1}^p$, where
\begin{equation*}
	\hat{\Sigma}_{n, l, l'} = \frac{1}{n} \sum_{i=1}^{\lfloor n(t_l\wedge t_{l'})\rfloor} b_l(\hat{Z}_{i, n}) b_{l'}(\hat{Z}_{i, n}).
\end{equation*}
Let $Q_n^b$ denote $Q_n$ from Lemma \ref{lem:bootstrap_cov} for the functional $b$, then $\hat{\Sigma}_{n, l, l'} = \frac{1}{4}\big(Q_n^{b_l + b_{l'}} - Q_n^{b_l - b_{l'}}\big)(t_l\wedge t_{l'})$. By Lemma \ref{lem:bootstrap_cov} and the symmetry $(b\otimes c)\sigma^2(u) = (c\otimes b)\sigma^2(u)$, for $b, c\in B^*$, which follows from the representation of $(b \otimes c)\sigma^2(u)$ in the proof of Proposition \ref{prop:lrv_approximation} and stationarity of $\{H(u, \Fc_i)\}_{i\in\Z}$, it follows that
\begin{equation*}
	\hat{\Sigma}_{n, l, l'} \xrightarrow{\pr} \Sigma_{l, l'} := (b_l\otimes b_{l'}) \int_0^{t_l\wedge t_{l'}} \sigma^2(u) \diff u = \cov\big(b_l(G(t_l)), b_{l'}(G(t_{l'}))\big).
\end{equation*}
Let $a_0, \dots, a_K\in\R^p$. Since $Y_n$ is measurable with respect to $\sigma((\eta_l)_{l\in\Z})$,
\begin{equation*}
	\ex\bigg[\exp\bigg(\mathrm{i}\, a_0^\top Y_n + \mathrm{i} \sum_{k=1}^K a_k^\top \hat{Y}_n^{(k)}\bigg)\bigg]
	= \ex\bigg[\exp\big(\mathrm{i}\, a_0^\top Y_n\big) \exp\bigg(-\frac{1}{2}\sum_{k=1}^K a_k^\top \hat{\Sigma}_n a_k\bigg)\bigg].
\end{equation*}
By Lemma \ref{lem:fdd}, $Y_n\convw Y = (b_l(G(t_l)))_{l=1}^p$, and hence $(Y_n, \hat{\Sigma}_n)\convw (Y, \Sigma)$, by Slutsky's lemma. Since $\hat{\Sigma}_n$ is positive semi-definite, the integrand coincides with $g(Y_n, \hat{\Sigma}_n)$ for the bounded and continuous function $g(y, S) = \exp\big(\mathrm{i}\, a_0^\top y - \frac{1}{2} \big(\sum_{k=1}^K a_k^\top S a_k\big) \vee 0\big)$ on $\R^p\times\R^{p\times p}$. Hence, by the continuous mapping theorem, the right-hand side converges to
\begin{equation*}
	\ex\big[\exp(\mathrm{i}\, a_0^\top Y)\big] \prod_{k=1}^K \exp\Big(-\frac{1}{2} a_k^\top \Sigma a_k\Big).
\end{equation*}
This is the characteristic function of the limit in Lemma \ref{lem:bootstrap_fdd}, since $G^{(1)}, \dots, G^{(K)}$ are independent copies of $G$, which are independent of $G$. The assertion follows by Lévy's continuity theorem. \hfill $\square$

\subsection{Proof of Lemma \ref{lem:bootstrap_aldous}}

The proof of Proposition \ref{prop:norm_aldous} only uses \eqref{eq:functional_aldous} and \eqref{eq:escaping_mass}. Hence, with $G_n$ replaced by $\hat{G}_n$, the filtration $\{\sigma(\Fc_i)\}_{i=1}^n$ replaced by $\{\Fc_i^R\}_{i=1}^n$, and \eqref{eq:escaping_mass} replaced by \eqref{eq:escaping_mass_bootstrap} from Lemma \ref{lem:bootstrap_tail}, it yields \eqref{eq:aldous_bootstrap}, provided that
\begin{equation} \label{eq:functional_aldous_bootstrap}
	\lim_{n\to\infty}\sup_\tau \pr\big(|b(\hat{G}_n(\tau + h_n) - \hat{G}_n(\tau))| > \eps\big) = 0,
\end{equation}
for all $b\in B^*$, $\eps>0$ and all positive sequences $(h_n)_{n\in\N}$ with $h_n\to 0$, where the supremum is taken over all stopping times $\tau$ as in Lemma \ref{lem:bootstrap_aldous}.

Let $b\in B^*$ and $\tau$ be a stopping time as in Lemma \ref{lem:bootstrap_aldous}. Since $\{\tau n = \nu\}$ and $\hat{Z}_{1, n}, \dots, \hat{Z}_{n, n}$ are $\Fc_\nu^R$-measurable, whereas $R_{\nu+1}, \dots, R_n$ are independent of $\Fc_\nu^R$, it holds
\begin{align*}
	& \ex\big[|b(\hat{G}_n(\tau + h_n) - \hat{G}_n(\tau))|^2\big] \\
	& = \sum_{\nu=1}^n \ex\bigg[\id(\tau n = \nu)\, \ex\bigg[\bigg(\frac{1}{\sqrt{n}} \sum_{i=\nu+1}^{\nu+\lfloor nh_n\rfloor} R_i b(\hat{Z}_{i, n})\bigg)^2 \bigg| \Fc_\nu^R\bigg]\bigg] \\
	& = \ex\bigg[\frac{1}{n} \sum_{i=\tau n+1}^{\tau n+\lfloor nh_n\rfloor} b(\hat{Z}_{i, n})^2\bigg]
	\le \ex\bigg[\max_{0\le \nu\le n-\lfloor nh_n\rfloor} \frac{1}{n} \sum_{i=\nu+1}^{\nu+\lfloor nh_n\rfloor} b(\hat{Z}_{i, n})^2\bigg].
\end{align*}
Let $S_i = U_{i, n}(b) + V_{i, n}^{(1)}(b)$, such that $b(\hat{Z}_{i, n}) = S_i + V_{i, n}^{(2)}(b)$, by \eqref{eq:bootstrap_decomposition}, and $\max_{1\le i \le n}\ex[|S_i|^{2+\kappa}] \le C$, by Lemma \ref{lem:bootstrap_blocks} and since $r_n\le m_n$, for sufficiently large $n$. For $M > 0$, it holds $S_i^2 \le M + S_i^2\id(S_i^2 > M)$, such that
\begin{equation*}
	\max_{0\le \nu\le n-\lfloor nh_n\rfloor} \frac{1}{n} \sum_{i=\nu+1}^{\nu+\lfloor nh_n\rfloor} b(\hat{Z}_{i, n})^2
	\le \frac{2\lfloor nh_n\rfloor M}{n} + \frac{2}{n}\sum_{i=1}^n S_i^2\id(S_i^2 > M) + \frac{2}{n}\sum_{i=1}^n V_{i, n}^{(2)}(b)^2.
\end{equation*}
Since $\ex[S_i^2\id(S_i^2 > M)]\le CM^{-\kappa/2}$, the expectation of the right-hand side is bounded by $2h_n M + 2CM^{-\kappa/2} + o(1)$, by Lemma \ref{lem:bootstrap_blocks}. With $M = h_n^{-1/2}$, the bound converges to $0$, uniformly in $\tau$, and \eqref{eq:functional_aldous_bootstrap} follows by Markov's inequality. \hfill $\square$

\subsection{Proof of Lemma \ref{lem:bootstrap_tail}}

\textbf{Step 1 (Supremum over $t$):} Conditionally on $\Fc_0^R = \sigma((\eta_l)_{l\in\Z})$, the random variables $R_i T_k\hat{Z}_{i, n}$, $i = 1, \dots, n$, are independent and symmetric. Since $\hat{G}_n$ is constant between grid points, Lévy's inequality \citep[see Proposition 2.3 in][]{ledoux2013} yields, for $x>0$,
\begin{align*}
	\pr\Big(\sup_{t\in[0, 1]} \|T_k(\hat{G}_n(t))\|_B > x \,\Big|\, \Fc_0^R\Big)
	& = \pr\bigg(\max_{1\le \nu\le n} \bigg\|\frac{1}{\sqrt{n}}\sum_{i=1}^\nu R_i T_k \hat{Z}_{i, n}\bigg\|_B > x \,\bigg|\, \Fc_0^R\bigg) \\
	& \le 2\pr\big( \|T_k(\hat{G}_n(1))\|_B > x \,\big|\, \Fc_0^R\big).
\end{align*}
Integrating over $x$ and taking expectations gives
\begin{equation} \label{eq:bootstrap_levy}
	\ex\big[\sup_{t\in[0, 1]} \|T_k(\hat{G}_n(t))\|_B\big] \le 2 \ex\big[\|T_k(\hat{G}_n(1))\|_B\big].
\end{equation}

\textbf{Step 2 (Representation at $t=1$):} Interchanging the order of summation yields
\begin{equation*}
	\hat{G}_n(1) = \frac{1}{\sqrt{n}} \sum_{j=1}^n w_{j, n} (X_{j, n} - \hat{\mu}_{j, n}),
\end{equation*}
where $w_{j, n} = r_n^{-1/2} \sum_{i=(j-r_n+1)\vee 1}^{j} R_i$ satisfies $\ex[w_{j,n}^2] = (j\wedge r_n)/r_n\le 1$. Let $\gamma_{j, l} = |B_j|^{-1}\id(l-j\in B_j)$, such that $\bar{\eps}_{j, n} = \sum_{l=1}^n \gamma_{j, l} \eps_{l, n}$ and $\sum_{j=1}^n \gamma_{j, l} \le (2m_n+1)/(m_n+1) \le 2$. With the decomposition $X_{j, n} - \hat{\mu}_{j, n} = \eps_{j, n} - \bar{\eps}_{j, n} + \mu(j/n) - \bar{\mu}_{j, n}$, we obtain
\begin{equation} \label{eq:bootstrap_representation}
	T_k(\hat{G}_n(1)) = \frac{1}{\sqrt{n}}\sum_{l=1}^n \tilde{w}_{l, n} T_k\eps_{l, n} + \frac{1}{\sqrt{n}}\sum_{j=1}^n w_{j, n} T_k\big(\mu(\tfrac{j}{n}) - \bar{\mu}_{j, n}\big),
\end{equation}
where $\tilde{w}_{l, n} = w_{l, n} - \sum_{j=1}^n \gamma_{j, l} w_{j, n}$. The weights $\tilde{w}_{l, n}$ are $\sigma(R_1, \dots, R_n)$-measurable, and thus independent of $(\eta_l)_{l\in\Z}$. Moreover, they are centered Gaussian with $\ex[\tilde{w}_{l, n}^2]^{1/2}\le \ex[w_{l,n}^2]^{1/2} + \sum_{j=1}^n \gamma_{j, l} \ex[w_{j,n}^2]^{1/2} \le 3$, by the triangle inequality. For the second term in \eqref{eq:bootstrap_representation}, we have by $\ex[|w_{j,n}|]\le 1$ and Lemma \ref{lem:bootstrap_blocks},
\begin{equation*}
	\ex\bigg[\bigg\| \frac{1}{\sqrt{n}}\sum_{j=1}^n w_{j, n} T_k\big(\mu(\tfrac{j}{n}) - \bar{\mu}_{j, n}\big)\bigg\|_B\bigg]
	\le \frac{1+K_B}{\sqrt{n}}\sum_{j=1}^n \big\|\mu(\tfrac{j}{n}) - \bar{\mu}_{j, n}\big\|_B \to 0,
\end{equation*}
uniformly in $k$.

\textbf{Step 3 (Conditional maximal inequality):} Fix $n \ge 2m_n$ and $x = (x_1, \dots, x_n)\in\R^n$. Let $\omega_x: [0, 1] \to \R$ be defined by $\omega_x(u) = x_{\lceil nu\rceil}$, for $u\in(0, 1]$, and $\omega_x(0) = 0$, such that $\omega_x(i/n) = x_i$, and let $\nu_n(x) = (\frac{1}{n}\sum_{i=1}^n x_i^2)^{1/2}$. Consider the class
\begin{equation*}
	\Gc_k^x = \{f: B\times[0, 1]\to\R \,|\, f(y, u) = \omega_x(u) \bar{f}(y, u), \bar{f}\in\bar{\Gc}_k\}.
\end{equation*}
For $f(y, u) = \omega_x(u) b(T_k y)$, it holds $\ex[f(X_{i, n}, i/n)] = x_i b(T_k\mu(i/n))$. Since $\|z\|_B = \sup_{b\in B_1^*} |b(z)|$, by the Hahn-Banach theorem, we obtain
\begin{equation*}
	\sup_{f\in\Gc_k^x} |\Gb_n(f)| = \bigg\|\frac{1}{\sqrt{n}}\sum_{l=1}^n x_l T_k\eps_{l, n}\bigg\|_B.
\end{equation*}
We apply Theorem \ref{thm:max_inequality} to $\Gc_k^x$, with $\omega = \omega_x$, $\Delta(i) = \delta_2(H, i)$ and $D_{\omega, n} = (1\vee\nu_n(x))D$. Since $|b(T_k y) - b(T_k z)|\le (1+K_B)\|y - z\|_B \le D \|y-z\|_B$, for $b\in B_1^*$, and $X_{i, n} - X_{i, n}^{(i-l)} = \eps_{i, n} - \eps_{i, n}^{(i-l)}$, we have
\begin{align*}
	\ex\big[\sup_{f\in\Gc_k^x}|W_i(f) - W_i^{(l)}(f)|^2\big]^{1/2}
	& \le |x_i| D\, \ex\big[\|\eps_{i, n} - \eps_{i, n}^{(i-l)}\|_B^2\big]^{1/2} \\
	& \le |\omega_x(i/n)| D\, \delta_2(H, l),
\end{align*}
for all $1\le i\le n$ and $l\in\N$, where $\eps_{i, n}^{(i-l)}$ is defined analogously to $X_{i, n}^{(i-l)}$. Moreover, for any measurable $g:B\times[0, 1]\to\R$,
\begin{equation*}
	\|\omega_x g\|_{2, n} = \bigg(\frac{1}{n}\sum_{i=1}^n x_i^2\, \ex[|g(X_{i, n}, i/n)|^2]\bigg)^{1/2} \le \nu_n(x)\|g\|_{2, \infty, n}.
\end{equation*}
In particular, $\|\omega_x\bar{f}\|_{2, n} \le \nu_n(x) D \le D_{\omega, n}$, for $\bar f\in \bar{\Gc}_k$, so that Assumption \ref{assump:functions} is satisfied. Since $a\mapsto a + \sum_{i=1}^\infty \min\{a, D_{\omega, n} \delta_2(H, i)\}$ is non-decreasing and $D_{\omega, n} = (1\vee\nu_n(x))D$, the previous display implies
\begin{equation} \label{eq:weighted_seminorm_bound}
	V_n(\omega_x g) \le (1\vee\nu_n(x)) \tilde{V}_{n, k}(g),
\end{equation}
for any measurable $g:B\times[0, 1]\to\R$. Hence, $\sup_{f\in\Gc_k^x} V_n(f)\le \sigma := (1\vee\nu_n(x))\tilde{\sigma}_k$. If $[\ell, u]$ is a bracket for $\bar{\Gc}_k$ with $\tilde{V}_{n, k}(u-\ell)\le\eps$, then $[\omega_x^+\ell - \omega_x^- u, \omega_x^+ u - \omega_x^-\ell]$ is a bracket containing $\omega_x \bar{f}$, for all $\bar{f}\in[\ell, u]$, where $\omega_x^+$ and $\omega_x^-$ denote the positive and negative part of $\omega_x$. Its width is $|\omega_x|(u-\ell)$, such that $V_n(|\omega_x|(u-\ell)) \le (1\vee\nu_n(x))\eps$, by \eqref{eq:weighted_seminorm_bound}. Consequently,
\begin{equation*}
	N(\eps, \Gc_k^x, V_n) \le N\big(\eps/(1\vee\nu_n(x)), \bar{\Gc}_k, \tilde{V}_{n, k}\big),
\end{equation*}
and the substitution $\eps = (1\vee\nu_n(x))\eps'$ yields
\begin{equation*}
	\int_0^\sigma \sqrt{1\vee\log(N(\eps, \Gc_k^x, V_n))}\diff\eps
	\le (1\vee\nu_n(x)) \tilde{I}_{n, k},
\end{equation*}
where $\tilde{I}_{n,k} = \int_0^{\tilde{\sigma}_k} \sqrt{1 \vee \log(N(\eps, \bar{\Gc}_k, \tilde{V}_{n, k}))} \diff \eps$, such that $\tilde{I}_k = \limsup_{n\to\infty}\tilde{I}_{n, k}$.
For the truncation level, we obtain
\begin{align*}
	m_{\omega_x}\big(n, \sigma, N(\sigma/2, \Gc_k^x, V_n)\big)
	& = \frac{r(\tilde{\sigma}_k/D) (1\vee\nu_n(x)) D\sqrt{n}}{\sqrt{1\vee\log N(\sigma/2, \Gc_k^x, V_n)}} \\
	& \ge \frac{r(\tilde{\sigma}_k/D) D\sqrt{n}}{\sqrt{1\vee\log N(\tilde{\sigma}_k/2, \bar{\Gc}_k, \tilde{V}_{n, k})}} =: 4M_{n, k},
\end{align*}
where $M_{n,k}$ does not depend on $x$. With the envelope $F^x(y, u) = |\omega_x(u)|\|T_ky\|_B$ of $\Gc_k^x$, Theorem \ref{thm:max_inequality} yields
\begin{multline} \label{eq:conditional_max_ineq}
	\ex\bigg[\bigg\|\frac{1}{\sqrt{n}}\sum_{l=1}^n x_l T_k\eps_{l, n}\bigg\|_B\bigg] \\
	\le c\bigg[ (1\vee\nu_n(x))\tilde{I}_{n, k} + \frac{1}{\sqrt{n}}\sum_{i=1}^n |x_i| \ex\big[\|T_kX_{i, n}\|_B\id\big(|x_i|\|T_kX_{i, n}\|_B > M_{n, k}\big)\big]\bigg].
\end{multline}

\textbf{Step 4 (Integration over the multipliers):} Let $\tilde{w}_n = (\tilde{w}_{1,n}, \dots, \tilde{w}_{n, n})$. Since $\tilde{w}_n$ is independent of $(\eta_l)_{l\in\Z}$, Fubini's theorem and \eqref{eq:conditional_max_ineq} with $x = \tilde{w}_n$ yield
\begin{align*}
	\ex\bigg[\bigg\|\frac{1}{\sqrt{n}}\sum_{l=1}^n \tilde{w}_{l, n} T_k\eps_{l, n}\bigg\|_B\bigg] 
	& \le c\, \ex\big[1\vee\nu_n(\tilde{w}_n)\big]\, \tilde{I}_{n, k} \\
	& \quad + \frac{c}{\sqrt{n}} \sum_{i=1}^n \ex\big[|\tilde{w}_{i, n}|\|T_kX_{i, n}\|_B \id\big(|\tilde{w}_{i, n}|\|T_kX_{i, n}\|_B > M_{n, k}\big)\big].
\end{align*}
First, $\ex[1\vee\nu_n(\tilde{w}_n)] \le 1 + \ex[\nu_n(\tilde{w}_n)^2]^{1/2}\le 4$. Second, since $y\id(y > M) \le y^{2+\kappa}/M^{1+\kappa}$, for $y\ge 0$, and by independence,
\begin{equation*}
	\ex\big[|\tilde{w}_{i, n}|\|T_kX_{i, n}\|_B \id\big(|\tilde{w}_{i, n}|\|T_kX_{i, n}\|_B > M_{n, k}\big)\big]
	\le \frac{\ex[|\tilde{w}_{i, n}|^{2+\kappa}]\, \ex[\|T_kX_{i, n}\|_B^{2+\kappa}]}{M_{n, k}^{1+\kappa}}
	\le \frac{C}{M_{n, k}^{1+\kappa}},
\end{equation*}
since $\tilde{w}_{i, n}$ is centered Gaussian with variance at most $9$, $\|T_k\|_{op}\le 1+K_B$, $\mu$ is bounded and $\sup_{t\in[0, 1]} \ex[\|H(t, \Fc_0)\|_B^{2+\kappa}]<\infty$. Since $\eps\mapsto N(\eps, \bar{\Gc}_k, \tilde{V}_{n, k})$ is non-increasing,
\begin{equation*}
	\frac{\tilde{\sigma}_k}{2}\sqrt{1\vee\log N(\tilde{\sigma}_k/2, \bar{\Gc}_k, \tilde{V}_{n, k})} \le \int_0^{\tilde{\sigma}_k/2} \sqrt{1 \vee \log(N(\eps, \bar{\Gc}_k, \tilde{V}_{n, k}))} \diff \eps \le \tilde{I}_{n, k},
\end{equation*}
and hence $M_{n, k} \ge r(\tilde{\sigma}_k/D) D \tilde{\sigma}_k \sqrt{n}/(8 \tilde{I}_{n, k})$, where $r(\tilde{\sigma}_k/D) > 0$, since $q^*(x)x \to 0$, as $x\downarrow 0$. Indeed, if $q^*(x) \ge 2$, then $q^*(x)x < \beta(q^*(x) - 1) + x$, by minimality of $q^*(x)$. Since $\limsup_{n\to\infty}\tilde{I}_{n, k} = \tilde{I}_k < \infty$, by Assumption \ref{assump:entropy_uniform}, we obtain
\begin{equation*}
	\limsup_{n\to\infty} \frac{1}{\sqrt{n}} \sum_{i=1}^n \frac{C}{M_{n, k}^{1+\kappa}} \le \limsup_{n\to\infty} C_k \frac{\tilde{I}_{n,k}^{1+\kappa}}{n^{\kappa/2}} = 0,
\end{equation*}
for each fixed $k\in\N$ and a constant $C_k$, which does not depend on $n$.

\textbf{Step 5 (Conclusion):} Combining \eqref{eq:bootstrap_levy}, \eqref{eq:bootstrap_representation} and the bounds from Steps 2 and 4 yields
\begin{equation*}
	\limsup_{n\to\infty} \ex\big[\sup_{t\in[0, 1]} \|T_k(\hat{G}_n(t))\|_B\big] \le 8c\, \tilde{I}_k,
\end{equation*}
which converges to $0$, as $k\to\infty$, by Assumption \ref{assump:entropy_uniform}. \hfill $\square$

\subsection{Proof of Corollary \ref{cor:bootstrap_test}}

Define the maps $\Psi_n, \Psi: D([0, 1], B)\to[0, \infty)$ by
\begin{equation*}
    \Psi_n(x) = \sup_{t\in[0, 1]} \zeta\big(x(t) - \tfrac{\lfloor nt\rfloor}{n}x(1)\big) \quad \text{and} \quad \Psi(x) = \sup_{t\in[0, 1]} \zeta\big(x(t) - t x(1)\big),
\end{equation*}
such that $\hat{\mathbf{C}}_n^{(k)} = \Psi_n(\hat{G}_n^{(k)})$, for $k\in\N$. Since $\ex[X_{i, n}] = \mu(i/n)$, it holds
\begin{equation} \label{eq:cusum_decomposition}
    \mathbf{C}_n = \sup_{t\in[0, 1]} \zeta\Big(G_n(t) - \tfrac{\lfloor nt\rfloor}{n}G_n(1) + \sqrt{n}\Lambda_n(t)\Big), \qquad
    \Lambda_n(t) = \frac{1}{n}\sum_{i=1}^{\lfloor nt\rfloor}\mu(\tfrac{i}{n}) - \frac{\lfloor nt\rfloor}{n^2}\sum_{i=1}^n \mu(\tfrac{i}{n}).
\end{equation}
Since $|\lfloor nt\rfloor/n - t|\le 1/n$, we have $|\Psi_n(x) - \Psi(x)|\le C_\zeta\|x(1)\|_B/n$, for all $x\in D([0, 1], B)$. Moreover, $\Psi$ is Lipschitz continuous with respect to the uniform norm, since $|\Psi(x) - \Psi(y)| \le 2C_\zeta\sup_{t\in[0, 1]}\|x(t) - y(t)\|_B$. As the Skorokhod topology on $D([0, 1], B)$ is weaker than the uniform topology, this does not immediately imply continuity of $\Psi$ on $D([0, 1], B)$. However, by definition of the Skorokhod topology, $x_n\to x$ in $D([0, 1], B)$ if and only if there exist strictly increasing, continuous bijections $\lambda_n:[0, 1]\to[0, 1]$ with $\sup_{t\in[0, 1]}|\lambda_n(t) - t|\to 0$ and $\sup_{t\in[0, 1]}\|x_n(\lambda_n(t)) - x(t)\|_B\to 0$. Since $\lambda_n$ is a bijection with $\lambda_n(1) = 1$, it holds $\Psi(x_n) = \sup_{t\in[0, 1]}\zeta\big(x_n(\lambda_n(t)) - \lambda_n(t)x_n(\lambda_n(1))\big)$, and therefore
\begin{equation*}
    |\Psi(x_n) - \Psi(x)| \le 2C_\zeta\sup_{t\in[0, 1]}\|x_n(\lambda_n(t)) - x(t)\|_B + C_\zeta\sup_{t\in[0, 1]}|\lambda_n(t) - t|\,\|x_n(1)\|_B \to 0.
\end{equation*}
Hence, $\Psi$ is continuous on $D([0, 1], B)$.

\textbf{Step 1 (Continuity of $G$):} For $x\in D([0, 1], B)$, let $\mathbf{J}(x) = \sup_{t\in(0, 1]}\|x(t) - x(t-)\|_B$ denote the largest jump of $x$. We first show that $\mathbf{J}(x)\le\liminf_{n\to\infty}\mathbf{J}(x_n)$, whenever $x_n\to x$ in $D([0, 1], B)$. With $\lambda_n$ as above, let $y_n = x_n\circ\lambda_n$. Since $\lambda_n$ is continuous and strictly increasing, $y_n$ has the same jump sizes as $x_n$, such that $\mathbf{J}(y_n) = \mathbf{J}(x_n)$. Moreover, $\|y_n(t-) - x(t-)\|_B\le\sup_{s\in[0, 1]}\|y_n(s) - x(s)\|_B$, and thus, for all $t\in(0, 1]$,
\begin{align*}
    \|x(t) - x(t-)\|_B 
    & \le \|y_n(t) - y_n(t-)\|_B + 2\sup_{s\in[0, 1]}\|y_n(s) - x(s)\|_B \\
    & \le \mathbf{J}(x_n) + 2\sup_{s\in[0, 1]}\|y_n(s) - x(s)\|_B.
\end{align*}
Letting $n\to\infty$ and taking the supremum over $t$ yields $\mathbf{J}(x)\le\liminf_{n\to\infty}\mathbf{J}(x_n)$. 
Consequently, for $\eps > 0$, the set $\{x: \mathbf{J}(x)\le\eps\}$ is closed, since it contains the limit of each of its convergent sequences, as $\mathbf{J}(x)\le\liminf_{n\to\infty}\mathbf{J}(x_n)\le\eps$, and since the Skorokhod topology on $D([0, 1], B)$ is metrizable. Hence, $\{x: \mathbf{J}(x) > \eps\}$ is open. Since $G_n\convw G$, by Theorem \ref{thm:fclt}, the portmanteau theorem yields $\pr(\mathbf{J}(G) > \eps)\le\liminf_{n\to\infty}\pr(\mathbf{J}(G_n) > \eps)$. The jumps of $G_n$ are $\eps_{i, n}/\sqrt{n}$, i.\,e., $\mathbf{J}(G_n) = n^{-1/2}\max_{1\le i\le n}\|\eps_{i, n}\|_B$, and Markov's inequality and Assumption \ref{assump:local_stationarity} yield
\begin{equation*}
    \pr(\mathbf{J}(G_n) > \eps) \le \sum_{i=1}^n \pr\big(\|\eps_{i, n}\|_B > \eps\sqrt{n}\big) \le \frac{\sup_{t\in[0, 1]}\ex[\|H(t, \Fc_0)\|_B^{2+\kappa}]}{\eps^{2+\kappa} n^{\kappa/2}} \to 0.
\end{equation*}
Hence, $\pr(\mathbf{J}(G) > \eps) = 0$, for all $\eps > 0$, i.\,e., $G$ has continuous sample paths almost surely, and the same holds for the independent copies $G^{(1)}, G^{(2)}, \dots$ from Theorem \ref{thm:bootstrap}.

\textbf{Step 2 (Continuity of the limiting distribution):} Let $Y(t) = G(t) - tG(1)$, $T = \Psi(G) = \sup_{t\in[0, 1]}\zeta(Y(t))$ and $F(x) = \pr(T\le x)$, for $x\in\R$. Let $\{x_j\}_{j\in\N}$ be dense in $B$. By the Hahn-Banach theorem, there exist $b_j\in B^*$ with $|b_j(x)|\le\zeta(x)$, for all $x\in B$, and $b_j(x_j) = \zeta(x_j)$. Since $\zeta(x) - |b_j(x)| \le 2\zeta(x - x_j)\le 2C_\zeta\|x - x_j\|_B$, it follows that $\zeta(x) = \sup_{j\in\N}|b_j(x)|$, for all $x\in B$. As $Y$ has continuous sample paths by Step 1,
\begin{equation*}
    T = \sup_{t\in[0, 1]\cap\mathbb{Q}} \sup_{j\in\N} \sup_{s\in\{-1, 1\}} s\, b_j(Y(t))
\end{equation*}
is the supremum of a separable, centered Gaussian process, which is finite almost surely. By Theorem 1 of \cite{tsirelson1976}, $F$ is continuous on $\R\setminus\{a_0\}$, where $a_0 = \inf\{x\in\R: F(x) > 0\}$.

We show that $a_0 = 0$, i.\,e., $F(x) > 0$, for all $x > 0$. The idea is that $Y$ lies close to some path $z$ with positive probability, and that the difference of two independent copies of $Y$ which both lie close to $z$ is close to $0$, while it has, after rescaling, the same distribution as $Y$. Let $x > 0$ and $\bar{\zeta}(y) = \sup_{t\in[0, 1]}\zeta(y(t))$, for $y\in C([0, 1], B)$, which defines a seminorm with $\bar{\zeta}(y)\le C_\zeta\sup_{t\in[0, 1]}\|y(t)\|_B$. Since $B$ is separable, $C([0, 1], B)$ is separable with respect to the supremum norm. Let $\{z_m\}_{m\in\N}$ be dense in $C([0, 1], B)$. Then, the (deterministic) sets $A_m = \{y\in C([0, 1], B): \bar{\zeta}(y - z_m)\le x/2\}$ cover $C([0, 1], B)$. As $Y$ takes values in $C([0, 1], B)$, by Step 1, it follows that $1\le\sum_{m\in\N}\pr(Y\in A_m)$, such that 
$p = \pr(\bar{\zeta}(Y - z)\le x/2) > 0$, for $z = z_m$ and some $m\in\N$.

Let $Y_1, Y_2$ be independent copies of $Y$. Since $Y$ is a centered Gaussian process, $(Y_1 - Y_2)/\sqrt{2}$ and $Y$ have the same finite-dimensional distributions and thus the same distribution, by Proposition \ref{prop:measure_equality}. By homogeneity of $\bar{\zeta}$, it follows that $F(x) = \pr(\bar{\zeta}(Y_1 - Y_2)\le\sqrt{2}x)$. Moreover, $\bar{\zeta}(Y_1 - Y_2)\le\bar{\zeta}(Y_1 - z) + \bar{\zeta}(Y_2 - z)$, by the triangle inequality. Hence, by independence,
\begin{equation*}
    F(x) \ge \pr\big(\bar{\zeta}(Y_1 - Y_2)\le x\big) \ge \pr\big(\bar{\zeta}(Y_1 - z)\le x/2\big)\,\pr\big(\bar{\zeta}(Y_2 - z)\le x/2\big) = p^2 > 0,
\end{equation*}
and hence $a_0 = 0$.

Finally, let $b\in B^*$ be as in part (1) of the corollary. Then, $T\ge|b(Y(1/2))|$, where $b(Y(1/2))$ is centered normally distributed with variance $\tfrac{1}{4}\int_0^1(b\otimes b)\sigma^2(u)\diff u$, by Theorem \ref{thm:fclt}, which is positive by assumption. Thus, $F(0) = \pr(T = 0) = 0$, and $F$ is continuous on $\R$, under the assumptions of part (1).

\textbf{Step 3 (Proof of part (1)):} Under $H_0$, it holds $\Lambda_n\equiv 0$, such that $\mathbf{C}_n = \Psi_n(G_n)$, by \eqref{eq:cusum_decomposition}. By Theorem \ref{thm:bootstrap}, $(G_n, \hat{G}_n^{(1)}, \hat{G}_n^{(2)})\convw(G, G^{(1)}, G^{(2)})$. Since $|\Psi_n(x) - \Psi(x)|\le C_\zeta\|x(1)\|_B/n$ and $\Psi$ is continuous on $D([0, 1], B)$, the continuous mapping theorem and Slutsky's lemma yield
\begin{equation*}
    \big(\mathbf{C}_n, \hat{\mathbf{C}}_n^{(1)}, \hat{\mathbf{C}}_n^{(2)}\big) \convw \big(T, T^{(1)}, T^{(2)}\big),
\end{equation*}
where $T^{(1)} = \Psi(G^{(1)})$ and $T^{(2)} = \Psi(G^{(2)})$ are independent copies of $T$, which are independent of $T$. The replicates $\hat{\mathbf{C}}_n^{(1)}, \hat{\mathbf{C}}_n^{(2)}, \dots$ are identically distributed and conditionally independent given the data, since they are based on the independent multipliers $\{R_i^{(k)}\}_{i\in\N}$, $k\in\N$, which are independent of the data. Hence, assertion (a) of Lemma 2.2 of \cite{bucher2019} holds, and, since $F$ is continuous by Step 2, Corollary 4.3 of \cite{bucher2019} implies that $\hat{p}_n$ converges weakly to the uniform distribution on $[0, 1]$, as $K_n\to\infty$.

\textbf{Step 4 (Proof of part (2)):} By \eqref{eq:cusum_decomposition} and the triangle inequality,
\begin{equation*}
    \mathbf{C}_n \ge \sqrt{n}\sup_{t\in[0, 1]}\zeta(\Lambda_n(t)) - 2C_\zeta\sup_{t\in[0, 1]}\|G_n(t)\|_B.
\end{equation*}
By Theorem \ref{thm:fclt} and the continuous mapping theorem, $\sup_{t\in[0, 1]}\|G_n(t)\|_B = \Oc_\pr(1)$. Moreover, by Assumption \ref{assump:bootstrap} (2), $\|\mu(i/n) - \mu(u)\|_B\le L_\mu/n$, for $u\in[(i-1)/n, i/n]$, unless the interval $((i-1)/n, i/n]$ contains one of the points $\theta_1, \dots, \theta_J$, which is the case for at most $J$ indices $i$. Since $\mu$ is bounded and $|\lfloor nt\rfloor/n - t|\le 1/n$, it follows that
\begin{equation*}
    \sup_{t\in[0, 1]}\|\Lambda_n(t) - \Lambda_\mu(t)\|_B \le \frac{2L_\mu + 2(2J+2)\|\mu\|_\infty}{n}.
\end{equation*}
Hence, $\sup_{t\in[0, 1]}\zeta(\Lambda_n(t)) \ge \sup_{t\in[0, 1]}\zeta(\Lambda_\mu(t)) - \Oc(n^{-1}) > 0$, for sufficiently large $n$, and $\mathbf{C}_n\to\infty$ in probability. On the other hand, Assumption \ref{assump:bootstrap} (2) allows for non-constant mean functions, such that Theorem \ref{thm:bootstrap} and the same arguments as in Step 3 yield $\hat{\mathbf{C}}_n^{(1)}\convw T$. In particular, $\hat{\mathbf{C}}_n^{(1)} = \Oc_\pr(1)$, and, since the replicates are identically distributed,
\begin{equation*}
    \ex[\hat{p}_n] = \pr\big(\hat{\mathbf{C}}_n^{(1)} > \mathbf{C}_n\big) \le \pr\big(\hat{\mathbf{C}}_n^{(1)} > M\big) + \pr(\mathbf{C}_n \le M) ,
\end{equation*}
for any $M > 0$. Letting first $n\to\infty$ and then $M\to\infty$ yields $\ex[\hat{p}_n]\to 0$, and the assertion follows by Markov's inequality. \hfill $\square$

%% file: app_proofs.tex
% !TeX spellcheck = en_US
%% ====================
\section{Proof of Example \ref{ex:continuous}} \label{app:proof_example}
%% ====================

First, $m$-dependence follows directly from the definition, and trivially implies the required bound on the physical dependence measure. 

\noindent \textbf{Moment bound.} Fix $\kappa > 0$. Since $\|s_{\ell, j}\|_\infty = 1$, for each $j, \ell$, 
\begin{equation*}
    \| X_{i, n}\|_\infty \le \sum_{\ell=0}^\infty \sum_{j=0}^{2^\ell - 1} 2^{-(\alpha + 1)\ell} |\xi_{i,n,\ell, j}| \|s_{\ell, j}\|_\infty \le  \sum_{\ell=0}^\infty \sum_{j=0}^{2^\ell - 1} 2^{-(\alpha + 1)\ell} =  \sum_{\ell=0}^\infty 2^{-\alpha \ell} = \frac{1}{1 - 2^{-\alpha}} < \infty.
\end{equation*}
In particular, $\sup_{i, n} \ex[\|X_{i, n}\|_\infty^{2+\kappa}] < \infty$. Moreover, the event $\{\|T_k X_{i, n}\|_\infty > \tfrac{1}{4} m(n, \sigma_k, N(\sigma_k/2, \Gc_k, V_{n,k}))\}$ is empty for sufficiently large $n$, since the left-hand side is bounded and the right-hand side is of order $\sqrt{n}$ by definition.

\noindent \textbf{Entropy condition.} By definition of $T_k$,
\begin{equation}\label{Eqn:ProjectionBound}
    \| T_k X_{i, n}\|_\infty \le \sum_{\ell =k+1}^\infty 2^{-\alpha\ell} = \frac{2^{-\alpha(k+1)}}{1 - 2^{-\alpha}} =: A_k,
\end{equation}
uniformly over $1\le i\le n$ and $n\in\N$. Let $f \in \Gc_k$ with $f(x, u) = (b\circ T_k)(x) \id(u\le t)$. Then, $|f(X_{i,n}, i/n)| \le \|T_k X_{i, n}\|_\infty$, and
\begin{equation*}
    \|f\|_{2, n} \le \bigg(\frac{1}{n}\sum_{i=1}^n \ex[\|T_k X_{i,n}\|_\infty^2]\bigg)^{1/2} \le A_k.
\end{equation*}
Since $\{(X_{i,n})_{1\le i\le n}\}_{n\in\N}$ is $m$-dependent, $V_{n,k}(f)\le (m+1)\|f\|_{2, n}$, and further
\begin{equation*}
    \sigma_k = \sup_{f\in\Gc_k} V_{n,k}(f) \le (m+1) \sup_{f\in\Gc_k} \|f\|_{2, n} \le (m+1) A_k \xrightarrow{k\to\infty}0
\end{equation*}

By definition, $X_{i, n}\in \Kc = \{ \sum_{\ell=0}^\infty \sum_{j=0}^{2^\ell - 1} 2^{-(\alpha + 1)\ell} \theta_{\ell, j} s_{\ell, j}: |\theta_{\ell, j}|\le 1\}$, and $\Kc$ is compact in $C([0, 1])$ by Arzelà-Ascoli. Since the norms $\|\cdot\|_{2,n}$ and $V_{n,k}$ are defined based on $f(X_{i, n}, i/n)$, it is sufficient to find brackets for $\Kc \times [0, 1]$.

Since the Faber-Schauder system is a Schauder basis, there exist continuous coordinate functionals $\{s_{\ell, j}^*\}\subset B^*$ such that every $x \in C([0, 1])$ has the unique expansion $x = \sum_{\ell, j} s_{\ell, j}^*(x) s_{\ell, j}$. In particular, for $x\in\Kc$ with $x = \sum_{\ell=0}^\infty \sum_{j=0}^{2^\ell - 1} 2^{-(\alpha + 1)\ell} \theta_{\ell, j} s_{\ell, j}$, with $|\theta_{\ell, j}|\le 1$, it holds $s_{\ell, j}^*(x) = 2^{-(\alpha + 1)\ell} \theta_{\ell, j}$.

Fix $\eps > 0$, and choose the cutoff level $L(\eps) = \max\{\lceil \alpha^{-1}\log_2(1/\eps)\rceil, k+1\}$, such that $\sum_{\ell=L(\eps)+1}^\infty 2^{-\alpha \ell} \le \eps/6$. For each level $\ell\in\{k+1, \dots, L(\eps)\}$, define the mesh size $\eta_\ell = 2^{-(1-\alpha)\ell}\eps/6$, and use a grid of mesh size $\eta_\ell$ for each coordinate $u_{\ell, j}$, $\ell\in\{k+1, \dots, L(\eps)\}, 0\le j< 2^\ell$. Let $U_{\Kc, \eps}$ denote the subset of $[-1, 1]^{\{(\ell, j): 0\le j < 2^\ell, k+1 \le \ell \le L(\eps)\}}$ consisting of the product of these meshes. 

For $t \in [0, 1]$ and $b\in B_1^*$, denote $f_{b,t}(x, u) = b(T_k x) \id(u \le t)$. For $t_0, t_1\in [0, 1]$, 
\begin{align} \label{eq:time_discretization}
    \| f_{b,t_0} - f_{b,t_1} \|_{2,n}^2
    & = \frac{1}{n} \sum_{i=1}^n \ex\big[\big(b(T_k X_{i, n})\big)^2\big(\id(i/n\le t_0) - \id(i/n\le t_1\big)^2\big] \\
    & \le A_k^2 \frac{\#\{i: i/n\in (t_0\wedge t_1, t_0\vee t_1]\}}{n}
    \le A_k^2 (|t_0 - t_1| + 1/n). \notag
\end{align}
For the time-discretization, let $\Tc_\eps$ denote a grid with mesh size $\eta^t := \min\{1, \eps^2/(72 A_k^2)\}$. For sufficiently large $n$, $n \ge 72 A_k^2 / \eps^2$, yielding a discretization error bounded by $\eps / 6$.

For each grid point $(\tilde{b}, \tilde{t})\in U_{\Kc, \eps} \times \Tc_\eps$, define $\tilde{f}_{\tilde{b}, \tilde{t}}: \Kc \times [0, 1] \to \R$ by
\begin{equation*}
    \tilde{f}_{\tilde{b}, \tilde{t}}(x, u) = \bigg(\sum_{\ell=k+1}^{L(\eps)} \sum_{j=0}^{2^\ell -1} s_{\ell, j}^*(x) \tilde{b}_{\ell, j} \bigg) \id(u \le \tilde{t}).
\end{equation*}
Based on $\tilde{f}_{\tilde{b}, \tilde{t}}$, define the bracket endpoints $l_{\tilde{b}, \tilde{t}} = \tilde{f}_{\tilde{b}, \tilde{t}} - \eps/2$ and $u_{\tilde{b}, \tilde{t}} = \tilde{f}_{\tilde{b}, \tilde{t}} + \eps/2$. The endpoints are measurable, but not necessarily in $\Gc_k$. 

Each pair $(b, t)\in B_1^*\times [0, 1]$ corresponds to a function $f_{b,t}\in \Gc_k$. Define the coefficients $u_{\ell, j}(b) = b(s_{n,j})$ with $|u_{\ell, j}(b)| = |b(s_{n,j})|\le \|b\|_{B^*} \|s_{n,j}\|_\infty\le 1$, for $\ell\in\{k+1, \dots, L(\eps)\}, 0\le j< 2^\ell$. Choose $\tilde{b} \in U_{\Kc,\eps}$ by rounding each coordinate $u_{\ell, j}(b)$ to the mesh, such that $|u_{\ell, j}(b) - \tilde{b}_{\ell, j}| \le \eta_\ell$, and $\tilde{t}$ such that $\id(u \le t)$ is bracketed by $\id(u \le t_0)$ up to error $\eps/6$ in $\|\cdot\|_{2, n}$. 

Then, for every $(x, u) \in \Kc\times [0, 1]$,
\begin{align*}
    |f(x, u) - \tilde{f}_{\tilde{b}, \tilde{t}}(x, u)|
    & = \bigg| \bigg(\sum_{\ell=k+1}^\infty \sum_{j=0}^{2^\ell - 1} 2^{-(\alpha + 1)\ell} \theta_{\ell, j} u_{\ell, j}(b) \bigg) \id(u \le t) - \bigg(\sum_{\ell=k+1}^{L(\eps)} \sum_{j=0}^{2^\ell -1} s_{\ell, j}^*(x) \tilde{b}_{\ell, j} \bigg) \id(u \le \tilde{t}) \bigg| \\
    & = \bigg| \bigg(\sum_{\ell=k+1}^\infty \sum_{j=0}^{2^\ell - 1} 2^{-(\alpha + 1)\ell} \theta_{\ell, j} u_{\ell, j}(b) \bigg) \id(u \le t) - \bigg(\sum_{\ell=k+1}^{L(\eps)} \sum_{j=0}^{2^\ell -1} 2^{-(\alpha + 1)\ell} \theta_{\ell, j} \tilde{b}_{\ell, j} \bigg) \id(u \le \tilde{t}) \bigg| \\
    & \le \frac{\eps}{6} + \bigg| \bigg(\sum_{\ell=k+1}^\infty \sum_{j=0}^{2^\ell - 1} 2^{-(\alpha + 1)\ell} \theta_{\ell, j} \big(u_{\ell, j}(b) -\tilde{b}_{\ell, j}\big)\bigg) \bigg| + C | \id(u \le t) - \id(u \le \tilde{t})| \\
    & \le \frac{\eps}{6} + \sum_{\ell=k+1}^\infty 2^{-\alpha\ell} \eta_\ell  + C | \id(u \le t) - \id(u \le \tilde{t})| \\
    & \le \frac{\eps}{3} + C | \id(u \le t) - \id(u \le \tilde{t})|
\end{align*}
where $|\theta_{\ell, j}|\le 1$ denote the coordinates of $x\in\Kc$. Thus, $\| f - \tilde{f}_{\tilde{b}, \tilde{t}}\|_{2, n} \le \eps/2$ by \eqref{eq:time_discretization}.

Hence, we can use the elements of $U_{\Kc, \eps} \times \Tc_\eps$ to construct brackets that cover $\Gc_k$. In particular, 
\begin{align*}
    \log\big(N(\eps, \Gc_k, V_{n,k})\big) 
    & \le \bigg(\sum_{\ell = k+1}^{L(\eps)} \sum_{j=0}^{2^\ell -1} \log(2/\eta_\ell)\bigg) + 2 \log(A_k/\eps) \\
    % = \bigg(\sum_{\ell = k+1}^{L(\eps)} 2^\ell \log(6 \cdot 2^{(1-\alpha)\ell + 1}/\eps)\bigg) + 2 \log(A_k/\eps)
    & = \bigg(\sum_{\ell = k+1}^{L(\eps)} 2^\ell \big([(1-\alpha)\ell + 1]\log(2) + \log(6/\eps)\bigg) + 2 \log(A_k/\eps)  \\
    & \le C 2^{L(\eps)} \log(1/\eps)
    \le C \eps^{-1/\alpha} \log(1/\eps),
\end{align*}
for some constant $C > 0$, such that the entropy integral exists and indeed converges to $0$ as $k\to\infty$.

%% ====================
\section{Proof of the Example from Section \ref{sec:pde_example}} \label{app:proof_example_pde}
%% ====================

The solution to the problem in \eqref{Eqn:WeakProblem1D} is given by
\begin{equation}\label{Eqn:PDESolu}
    \begin{aligned}
        u_{i,t}(x)=&\int_{[0,x]}\Phi\left(C_{i}-a(t)\,\sum_{k=0}^{\infty}\lambda^{k}\,2^{\frac{\ell_{k}\,(p-2)}{p}}\,\eta_{i-k}\,e_{k}^{\prime}(s)\right)\,\diff s,
        \\
        &\int_{[0,1]}\Phi\left(C_{i}-a(t)\,\sum_{k=0}^{\infty}\lambda^{k}\,2^{\frac{\ell_{k}\,(p-2)}{p}}\,\eta_{i-k}\,e_{k}^{\prime}(s)\right)\,\diff s=0,
    \end{aligned}
\end{equation}
where $\Phi(a)= \sgn(a)\,|a|^{\frac{1}{p-1}}=|a|^{\frac{2-p}{p-1}}\,a$ and $\ell_{k}=\max(\{\ell\in\N:2^{\ell}\leq k+1\})$.

In the following, we prove that this model satisfies Assumption \ref{assump:entropy}. The conditions in Assumption \ref{assump:local_stationarity} follow analogously to Example \ref{ex:continuous}.

Take $J_{\ell}=2^{\ell+1}-1$ and define $F_{i,n}^{[\ell]}=a(t)\,\sum_{k=0}^{J_{\ell}-1}\lambda^{k}\,\eta_{i-k}\,e_{k}^{\ast}$. Since $F_{i,n}^{[\ell]}$ is obtained from truncating $F_{i,n}$, we immedately obtain
\[
    \big\|F_{i,n}^{[\ell]}-F_{i,n}\big\|_{B^{\ast}}\leq\frac{M\,|\lambda|^{J_{\ell}}}{1-|\lambda|}.
\]
At this point we will use the properties of the model to obtain a similar bound for $\|X_{i,n}^{[\ell]}-X_{i,n}\|_{B}$, where 
$X_{i,n}^{[\ell]}$ is the solution obtained from $F_{i,n}^{[\ell]}$. Take $v_{1},v_{2}\in B$. Then, since $-\Delta_{p}$ is bounded, 
\begin{equation*}
    \left\langle\left(-\Delta_{p}\right)\,v_{1}-\left(-\Delta_{p}\right)\,v_{2},v_{1}-v_{2}\right\rangle_{B^{\ast}\times B}\leq\left\|\left(-\Delta_{p}\right)\,v_{1}-\left(-\Delta_{p}\right)\,v_{2}\right\|_{B^{\ast}}\,\|v_{1}-v_{2}\|_{B}.
\end{equation*}
Now we work on the left hand side to obtain a lower bound. By the mean value theorem
\[
    \begin{aligned}
        & \left\langle\left(-\Delta_{p}\right)\,v_{1}-\left(-\Delta_{p}\right)\,v_{2},v_{1}-v_{2}\right\rangle_{B^{\ast}\times B} \\
        &=\int_{[0,1]}\left(\left|v_{1}^{\prime}(x)\right|^{p-2}\,v_{1}^{\prime}(x)-\left|v_{2}^{\prime}(x)\right|^{p-2}\,v_{2}^{\prime}(x)\right)\,\left(v_{1}^{\prime}(x)-v_{2}^{\prime}(x)\right)\,\diff x
        \\
        &\geq(p-1)\,\int_{[0,1]}\left(\left|v_{1}^{\prime}(x)\right|+\left|v_{2}^{\prime}(x)\right|^{p-2}\,\right)\,\left(v_{1}^{\prime}(x)-v_{2}^{\prime}(x)\right)^{2}\,\diff x.
    \end{aligned}
\]
By H\"{o}lder's inequality with exponents $2/p$ and $2/(2-p)$,
\[
    \begin{aligned}
        &\left(\int_{[0,1]}\left|v_{1}^{\prime}(x)-v_{2}^{\prime}(x)\right|^{p}\,\diff x\right)^{\frac{2}{p}}
        \\
        &\quad\leq\left(\int_{[0,1]}\left(\left|v_{1}^{\prime}(x)\right|+\left|v_{1}^{\prime}(x)\right|^{p-2}\,\right)\,\left(v_{1}^{\prime}(x)-v_{2}^{\prime}(x)\right)^{2}\,\diff x\right)\,\left(\int_{[0,1]}\left(\left|v_{1}^{\prime}(x)\right|+\left|v_{2}^{\prime}(x)\right|^{p}\right)\,\diff x\right)^{\frac{2}{2-p}}.
    \end{aligned}
\]
Chaining both inequalities and using the inverse triangle inequality, we conclude that
\[
    \left\langle\left(-\Delta_{p}\right)\,v_{1}-\left(-\Delta_{p}\right)\,v_{2},v_{1}-v_{2}\right\rangle_{B^{\ast}\times B}\geq(p-1)\,\left\|v_{1}-v_{2}\right\|_{B}^{2}\,\left(\left\|v_{1}\right\|_{B}+\left\|v_{2}\right\|_{B}\right)^{p-2},
\]
So,
\[
    \frac{\left(\left\|v_{1}\right\|_{B}+\left\|v_{2}\right\|_{B}\right)^{2-p}}{(p-1)}\,\left\|\left(-\Delta_{p}\right)\,v_{1}-\left(-\Delta_{p}\right)\,v_{2}\right\|_{B^{\ast}}\geq\left\|v_{1}-v_{2}\right\|_{B}.
\]
By definition of $X_{i,n}^{[\ell]}$ and $X_{i,n}$,
\[
    \big\|X_{i,n}^{[\ell]}-X_{i,n}\big\|_{B}\leq\frac{\big(\big\|X_{i,n}^{[\ell]}\big\|_{B}+\|X_{i,n}\|_{B}\big)^{2-p}}{(p-1)}\,\big\|F_{i,n}^{[\ell]}-F_{i,n}\big\|_{B^{\ast}}
\]
From \eqref{Eqn:WeakProblem1D} and \eqref{Eqn:PDESolu}, differentiating $u_{i,t}$ we can prove $\|X_{i,n}^{[\ell]}\|_{B},\|X_{i,n}\|_{B}\leq(M/(1-|\lambda|))^{\frac{1}{p-1}}$. In other words, the time series and its truncation are uniformly bounded in $B$. Hence, there exists a constant $C>0$ such that
\[
    \big\|X_{i,n}^{[\ell]}-X_{i,n}\big\|_{B}\leq\frac{C\,M\,|\lambda|^{J_{\ell}}}{1-|\lambda|}.
\]
Observe now that, since $F_{i,n}^{[\ell]}\in span\left(\left\{e_{k}^{\prime}:k\leq J_{\ell}\right\}\right)$, the same holds for $X_{i,n}^{[\ell]}$ since the action of $-\Delta_{p}$ keeps this space invariant (see also \eqref{Eqn:PDESolu}). Consequently, $T_{\ell}\,X_{i,n}^{[\ell]}=0$. Furthermore, $T_{\ell}\,(X_{i,n}^{[\ell]}-X_{i,n})=T_{\ell}\,X_{i,n}$. Since $\left\{e_{k}\right\}_{k\in\N}$ is a Schauder basis, we have $\left\|T_{\ell}\,X_{i,n}\right\|_{B}\lesssim |\lambda|^{2^{\ell+1}-1}$. At this point, we have found a bound analogous to \eqref{Eqn:ProjectionBound} and Assumption \ref{assump:entropy} follows the same steps as in Example \ref{ex:continuous}.

%% ====================
\section{Proofs of Auxiliary Results from Section \ref{sec:aux_results}} \label{app:proofs}
%% ====================

\subsection{Proofs of the Results from Section \ref{sec:general_banach_results}}

\begin{proof}(of Proposition \ref{prop:measure_equality})
	Consider the set
	\begin{align*}
		\Fc = \{f: D([0, 1], B) \to \R | & f(x) = g\big(b_1(x(t_1)), \dots, b_p(x(t_p)),~g \in \Lc(\R^p, \R) \big), \\ 
		& b_1, \dots, b_p \in B^*, t_1, \dots, t_p \in [0, 1] , p\in\N \}.
	\end{align*}
	A partial order on $\Fc$ can be defined pointwise, where $f_1 \le f_2$ if and only if $f_1(x) \le f_2(x)$ for all $x \in D([0, 1], B)$. With this definition, $(\Fc, \le)$ becomes a vector lattice. Moreover, $\Fc$ contains the constant functions $f(x) = c$, for $c\in \R$. Finally, $\Fc$ separates points, that is, for any $x \neq y$, some $f\in \Fc$ exists with $f(x) \neq f(y)$. More specifically, for $x, y \in D([0, 1], B)$ with $x \neq y$, some $t\in [0, 1]$ exists so that $x(t) \neq y(t)$. By the Hahn-Banach theorem, some $b\in B^*$ exists, with $b(x(t)) \neq b(y(t))$, so that $f(x) \neq f(y)$ for $f(x):=b(x(t))$. 
	
	By Lemma 1.3.12 of \cite{vandervaart1996}, $X \stackrel{\Dc}{=} Y$, if and only if $\ex[f(X)] = \ex[f(Y)]$ for every $f \in \Fc$, which follows from \eqref{eq:prop_fdd} and the definition of $\Fc$.
\end{proof}

\begin{proof}(of Proposition \ref{prop:independence})
    Let $(\Omega, \Ac, \pr)$ denote the underlying probability space. Further, let $X_n$ and $Y_n$ be simple Bochner-measurable functions with $\lim_{n\to\infty}\|X - X_n\|_B \to 0$ and $\lim_{n\to\infty} \|Y - Y_n \|_B \to 0$, that is, $X_n = \sum_{i=1}^n x_i \id_{A_i}$ and $Y_n = \sum_{i=1}^n y_i \id_{B_i}$ for elements $x_1, \dots, x_n, y_1, \dots, y_n \in B$ and sets $A_1, \dots, A_n, B_1, \dots, B_n \in \Ac$. Then,
    \begin{equation}\label{eq:product_approximation}
        \ex[X_n \otimes Y_n] = \sum_{i, j=1}^n x_i \otimes y_i \pr(A_i \cap B_i) = \sum_{i=1}^n x_i \pr(A_i) \otimes \sum_{j=1}^n y_j \pr(B_j) = \ex[X_n] \otimes \ex[Y_n].
    \end{equation}
    By linearity and continuity of $(u, v)\mapsto u \otimes v$ with respect to the projective norm, it holds that $\ex[X_n \otimes Y_n] \to \ex[X \otimes Y], \ex[X_n] \to \ex[X]$ and $\ex[Y_n] \to \ex[Y]$. Hence, the assertion follows from \eqref{eq:product_approximation}.
\end{proof}

\begin{proof}(of Proposition \ref{prop:bound_proj_norm})
    Recall that the dual space of $B \hat{\otimes} B$ consists of all continuous, bilinear maps $B \times B \to \R$ \citep[see Theorem 4.6 of][]{janson2015}. Hence, the projective norm has the representation
    \begin{equation} \label{eq:proj_norm_representation}
        \| u \|_\pi = \sup\{ |\beta(u)|: \beta~\text{is continuous and bilinear on}~B\otimes B, \|\beta\|_{op}\le 1 \},
    \end{equation}
    for any $u\in B\hat{\otimes} B$. In particular
    \begin{equation}\label{eq:bound_projective_norm}
        \| \ex[X\otimes Y]\|_\pi 
        = \sup_{\|\beta\|_{op}\le 1} |\beta(\ex[X\otimes Y])| 
        = \sup_{\|\beta\|_{op}\le 1} |\ex[\beta(X\otimes Y)]| 
        \le \sup_{\|\beta\|_{op}\le 1} \ex[|\beta(X\otimes Y)|] 
    \end{equation}
    where the supremum is taken over all continuous and bilinear maps $\beta$. Since the operator norm is submultiplicative, it holds $|\beta(X\otimes Y)| \le \|\beta\|_{op} \|X\|_B \|Y\|_B$, so that the assertion follows from \eqref{eq:bound_projective_norm}.
\end{proof}

\begin{proof}(of Proposition \ref{prop:ottaviani})
    Let $\tilde{X}_1, \dots, \tilde{X}_n$ be an independent copies of $X_1, \dots, X_n$. The differences $Z_i = X_i - \tilde{X}_i$ are symmetric. In the following, we denote by $\ex_X$ and $\ex_{\tilde{X}}$ the expectations with respect to $X = (X_1, \dots, X_n)$ and $\tilde{X} = (\tilde{X}_1, \dots, \tilde{X}_n)$. By Proposition 1.10 in \cite{daprato2014},
    \begin{align*}
        \ex_X\bigg[ \max_{i=1}^n \bigg\| \sum_{j=1}^i X_j \bigg\|_B \bigg]
        & = \ex_X\bigg[ \max_{i=1}^n \bigg\| \sum_{j=1}^i X_j - \ex_{\tilde{X}}[\tilde{X}_j] \bigg\|_B \bigg] \\
        & \le \ex_X\bigg[ \max_{i=1}^n \ex_{\tilde{X}}\bigg[\bigg\| \sum_{j=1}^i X_j - \tilde{X}_j \bigg\|_B\bigg] \bigg] \\
        & \le \ex\bigg[ \max_{i=1}^n \bigg\| \sum_{j=1}^i X_j - \tilde{X}_j \bigg\|_B\bigg] \\
        & = \ex\bigg[ \max_{i=1}^n \bigg\| \sum_{j=1}^i Z_j \bigg\|_B\bigg].
    \end{align*}
    Further, by Levy's inequality \citep[see, e.\,g., Proposition 2.3 in][]{ledoux2013}, it holds
    \begin{align*}
        \ex\bigg[ \max_{i=1}^n \bigg\| \sum_{j=1}^i Z_j \bigg\|_B\bigg]
        & = \int_0^\infty \pr\bigg( \max_{i=1}^n \bigg\| \sum_{j=1}^i Z_j \bigg\|_B > t\bigg) \diff t \\
        & \le \int_0^\infty 2 \pr\bigg( \bigg\| \sum_{i=1}^n Z_i \bigg\|_B > t\bigg) \diff t 
        = 2 \ex\bigg[ \bigg\| \sum_{i=1}^n Z_i \bigg\|_B \bigg].
    \end{align*}
    Now, by the triangle inequality and Lemma 2.3.1 in \cite{vandervaart1996}, it finally holds
    \begin{equation*}
        2 \ex\bigg[ \bigg\| \sum_{i=1}^n Z_i \bigg\|_B \bigg]
        \le 2 \ex\bigg[ \bigg\| \sum_{i=1}^n X_i \bigg\|_B \bigg] + 2 \ex\bigg[ \bigg\| \sum_{i=1}^n \tilde{X}_i \bigg\|_B \bigg]
        = 4 \ex\bigg[ \bigg\| \sum_{i=1}^n X_i \bigg\|_B \bigg]
        \le 8 \ex\bigg[ \bigg\| \sum_{i=1}^n R_i X_i \bigg\|_B \bigg].
    \end{equation*}
\end{proof}

\begin{proof}(of Proposition \ref{prop:lipschitz_finite_rank})
    Let $L$ denote the Lipschitz constant of $x$ and fix $\eps >0$. Without loss of generality, assume that $L \ge 1$. Let $0 = t_0 < t_1 < \dots < t_m = 1$ be a partition with $\Delta = \max_{k=1}^m t_{k} - t_{k-1}$, such that $\Delta \le \frac{\eps}{2 L}$. For each $t_k$ select $x_k \in B\otimes B$, such that $\|x(t_k) - x_k\|_\pi \le \frac{\eps}{2}$. Define $\tilde{x}(t)$ by linear interpolation, i.\,e., define
    \begin{equation*}
        \tilde{x}(t) = \frac{t - t_{k-1}}{t_k - t_{k-1}} x_{k-1} + \frac{t_k - t}{t_k - t_{k-1}} x_k, 
    \end{equation*}
    for $k=1, \dots, m$. By the triangle inequality,
    \begin{align*}
        & \| x(t) - \tilde{x}(t) \|_\pi \\
        & = \bigg\| \frac{t - t_{k-1}}{t_k - t_{k-1}} [x(t) - x(t_{k-1}) + x(t_{k-1}) - x_{k-1}] + \frac{t_k - t}{t_k - t_{k-1}} [x(t) - x(t_k) + x(t_k) - x_k] \bigg\|_\pi \\
        & \le \frac{t - t_{k-1}}{t_k - t_{k-1}} [\| x(t) - x(t_{k-1}) \|_\pi + \|x(t_{k-1}) - x_{k-1}\|_\pi] \\ 
        & \quad +  \frac{t_k - t}{t_k - t_{k-1}} [\|x(t) - x(t_k)\|_\pi + \|x(t_k) - x_k\|_\pi].
    \end{align*}
    Due to Lipschitz continuity of $x$, and the bound $\|x(t_k) - x_k\|_\pi \le \frac{\eps}{2}$, the right-hand side can be bounded by $ L \Delta + \frac{\eps}{2} \le \eps$. Since this bound does not depend on $t$ or $k$, we have $\sup_{t \in [0, 1]} \| x(t) - \tilde{x}(t) \|_\pi \le \eps$. 

    It remains open to prove that $\tilde{x}(t)$ is Lipschitz continuous. First, observe that
    \begin{equation*}
        \| x_k - x_{k-1}\|_\pi 
        \le \| x_k - x(t_k) \|_\pi + \| x(t_k) - x(t_{k-1}) \|_\pi +  \| x(t_{k-1}) - x(t_{k-1}) \|_\pi
        \le \frac{\eps}{2} + L \Delta + \frac{\eps}{2} \le \frac{3}{2}\eps.
    \end{equation*}
    For any $s, t\in [t_{k-1}, t_k]$ with $k \in \{1, \dots, m\}$,
    \begin{align*}
        | \tilde{x}(t) - \tilde{x}(s) |
        \le \frac{\| x_k - x_{k-1}\|_\pi}{|t_k - t_{k-1}|} |t-s|
        \le \frac{3 \eps}{2 \min_{k=1, \dots, m}|t_k - t_{k-1}|} |t-s|.
    \end{align*}
    Hence, $\tilde{x}$ is Lipschitz continuous on each interval. By definition, $\tilde{x}(t)$ is piecewise linear, so that Lipschitz continuity extends to the whole interval $[0, 1]$, by gluing parts together.
\end{proof}

\subsection{Proof of the Maximal Inequality from Section \ref{sec:max_inequality}}

\begin{proof}
    The proof uses the same arguments as the proof of Theorem 4.4 of \cite{phandoidaen2022}. Note that their arguments are applied to $W_i(f) = f(X_{i, n}, i/n)$ and $W_i^{(k)}(f) = f(X_{i,n}^{(i-k)}, i/n)$, which is real-valued independent of the domain of $f$. We briefly summarize the main arguments.

    \textbf{Truncation.} Let $\phi_{\wedge M}$ denote the clipping function defined by $\phi_{\wedge M}(x) = (x \vee (-M))\wedge M$, and let $\Gc(M) = \{\phi_{\wedge M} \circ f: f \in \Gc\}$. Set $M_n = \tfrac{1}{4}m_\omega(n, \sigma, N(\sigma/2, \Gc, V_n))$, such that
    \begin{equation} \label{eq:truncation}
        \ex[\sup_{f\in\Gc}|\Gb_n(f)|] \le \ex[\sup_{f\in\Gc(M_n)}|\Gb_n(f)|] + \frac{1}{\sqrt{n}} \sum_{i=1}^n \ex[W_i(F \id(F > M_n)]
    \end{equation}
    by construction. 

    \textbf{Bracketing.} Let $\delta_j = 2^{-j} \sigma$ for $j\in\N_0$, choose bracketing covers at scale $\delta_j$, refine into nested partitions, select representatives $\pi_j f$, and define bracket widths $\Delta_j f$. Apply the same truncation-preserving decomposition as in (0.20) of the supplementary material of \cite{phandoidaen2022}, to obtain
    \begin{equation*}
        \sup_{f\in\Gc(M_n)}|\Gb_n(f)|  \le R_1 + R_2 + R_3 + R_4 + R_5,
    \end{equation*}
    where each of the summands $R_i$ is a sum of suprema over finite subclasses generated by $\pi_j f$, increments and truncated ``peaky parts''. The decomposition is purely real-valued and carries over to the Banach space setting.

    \textbf{Finite-class maxima.} Each $R_i$ is bounded via the ``compatibility lemma'' \citep[Lemma A.2 in][]{phandoidaen2022}, which in turn relies on the finite-class maximal inequality in Theorem 4.1 therein. The proof of the latter theorem uses only (i) a decomposition of the real-valued random variables $W_i(f)$ into martingale differences and approximately independent blocks, and (ii) dependence bounds of the form 
    \begin{equation*}
    	\ex[\sup_{f\in\Gc} | W_i(f) - W_i^{(k)}(f)|^2]^{1/2} \le |\omega(i/n)| D \Delta(k), \qquad 1\le i\le n,
    \end{equation*}
    which are satisfied by Assumption \ref{assump:functions}. In the notation of \cite{phandoidaen2022}, this corresponds to $D_{f, n} = \omega$, for all $f\in\Gc$, with $\Delta(k)$ replaced by $D\Delta(k)$. The weights enter their proofs only through quadratic means, namely through $(\frac{1}{n}\sum_{i=1}^n \omega(i/n)^2)^{1/2}$ in the martingale part (see (0.4) and (0.5) in their supplementary material), and through $(\frac{1}{n}\sum_{i=1}^n \omega(i/n)^2)^{1/2} D\Delta(j)$ in the bound for the independent blocks (see (0.7) therein), both of which are controlled by \eqref{eq:def_weight_constant}. Moreover, the clipping functions $\phi_{\wedge M}$ are $1$-Lipschitz, so that the dependence bound in (ii) is preserved under truncation \citep[cf. Lemma A.1 (iii) in][]{phandoidaen2022}. The proofs of Theorem 4.1 and Lemma A.2 of \cite{phandoidaen2022} carry over to $B$, since they do not rely on the geometry of $\R^d$ or its finite dimension. Therefore, the bounds on the $\ex[R_i]$ following (0.20) carry over to the Banach space setting and yield
    \begin{equation*}
        \sup_{f\in\Gc(M_n)}|\Gb_n(f)|  \le C \int_0^\sigma \sqrt{1 \vee \log(N(\eps, \Gc, V_n))} \diff \eps,
    \end{equation*}
    for some universal constant $C> 0$. Jointly with \eqref{eq:truncation}, this finishes the proof.
\end{proof}

\subsection{Proofs of the Results from Section \ref{sec:autocov_lrv}}

\begin{proof}(of Proposition \ref{prop:sn_dependence})
	By Proposition 2.6.31 of \cite{hytonen2016}, the conditional expectation commutes with $b$, so that $b(\ex[X | \Cc]) = \ex[b(X) | \Cc]$ for an integrable random variable $X$ and a $\sigma$-algebra $\Cc$. 
    By a telescoping sum argument, 
    \begin{align} \label{eq:telescope_sum_bound_new}
        \ex\big[\big( b(\eps_{i, n}) - b(\tilde{\eps}_{i, n}) \big)^2\big]^{1/2}
		& = \ex\big[\big( b(\eps_{i, n}) - \ex[b(\eps_{i, n}) | \eta_i, \dots, \eta_{i-m}]  \big)^2\big]^{1/2} \\
        & \le \sum_{k=m}^\infty \ex[ |\ex[b(\eps_{i, n})| \eta_i,\dots, \eta_{i-(k+1)}] - \ex[b(\eps_{i, n})| \eta_i,\dots, \eta_{i-k}]|^2]^{1/2}. \notag
    \end{align}
    Define $\eps_{i,n}^\circ$ by replacing $\eta_{i-(k+1)}$ with the independent copy $\eta_{i-(k+1)}^*$. By Jensen's inequality and submultiplicativity of the operator norm,
    \begin{align*}
        & \ex[ |\ex[b(\eps_{i, n})| \eta_i,\dots, \eta_{i-(k+1)}] - \ex[b(\eps_{i, n})| \eta_i,\dots, \eta_{i-k}]|^2]^{1/2} \\
        & = \ex[ |\ex[b(\eps_{i, n} - \eps_{i, n}^\circ)| \eta_i,\dots, \eta_{i-(k+1)}]|^2]^{1/2} \\
        & \le \ex[ \ex[(b(\eps_{i, n} - \eps_{i, n}^\circ))^2| \eta_i,\dots, \eta_{i-(k+1)}]]^{1/2} \\
        & = \ex[(b(\eps_{i, n} - \eps_{i, n}^\circ))^2]^{1/2}
        \le \|b\|_{B^*} \delta_2(H, k+1),
    \end{align*}
    where the last equality follows from the law of iterated expectation. Finally, \eqref{eq:telescope_sum_bound_new} yields $\ex\big[\big( b(\eps_{i, n}) - b(\tilde{\eps}_{i, n}) \big)^2\big]^{1/2}  \le \|b\|_{B^*} \Theta_{m+1}$.

    For the second part of the proposition, by a similar telescoping sum argument as before, 
    \begin{equation*}
        \ex[ \|\eps_{i, n} - \tilde{\eps}_{i, n}\|_B^2]^{1/2} \le \sum_{k=m}^\infty \ex[ \|\ex[\eps_{i, n}| \eta_i,\dots, \eta_{i-(k+1)}] - \ex[\eps_{i, n}| \eta_i,\dots, \eta_{i-k}]\|_B^2]^{1/2}.
    \end{equation*}
    Then, by Proposition 1.10 in \cite{daprato2014}, Jensen's inequality and the law of iterated expectation, 
    \begin{multline*}
        \ex\big[ \|\ex[\eps_{i, n}| \eta_i,\dots, \eta_{i-(k+1)}] - \ex[\eps_{i, n}| \eta_i,\dots, \eta_{i-k}]\|_B^2\big]^{1/2} \\
        % & = \ex\big[ \|\ex[\eps_{i, n}-\eps_{i,n}^\circ| \eta_i,\dots, \eta_{i-(k+1)}] \|_B^2\big]^{1/2} \\
        \le \ex\big[ \ex[\|\eps_{i, n}-\eps_{i,n}^\circ\|_B^2 | \eta_i,\dots, \eta_{i-(k+1)}]\big]^{1/2} 
        %& = \ex[ \|\eps_{i, n}-\eps_{i,n}^\circ\|_B^2]^{1/2}
        \le \delta_2(H, k+1),
    \end{multline*}
    such that 
    \begin{equation}\label{eq:telescope_sum_bound}
        \ex[ \|\eps_{i, n} - \tilde{\eps}_{i, n}\|_B^2]^{1/2} \le \sum_{k=m}^\infty \delta_2(H, k+1) = \Theta_{m+1},
    \end{equation}
    where all inequalities hold uniformly for $1\le i\le n$ and $n\in \N$.
\end{proof}

\begin{proof}(of Proposition \ref{prop:lrv_approximation})
    First, we rewrite the double sum in terms of a single sum as
    \begin{align*}
        \sum_{j_1, j_2=1}^{\ell_n} \ex\big[ b_1\big(H(t, \Fc_{0})\big) b_2\big(H(t, \Fc_{j_2-j_1})\big)\big]
        & = \sum_{j = - \ell_n}^{\ell_n} (\ell_n - |j|) \ex\big[ b_1\big(H(t, \Fc_0)\big) b_2\big(H(t, \Fc_j)\big)\big] \\
        & = \ell_n \sum_{j = - \ell_n}^{\ell_n} (1 - \frac{|j|}{\ell_n}) \ex\big[ b_1\big(H(t, \Fc_0)\big) b_2\big(H(t, \Fc_j)\big)\big]. 
    \end{align*} 
    The expectation commutes with bounded linear operators \citep[see, e.\,g., Remark 6.21 in][]{janson2015}, so that $\ex\big[ b_1\big(H(t, \Fc_0)\big) b_2\big(H(t, \Fc_j)\big)\big] = (b_1 \otimes b_2) \ex[ H(t, \Fc_0) H(t, \Fc_j)]$, and in particular
    \begin{equation*}
        (b_1 \otimes b_2) \sigma^2(t) = \sum_{j \in \Z} \ex\big[ b_1\big(H(t, \Fc_0)\big) b_2\big(H(t, \Fc_j)\big)\big],
    \end{equation*}
    since $H(t, \Fc_i)$ is centered, for $t\in [0, 1]$ and $i\in\Z$. By the triangle inequality, it further follows that 
    \begin{align*}
        & \sup_{t\in [0, 1]} \bigg|\frac{1}{\ell_n}\sum_{j_1, j_2=1}^{\ell_n} \ex\big[ b_1\big(H(t, \Fc_{0})\big) b_2\big(H(t, \Fc_{j_2-j_1})\big)\big] - (b_1 \otimes b_2) \sigma^2(t)\bigg| \\
        & \le \sum_{|j| \le \ell_n} \frac{|j|}{\ell_n} \sup_{t\in [0, 1]} \big| \ex\big[ b_1\big(H(t, \Fc_0)\big) b_2\big(H(t, \Fc_j)\big)\big] \big| + \sum_{|j| > \ell_n} \sup_{t\in [0, 1]} \big|\ex\big[ b_1\big(H(t, \Fc_0)\big) b_2\big(H(t, \Fc_j)\big)\big]\big|.
    \end{align*}
    If $\sup_{t\in[0, 1]} \ex\big[\big|b\big(H(t, \Fc_{0})\big)\big|^2\big]^{1/2} = 0$, for $b\in\{b_1, b_2\}$, the right-hand side equals $0$ by the Cauchy-Schwarz inequality, so the proposition follows trivially. Conversely, suppose that $\sup_{t\in[0, 1]} \ex\big[\big|b\big(H(t, \Fc_{0})\big)\big|^2\big]^{1/2} \neq 0$ for $b\in\{b_1, b_2\}$. In this case, it is sufficient to prove that the right-hand side vanishes.
    
    The first sum on the right-hand side converges to $0$, as $n\to\infty$, whenever
    \begin{equation} \label{eq:abs_summability}
        \sum_{j \in \Z} \sup_{t\in [0, 1]} \big|\ex\big[ b_1\big(H(t, \Fc_0)\big) b_2\big(H(t, \Fc_j)\big)\big]\big| < \infty
    \end{equation}
    by Lebesgue's dominated convergence theorem. For the second sum, convergence to $0$ is trivial, whenever \eqref{eq:abs_summability} holds. Similar to Proposition \ref{prop:sn_dependence}, we can replace $H(t, \Fc_j)$ by some version $H(t, \tilde{\Fc}_j)$ that is independent of $H(t, \Fc_0)$ at the cost $\Theta_{j}$. More specifically, let $\tilde{\Fc}_j = (\dots, \eta_{-1}^*, \eta_0^*, \eta_1, \dots, \eta_j)$, then
    \begin{align*}
        & \big|\ex\big[ b_1\big(H(t, \Fc_0)\big) b_2\big(H(t, \Fc_j)\big)\big]\big|\\
        & \le \big|\ex\big[ b_1\big(H(t, \Fc_0)\big) b_2\big(H(t, \tilde{\Fc}_j)\big)\big]\big| 
        + \ex\big[ b_1^2\big(H(t, \Fc_0)\big)\big]^{1/2} \ex\big[\big( b_2\big(H(t, \tilde{\Fc}_j)\big) - b_2\big(H(t, \Fc_j)\big)\big)^2\big]^{1/2}. 
    \end{align*}
    By independence, the first term is $0$, whereas the second term can be bounded uniformly from above by $C\Theta_j$, for some constant $C\ge 0$, by the same arguments yielding the uniform bound in Proposition \ref{prop:sn_dependence}. Therefore, 
    \begin{equation*}
        \sum_{j \in \Z} \sup_{t\in [0, 1]} \big|\ex\big[ b_1\big(H(t, \Fc_0)\big) b_2\big(H(t, \Fc_j)\big)\big]\big| \le C \sum_{j \in \Z} \Theta_{|j|},
    \end{equation*}
    which is finite by part 1 of Assumption \ref{assump:local_stationarity}.
\end{proof}

\begin{proof}(of Proposition \ref{prop:cov_bound})
    Recall the definition $\Fc_i = (\eta_k)_{k\le i}$, and define $\Fc_{k, i}= (\dots, \eta_{-k}, \eta_{-k+1}^*, \dots, \eta_0^*, \eta_1, \dots, \eta_i)$. With this notation, $\Fc_{0, i} = \Fc_i$. Similarly to \eqref{eq:telescope_sum_bound}, it holds
    \begin{align} \label{eq:telescope_sum_bound2}
        \ex[\|H(t, \Fc_i) - H(t, \Fc_{\infty, i})\|_B^2]^{1/2}
        & \le \sum_{k=1}^\infty \ex[\|H(t, \Fc_{k-1, i}) - H(t, \Fc_{k, i})\|_B^2]^{1/2} \\
        & \le \sum_{k=1}^\infty \delta(H, i+k-1) 
        \le \Theta_i. \notag
    \end{align}
    Now, by the triangle inequality,
    \begin{align*}
        & \| \cov\big(H(t, \Fc_i), H(t, \Fc_0)\big) \|_\pi \\
        & \le \| \ex[H(t, \Fc_{\infty, i}) \otimes H(t, \Fc_0)] \|_\pi + \| \ex[\{H(t, \Fc_i) - H(t, \Fc_{\infty, i})\} \otimes H(t, \Fc_0)] \|_\pi.
    \end{align*}
    By definition of the filtration, $H(t, \Fc_{\infty, i})$ and $H(t, \Fc_0)$ are independent, so that by Proposition \ref{prop:independence}, 
    \begin{equation*}
        \| \ex[H(t, \Fc_{\infty, i}) \otimes H(t, \Fc_0)] \|_\pi 
        = \| \ex[H(t, \Fc_{\infty, i})] \otimes \ex[H(t, \Fc_0)] \|_\pi
        = 0.
    \end{equation*}
    Moreover, by Proposition \ref{prop:bound_proj_norm} and the Cauchy-Schwarz inequality,
    \begin{equation*}
        \| \ex[\{H(t, \Fc_i) - H(t, \Fc_{\infty, i})\} \otimes H(t, \Fc_0)] \|_\pi
        \le \ex[\| H(t, \Fc_i) - H(t, \Fc_{\infty, i}) \|_B^2]^{1/2} \ex[\| H(t, \Fc_0) \|_B^2]^{1/2},
    \end{equation*}
    where the right-hand side can be bounded from above by $\ex[\| H(t, \Fc_0) \|_B^2]^{1/2} \Theta_i \le C \Theta_i$ by the moment bound in Assumption \ref{assump:local_stationarity} and \eqref{eq:telescope_sum_bound2}.
    Similarly, by independence, it holds 
    \begin{equation*}
        \ex\big[b\big(H(t, \Fc_{\infty, i})\big) b\big(H(t, \Fc_0)\big)\big]
        = \ex\big[b\big(H(t, \Fc_{\infty, i})\big)\big] \ex\big[b\big(H(t, \Fc_0)\big)\big]
        = 0,
    \end{equation*}
    and by the Cauchy-Schwarz inequality and \eqref{eq:telescope_sum_bound2},
    \begin{align*}
        & \ex\big[b\big(H(t, \Fc_i) - H(t, \Fc_{\infty, i})\big) b\big(H(t, \Fc_0)\big)\big] \\
        & \le \ex\big[\big|b\big(H(t, \Fc_i) - H(t, \Fc_{\infty, i})\big)\big|^2\big]^{1/2} \ex\big[\big|b\big(H(t, \Fc_0)\big)\big|^2\big]^{1/2} \\
        & \le \|b\|_{B^*}\ex\big[\big\|H(t, \Fc_i) - H(t, \Fc_{\infty, i})\big\|_B^2\big]^{1/2} \ex\big[\big|b\big(H(t, \Fc_0)\big)\big|^2\big]^{1/2} \\
        & \le \|b\|_{B^*} \Theta_i \ex\big[\big|b\big(H(t, \Fc_0)\big)\big|^2\big]^{1/2}.
    \end{align*}
\end{proof}

\begin{proof}(of Lemma \ref{lem:lrv_lipschitz})
    First, let $x_i(t) :=\cov\big( H(t, \Fc_i), H(t, \Fc_0) \big) \in B \hat{\otimes} B$ denote the autocovariance of $H(t, \Fc_i)$, for $t\in [0, 1]$ and $i \in\Z$. By Proposition \ref{prop:bound_proj_norm} and the Cauchy-Schwarz inequality, it holds
    \begin{align*}
        & \| x_i(t) - x_i(s) \|_\pi  \\
        & = \| \ex[ \big(H(t, \Fc_i) - H(s, \Fc_i)\big) \otimes H(t, \Fc_0) \big) - H(s, \Fc_i) \otimes \big( H(s, \Fc_0) - H(t, \Fc_0)\big) ] \|_\pi \\
        & \le \ex[ \| H(t, \Fc_i) - H(s, \Fc_i) \|_B^2]^{1/2} \ex[\| H(t, \Fc_0)\|_B^2]^{1/2} \\
        & \quad + \ex[\| H(s, \Fc_i)\|_B^2]^{1/2} \ex[ \| H(s, \Fc_0) - H(t, \Fc_0) \|_B^2 ]^{1/2},
    \end{align*}
    for any $s, t\in [0, 1]$ and $i \in \Z$. Since $H$ is Lipschitz continuous with respect to the $L^2$-norm, expectations of differences can be bounded from above by $C|t-s|$, for some constant $C\ge 0$. Since the $2+\kappa$ moments of $H(t, \Fc_0)$ are uniformly bounded by assumption, the right-hand side of the previous display can be bounded by $C |t-s|$, for an adequate constant $C$ that does not depend on $i$. Hence, the autocovariances $x_i(t)$ are Lipschitz continuous in $t$.

    Let $(f_{i, n}(t))_{n\in\N}\subset B\otimes B$ be a sequence of finite rank tensors, such that each $f_{i, n}(t)$ is Lipschitz continuous in $t$ with $\sup_{t\in[0, 1]} \|x_i(t) - f_{i, n}(t)\|_\pi \le \frac{\min(1, \Theta_i)}{2^n}$. Such a sequence exists by Proposition \ref{prop:lipschitz_finite_rank}. Then, we may write $x_i(t) = f_{i, 1}(t) + \sum_{n=1}^\infty f_{i, n+1}(t) - f_{i, n}(t)$, where each difference $f_{i, n+1}(t) - f_{i, n}(t)$ is Lipschitz continuous in $t$. 
    
    Let $(\delta_{i, n})_{i \in \Z, n\in\N}$ be a family of positive numbers, such that $\sum_{i\in\Z, n\in \N} \delta_{i, n} < C$, for some constant $C\ge 0$. Since each $f_{i, n}$ is Lipschitz continuous and finite-rank, Lipschitz continuous maps $\bar{y}_{i, n, k}, \bar{z}_{i, n, k}:[0, 1]\to B$ exist, for $k = 1, \dots, K, i \in\Z$ and  $n, K\in \N$, such that $f_{i, n}(t) = \sum_{k=1}^K \bar{y}_{i, n, k}(t) \otimes \bar{z}_{i, n, k}(t)$ with $\|f_{i, n}\|_\pi \ge \sum_{k=1}^K \|\bar{y}_{i, n, k}(t)\|_B \|\bar{z}_{i, n, k}(t)\|_B - \delta_{i, n}$. We may enumerate these sequences and associate $(\bar{y}_{i, n, k})_{i, n, k}$ with $(\tilde{y}_m)_{m\in\N}$, and likewise $(\bar{z}_{i, n, k})_{i, n, k}$ with $(\tilde{z}_m)_{m\in\N}$. Then, 
    \begin{equation} \label{eq:lrv_lipschitz_decomposition2}
        \sigma^2(t) = \sum_{i\in \Z} x_i(t) = \sum_{m\in\N} \tilde{y}_m(t) \otimes \tilde{z}_m(t).
    \end{equation}
    Further, it holds $\sup_{t\in[0, 1]} \|f_{i, 1}(t)\|_\pi \le \sup_{t\in[0, 1]} \|x_i(t)\|_\pi + \frac{\min(1, \Theta_i)}{2} \le C \Theta_i + \frac{\min(1, \Theta_i)}{2}$, by Proposition \ref{prop:cov_bound},  and $\sup_{t\in[0, 1]} \|f_{i, n+1}(t) - f_{i, n}(t)\|_\pi \le \frac{\min(1, \Theta_i)}{2^{n+1}} + \frac{\min(1, \Theta_i)}{2^n} \le \frac{\min(1, \Theta_i)}{2^{n-1}}$ by construction. Therefore 
    \begin{align}
        & \sup_{t\in[0, 1]} \sum_{m=1}^\infty \| \tilde{y}_m(t)\|_B \|\tilde{z}_m(t)\|_B \notag \\
        & \le \sum_{i\in\Z} \bigg( \sup_{t\in[0, 1]} \|f_{i, 1}(t)\|_\pi + \sum_{n=1}^\infty \sup_{t\in[0, 1]} \|f_{i, n+1}(t) - f_{i, n}(t)\|_\pi\bigg) + \sum_{i\in\Z, n\in\N} \delta_{i, n} \label{eq:l1_summability}\\
        & \le C \sum_{i\in\Z} \Theta_i + \sum_{i\in\Z} \sum_{n=1}^\infty \frac{\min(1, \Theta_i)}{2^{n-1}} + C < \infty. \notag
    \end{align}
    In particular, $\sigma^2(t)$ is continuous.
    We now construct sequences $(y_m)_{m\in\N}$ and $(z_m)_{m\in\N}$, satisfying 
    \begin{equation} \label{eq:l2_summable_coeffs}
        y_m(t) \otimes z_m(t) = \tilde{y}_m(t) \otimes \tilde{z}_m(t),
    \end{equation} 
    and \eqref{eq:lrv_l2_summability}. Then, \eqref{eq:lrv_lipschitz_decomposition} follows from \eqref{eq:lrv_lipschitz_decomposition2}. 

    Let $a_m = \sup_{t\in[0, 1]}\|\tilde{y}_m(t)\|_B  \|\tilde{z}_m(t)\|_B, A_m = \sup_{t\in[0, 1]} \|\tilde{y}_m(t)\|_B, B_m = \sup_{t\in[0, 1]} \|\tilde{z}_m(t)\|_B$, and $\delta_m = \frac{a_m}{2^m (1 + A_m + B_m)}$. With these definitions, $\phi_m(t) := \sqrt{\frac{ \|\tilde{z}_m(t)\|_B + \delta_m}{ \|\tilde{y}_m(t)\|_B+\delta_m}}$ is Lipschitz continuous. Define $y_m(t) := \phi_m(t) \tilde{y}_m(t)$ and $z_m(t) := \phi_m^{-1}(t) \tilde{z}_m(t)$. Per construction, $y_m$ and $z_m$ are Lipschitz and satisfy \eqref{eq:l2_summable_coeffs}, and hence \eqref{eq:lrv_lipschitz_decomposition}. Moreover,
    \begin{align*}
        \| y_m(t) \|_B^2 
        & = \|\tilde{y}_m(t)\|_B^2 \frac{ \|\tilde{z}_m(t)\|_B + \delta_m}{ \|\tilde{y}_m(t)\|_B+\delta_m} \\
        & \le \|\tilde{y}_m(t)\|_B \|\tilde{z}_m(t)\|_B + \|\tilde{y}_m(t)\|_B \delta_m \\
        & \le a_m + A_m \delta_m 
        \le (1+2^{-m}) a_m.
    \end{align*}
    By \eqref{eq:l1_summability} and definition of $a_m$, it follows that $\sum_{m\in\N} \| y_m(t) \|_B^2 < \infty$, and by the same arguments, $\sum_{m\in\N} \| z_m(t) \|_B^2 < \infty$.
\end{proof}

%% file: app_empirical_results.tex
% !TeX spellcheck = en_US
%% ====================
\section{Additional Empirical Results} \label{app:empirical_results}
%% ====================

This section contains the complete results of the simulation study from Section \ref{sec:simulation_study}. Table \ref{tab:level_additional} displays the empirical rejection rates under the null hypothesis for the error processes (BB), (FAR-BB) and (tvFAR2), which are not included in Table \ref{tab:level}, based on $3000$ replications each. Tables \ref{tab:global_sn_b0}-\ref{tab:global_bs_l2} contain the empirical rejection rates for the abrupt global change, Tables \ref{tab:local_sn_b0}-\ref{tab:local_bs_l2} for the abrupt local change, and Tables \ref{tab:gradual_sn_b0}-\ref{tab:gradual_bs_l2} for the gradual change, each based on $1000$ replications. In particular, the columns for $\delta = 0$ contain the rejection rates under the null hypothesis separately for the three mean functions, which are pooled in Tables \ref{tab:level} and \ref{tab:level_additional}. For the global and the gradual change, the projected mean $\ex[b_0(X_{i, n})]$ is constant, such that the null hypothesis of the test SN-$b_0$ holds for all $\delta$.

{
\setlength{\tabcolsep}{3pt}
\begin{table}[ht]
    \centering
    \caption{Empirical rejection rates (in \%) under the null hypothesis ($\delta = 0$), based on $3000$ replications each.}
    \label{tab:level_additional}
    \footnotesize
    \begin{tabular}{l|rrrr|rrrr|rrrr}
    \toprule
    $n$ & \multicolumn{4}{c|}{$100$} & \multicolumn{4}{c|}{$200$} & \multicolumn{4}{c}{$500$} \\
    $\eps$ & SN-$b_0$ & SN-$b_1$ & BS-sup & BS-$L^2$ & SN-$b_0$ & SN-$b_1$ & BS-sup & BS-$L^2$ & SN-$b_0$ & SN-$b_1$ & BS-sup & BS-$L^2$ \\
    \midrule
     (BB) & 4.9 & 5.1 & 5.9 & 7.5 & 5.1 & 5.3 & 6.6 & 6.5 & 4.9 & 4.7 & 5.4 & 5.7 \\
     (FAR-BB) & 6.4 & 4.6 & 8.8 & 12.2  & 5.4 & 5.3 & 8.9 & 9.9 & 4.7 & 5.2 & 7.9 & 8.5 \\
     (tvFAR2) & 6.5 & 6.2 & 14.1 & 15.1  & 6.8 & 6.1 & 12.4 & 12.8 & 5.4 & 5.5 & 9.3 & 9.8 \\
    \bottomrule
    \end{tabular}
\end{table}
}

\clearpage

\begin{table}[ht]
    \centering
    \small
    \setlength{\tabcolsep}{4pt}
    \caption{Empirical rejection rates (in \%) of the test SN-$b_0$ for an abrupt global change, for different sample sizes $n$, error processes $\eps$ and heights $\delta$.}
    \label{tab:global_sn_b0}
    \begin{tabular}{l|rrrrrrrrrrr}
    \toprule
    $\delta$ & 0.0 & 0.1 & 0.2 & 0.3 & 0.4 & 0.5 & 0.6 & 0.7 & 0.8 & 0.9 & 1.0 \\
    \midrule
    \multicolumn{12}{l}{\textit{Panel A: $n = 100$}}\\    
    (BM) & 5.6 & 4.8 & 4.3 & 5.4 & 5.2 & 3.4 & 5.6 & 4.9 & 4.6 & 4.2 & 4.3 \\
    (BB) & 4.6 & 5.2 & 4.0 & 4.1 & 3.7 & 5.7 & 5.4 & 5.8 & 5.5 & 5.0 & 5.5 \\
    (FAR-BM) & 6.6 & 7.0 & 6.2 & 6.0 & 6.9 & 5.6 & 7.3 & 4.8 & 6.2 & 7.9 & 5.5 \\
    (FAR-BB) & 6.1 & 5.3 & 7.3 & 5.8 & 5.0 & 6.4 & 6.2 & 6.1 & 6.4 & 5.8 & 5.8 \\
    (tvBM) & 6.1 & 8.0 & 6.0 & 6.0 & 7.6 & 6.7 & 6.1 & 7.1 & 5.3 & 6.2 & 7.4 \\
    (tvFAR1) & 9.4 & 7.1 & 9.1 & 6.3 & 8.6 & 7.3 & 8.4 & 7.4 & 8.9 & 9.1 & 7.9 \\
    (tvFAR2) & 6.7 & 7.1 & 6.3 & 5.5 & 6.5 & 7.8 & 6.8 & 6.7 & 6.9 & 8.1 & 6.8 \\
    \midrule
    \multicolumn{12}{l}{\textit{Panel B: $n = 200$}}\\    
    (BM) & 5.5 & 5.3 & 5.4 & 6.1 & 5.2 & 4.3 & 7.7 & 4.7 & 4.4 & 4.9 & 5.2 \\
    (BB) & 5.1 & 5.5 & 4.9 & 3.9 & 5.0 & 6.7 & 4.6 & 4.4 & 4.7 & 5.1 & 5.3 \\
    (FAR-BM) & 5.8 & 4.9 & 5.5 & 4.8 & 5.4 & 5.4 & 4.8 & 5.6 & 5.4 & 6.0 & 5.4 \\
    (FAR-BB) & 5.7 & 5.0 & 4.8 & 6.7 & 5.8 & 4.7 & 6.2 & 5.4 & 5.4 & 5.4 & 5.8 \\
    (tvBM) & 6.0 & 7.3 & 7.5 & 8.1 & 6.8 & 6.1 & 7.4 & 7.0 & 6.5 & 8.2 & 6.5 \\
    (tvFAR1) & 7.6 & 7.9 & 8.1 & 6.1 & 8.1 & 8.4 & 7.3 & 7.2 & 6.6 & 7.7 & 6.5 \\
    (tvFAR2) & 7.6 & 5.6 & 7.3 & 6.3 & 7.3 & 6.6 & 7.7 & 5.7 & 5.4 & 5.9 & 6.4 \\
    \midrule
    \multicolumn{12}{l}{\textit{Panel C: $n = 500$}}\\    
    (BM) & 4.7 & 4.9 & 5.6 & 4.7 & 5.2 & 4.8 & 3.6 & 4.8 & 5.2 & 5.6 & 4.6 \\
    (BB) & 4.2 & 6.6 & 5.6 & 5.6 & 3.6 & 6.6 & 5.3 & 5.5 & 5.3 & 4.5 & 4.7 \\
    (FAR-BM) & 5.0 & 5.5 & 6.7 & 5.2 & 5.3 & 5.7 & 6.0 & 4.7 & 5.4 & 5.4 & 5.6 \\
    (FAR-BB) & 5.2 & 6.3 & 6.0 & 4.2 & 5.1 & 5.0 & 5.4 & 6.2 & 5.8 & 5.7 & 5.8 \\
    (tvBM) & 6.9 & 7.3 & 7.8 & 7.9 & 5.1 & 6.6 & 9.1 & 7.3 & 7.2 & 8.7 & 7.9 \\
    (tvFAR1) & 6.8 & 7.2 & 6.7 & 7.4 & 7.7 & 7.3 & 7.3 & 7.2 & 7.1 & 5.9 & 6.8 \\
    (tvFAR2) & 5.7 & 6.3 & 6.2 & 6.1 & 5.6 & 5.3 & 5.7 & 5.2 & 6.1 & 4.7 & 5.2 \\
    \bottomrule
    \end{tabular}
\end{table}

\begin{table}[ht]
    \centering
    \small
    \setlength{\tabcolsep}{4pt}
    \caption{Empirical rejection rates (in \%) of the test SN-$b_1$ for an abrupt global change, for different sample sizes $n$, error processes $\eps$ and heights $\delta$.}
    \label{tab:global_sn_b1}
    \begin{tabular}{l|rrrrrrrrrrr}
    \toprule
    $\delta$ & 0.0 & 0.1 & 0.2 & 0.3 & 0.4 & 0.5 & 0.6 & 0.7 & 0.8 & 0.9 & 1.0 \\
    \midrule
    \multicolumn{12}{l}{\textit{Panel A: $n = 100$}}\\    
    (BM) & 3.8 & 27.1 & 66.5 & 94.5 & 99.2 & 100.0 & 100.0 & 100.0 & 100.0 & 100.0 & 100.0 \\
    (BB) & 3.9 & 71.4 & 99.2 & 100.0 & 100.0 & 100.0 & 100.0 & 100.0 & 100.0 & 100.0 & 100.0 \\
    (FAR-BM) & 7.0 & 20.0 & 52.2 & 83.8 & 95.5 & 98.7 & 100.0 & 100.0 & 100.0 & 100.0 & 100.0 \\
    (FAR-BB) & 5.0 & 64.9 & 98.2 & 100.0 & 100.0 & 100.0 & 100.0 & 100.0 & 100.0 & 100.0 & 100.0 \\
    (tvBM) & 7.2 & 27.0 & 66.4 & 92.0 & 99.0 & 99.8 & 100.0 & 100.0 & 100.0 & 100.0 & 100.0 \\
    (tvFAR1) & 8.3 & 20.4 & 50.7 & 78.2 & 92.7 & 98.1 & 99.7 & 100.0 & 99.9 & 100.0 & 100.0 \\
    (tvFAR2) & 6.6 & 20.1 & 50.7 & 79.9 & 94.1 & 99.1 & 99.6 & 100.0 & 100.0 & 100.0 & 100.0 \\
    \midrule
    \multicolumn{12}{l}{\textit{Panel B: $n = 200$}}\\    
    (BM) & 3.9 & 43.2 & 93.2 & 99.8 & 100.0 & 100.0 & 100.0 & 100.0 & 100.0 & 100.0 & 100.0 \\
    (BB) & 6.1 & 92.0 & 100.0 & 100.0 & 100.0 & 100.0 & 100.0 & 100.0 & 100.0 & 100.0 & 100.0 \\
    (FAR-BM) & 5.9 & 30.8 & 77.6 & 97.8 & 99.9 & 100.0 & 100.0 & 100.0 & 100.0 & 100.0 & 100.0 \\
    (FAR-BB) & 4.4 & 89.2 & 100.0 & 100.0 & 100.0 & 100.0 & 100.0 & 100.0 & 100.0 & 100.0 & 100.0 \\
    (tvBM) & 8.7 & 44.9 & 88.3 & 99.0 & 100.0 & 100.0 & 100.0 & 100.0 & 100.0 & 100.0 & 100.0 \\
    (tvFAR1) & 7.8 & 29.4 & 75.2 & 94.9 & 99.4 & 99.9 & 100.0 & 100.0 & 100.0 & 100.0 & 100.0 \\
    (tvFAR2) & 5.9 & 31.5 & 76.6 & 95.6 & 99.6 & 100.0 & 100.0 & 100.0 & 100.0 & 100.0 & 100.0 \\
    \midrule
    \multicolumn{12}{l}{\textit{Panel C: $n = 500$}}\\    
    (BM) & 4.8 & 78.7 & 99.8 & 100.0 & 100.0 & 100.0 & 100.0 & 100.0 & 100.0 & 100.0 & 100.0 \\
    (BB) & 5.5 & 99.8 & 100.0 & 100.0 & 100.0 & 100.0 & 100.0 & 100.0 & 100.0 & 100.0 & 100.0 \\
    (FAR-BM) & 5.4 & 60.4 & 97.5 & 99.8 & 100.0 & 100.0 & 100.0 & 100.0 & 100.0 & 100.0 & 100.0 \\
    (FAR-BB) & 4.6 & 99.9 & 100.0 & 100.0 & 100.0 & 100.0 & 100.0 & 100.0 & 100.0 & 100.0 & 100.0 \\
    (tvBM) & 6.8 & 77.2 & 99.1 & 100.0 & 100.0 & 100.0 & 100.0 & 100.0 & 100.0 & 100.0 & 100.0 \\
    (tvFAR1) & 7.5 & 55.5 & 95.6 & 100.0 & 100.0 & 100.0 & 100.0 & 100.0 & 100.0 & 100.0 & 100.0 \\
    (tvFAR2) & 5.8 & 56.2 & 97.5 & 100.0 & 100.0 & 100.0 & 100.0 & 100.0 & 100.0 & 100.0 & 100.0 \\
    \bottomrule
    \end{tabular}
\end{table}

\begin{table}[ht]
    \centering
    \small
    \setlength{\tabcolsep}{4pt}
    \caption{Empirical rejection rates (in \%) of the test BS-sup for an abrupt global change, for different sample sizes $n$, error processes $\eps$ and heights $\delta$.}
    \label{tab:global_bs_sup}
    \begin{tabular}{l|rrrrrrrrrrr}
    \toprule
    $\delta$ & 0.0 & 0.1 & 0.2 & 0.3 & 0.4 & 0.5 & 0.6 & 0.7 & 0.8 & 0.9 & 1.0 \\
    \midrule
    \multicolumn{12}{l}{\textit{Panel A: $n = 100$}}\\    
    (BM) & 6.5 & 10.3 & 21.5 & 47.6 & 76.1 & 94.0 & 99.3 & 100.0 & 100.0 & 100.0 & 100.0 \\
    (BB) & 5.8 & 25.1 & 73.5 & 98.9 & 100.0 & 100.0 & 100.0 & 100.0 & 100.0 & 100.0 & 100.0 \\
    (FAR-BM) & 14.2 & 13.5 & 21.7 & 36.2 & 56.8 & 80.5 & 93.1 & 99.1 & 100.0 & 100.0 & 100.0 \\
    (FAR-BB) & 8.5 & 23.8 & 64.4 & 96.2 & 99.8 & 100.0 & 100.0 & 100.0 & 100.0 & 100.0 & 100.0 \\
    (tvBM) & 9.3 & 10.3 & 20.3 & 36.6 & 60.5 & 83.0 & 96.5 & 99.1 & 99.9 & 100.0 & 99.9 \\
    (tvFAR1) & 14.4 & 12.8 & 20.2 & 31.8 & 49.1 & 66.6 & 85.3 & 93.5 & 97.3 & 99.1 & 99.9 \\
    (tvFAR2) & 15.5 & 14.8 & 23.5 & 38.5 & 56.9 & 77.5 & 93.0 & 98.6 & 100.0 & 100.0 & 100.0 \\
    \midrule
    \multicolumn{12}{l}{\textit{Panel B: $n = 200$}}\\    
    (BM) & 7.5 & 12.2 & 39.2 & 83.8 & 99.2 & 100.0 & 100.0 & 100.0 & 100.0 & 100.0 & 100.0 \\
    (BB) & 5.9 & 44.1 & 98.9 & 100.0 & 100.0 & 100.0 & 100.0 & 100.0 & 100.0 & 100.0 & 100.0 \\
    (FAR-BM) & 9.7 & 15.2 & 31.9 & 60.3 & 90.4 & 99.7 & 100.0 & 100.0 & 100.0 & 100.0 & 100.0 \\
    (FAR-BB) & 11.0 & 37.9 & 95.4 & 100.0 & 100.0 & 100.0 & 100.0 & 100.0 & 100.0 & 100.0 & 100.0 \\
    (tvBM) & 7.1 & 11.4 & 29.5 & 67.8 & 95.7 & 99.9 & 100.0 & 100.0 & 100.0 & 100.0 & 100.0 \\
    (tvFAR1) & 13.1 & 13.7 & 29.0 & 45.4 & 78.9 & 95.0 & 99.1 & 100.0 & 100.0 & 100.0 & 100.0 \\
    (tvFAR2) & 13.6 & 17.1 & 32.2 & 59.6 & 87.8 & 99.2 & 100.0 & 100.0 & 100.0 & 100.0 & 100.0 \\
    \midrule
    \multicolumn{12}{l}{\textit{Panel C: $n = 500$}}\\    
    (BM) & 6.1 & 24.1 & 87.7 & 100.0 & 100.0 & 100.0 & 100.0 & 100.0 & 100.0 & 100.0 & 100.0 \\
    (BB) & 4.9 & 88.3 & 100.0 & 100.0 & 100.0 & 100.0 & 100.0 & 100.0 & 100.0 & 100.0 & 100.0 \\
    (FAR-BM) & 9.2 & 20.3 & 62.3 & 99.3 & 100.0 & 100.0 & 100.0 & 100.0 & 100.0 & 100.0 & 100.0 \\
    (FAR-BB) & 7.3 & 79.8 & 100.0 & 100.0 & 100.0 & 100.0 & 100.0 & 100.0 & 100.0 & 100.0 & 100.0 \\
    (tvBM) & 5.2 & 17.0 & 73.9 & 99.7 & 100.0 & 100.0 & 100.0 & 100.0 & 100.0 & 100.0 & 100.0 \\
    (tvFAR1) & 8.2 & 14.9 & 47.8 & 92.2 & 100.0 & 100.0 & 100.0 & 100.0 & 100.0 & 100.0 & 100.0 \\
    (tvFAR2) & 10.7 & 19.9 & 58.9 & 97.9 & 100.0 & 100.0 & 100.0 & 100.0 & 100.0 & 100.0 & 100.0 \\
    \bottomrule
    \end{tabular}
\end{table}

\begin{table}[ht]
    \centering
    \small
    \setlength{\tabcolsep}{4pt}
    \caption{Empirical rejection rates (in \%) of the test BS-$L^2$ for an abrupt global change, for different sample sizes $n$, error processes $\eps$ and heights $\delta$.}
    \label{tab:global_bs_l2}
    \begin{tabular}{l|rrrrrrrrrrr}
    \toprule
    $\delta$ & 0.0 & 0.1 & 0.2 & 0.3 & 0.4 & 0.5 & 0.6 & 0.7 & 0.8 & 0.9 & 1.0 \\
    \midrule
    \multicolumn{12}{l}{\textit{Panel A: $n = 100$}}\\    
    (BM) & 8.4 & 9.8 & 21.6 & 52.4 & 90.0 & 99.4 & 100.0 & 100.0 & 100.0 & 100.0 & 100.0 \\
    (BB) & 8.0 & 23.9 & 87.5 & 100.0 & 100.0 & 100.0 & 100.0 & 100.0 & 100.0 & 100.0 & 100.0 \\
    (FAR-BM) & 14.2 & 13.7 & 21.8 & 32.1 & 61.9 & 87.8 & 98.8 & 100.0 & 100.0 & 100.0 & 100.0 \\
    (FAR-BB) & 12.0 & 18.4 & 72.6 & 99.4 & 100.0 & 100.0 & 100.0 & 100.0 & 100.0 & 100.0 & 100.0 \\
    (tvBM) & 10.2 & 10.1 & 19.6 & 41.5 & 75.8 & 96.1 & 99.7 & 100.0 & 100.0 & 100.0 & 100.0 \\
    (tvFAR1) & 15.3 & 14.2 & 20.8 & 30.8 & 53.2 & 77.8 & 94.9 & 98.6 & 99.9 & 100.0 & 100.0 \\
    (tvFAR2) & 15.8 & 15.1 & 21.0 & 35.2 & 58.6 & 87.5 & 98.3 & 99.8 & 100.0 & 100.0 & 100.0 \\
    \midrule
    \multicolumn{12}{l}{\textit{Panel B: $n = 200$}}\\    
    (BM) & 7.5 & 10.2 & 40.0 & 94.3 & 100.0 & 100.0 & 100.0 & 100.0 & 100.0 & 100.0 & 100.0 \\
    (BB) & 6.4 & 44.7 & 100.0 & 100.0 & 100.0 & 100.0 & 100.0 & 100.0 & 100.0 & 100.0 & 100.0 \\
    (FAR-BM) & 11.0 & 14.0 & 28.0 & 63.1 & 98.4 & 100.0 & 100.0 & 100.0 & 100.0 & 100.0 & 100.0 \\
    (FAR-BB) & 11.3 & 30.2 & 99.2 & 100.0 & 100.0 & 100.0 & 100.0 & 100.0 & 100.0 & 100.0 & 100.0 \\
    (tvBM) & 6.0 & 9.3 & 30.8 & 84.5 & 99.4 & 100.0 & 100.0 & 100.0 & 100.0 & 100.0 & 100.0 \\
    (tvFAR1) & 14.0 & 12.9 & 26.8 & 49.4 & 91.1 & 99.7 & 99.8 & 100.0 & 100.0 & 100.0 & 100.0 \\
    (tvFAR2) & 13.5 & 15.7 & 28.2 & 65.2 & 96.5 & 100.0 & 100.0 & 100.0 & 100.0 & 100.0 & 100.0 \\
    \midrule
    \multicolumn{12}{l}{\textit{Panel C: $n = 500$}}\\    
    (BM) & 5.8 & 20.6 & 98.5 & 100.0 & 100.0 & 100.0 & 100.0 & 100.0 & 100.0 & 100.0 & 100.0 \\
    (BB) & 5.6 & 96.6 & 100.0 & 100.0 & 100.0 & 100.0 & 100.0 & 100.0 & 100.0 & 100.0 & 100.0 \\
    (FAR-BM) & 9.3 & 16.9 & 67.4 & 100.0 & 100.0 & 100.0 & 100.0 & 100.0 & 100.0 & 100.0 & 100.0 \\
    (FAR-BB) & 7.9 & 82.5 & 100.0 & 100.0 & 100.0 & 100.0 & 100.0 & 100.0 & 100.0 & 100.0 & 100.0 \\
    (tvBM) & 6.6 & 15.5 & 90.3 & 100.0 & 100.0 & 100.0 & 100.0 & 100.0 & 100.0 & 100.0 & 100.0 \\
    (tvFAR1) & 8.5 & 13.5 & 49.5 & 98.8 & 100.0 & 100.0 & 100.0 & 100.0 & 100.0 & 100.0 & 100.0 \\
    (tvFAR2) & 11.2 & 17.6 & 61.7 & 100.0 & 100.0 & 100.0 & 100.0 & 100.0 & 100.0 & 100.0 & 100.0 \\
    \bottomrule
    \end{tabular}
\end{table}

\clearpage

\begin{table}[ht]
    \centering
    \small
    \setlength{\tabcolsep}{4pt}
    \caption{Empirical rejection rates (in \%) of the test SN-$b_0$ for an abrupt local change, for different sample sizes $n$, error processes $\eps$ and heights $\delta$.}
    \label{tab:local_sn_b0}
    \begin{tabular}{l|rrrrrrrrrrr}
    \toprule
    $\delta$ & 0.0 & 0.1 & 0.2 & 0.3 & 0.4 & 0.5 & 0.6 & 0.7 & 0.8 & 0.9 & 1.0 \\
    \midrule
    \multicolumn{12}{l}{\textit{Panel A: $n = 100$}}\\    
    (BM) & 4.3 & 4.6 & 4.3 & 5.4 & 3.9 & 4.8 & 5.7 & 6.6 & 4.5 & 4.8 & 6.7 \\
    (BB) & 4.9 & 4.3 & 5.9 & 4.7 & 5.1 & 6.3 & 6.0 & 7.1 & 8.8 & 9.9 & 10.6 \\
    (FAR-BM) & 5.5 & 5.7 & 7.3 & 6.2 & 7.0 & 5.4 & 7.1 & 6.6 & 6.1 & 5.7 & 6.7 \\
    (FAR-BB) & 6.1 & 6.7 & 5.5 & 7.4 & 7.4 & 6.5 & 7.3 & 7.2 & 8.7 & 8.9 & 8.8 \\
    (tvBM) & 7.4 & 8.0 & 7.7 & 6.7 & 7.0 & 8.6 & 6.7 & 8.9 & 8.2 & 9.1 & 7.9 \\
    (tvFAR1) & 9.1 & 8.3 & 8.8 & 8.3 & 8.2 & 9.1 & 9.0 & 8.9 & 9.0 & 7.6 & 9.3 \\
    (tvFAR2) & 4.9 & 5.8 & 6.5 & 7.4 & 7.9 & 6.6 & 8.0 & 7.0 & 6.8 & 8.8 & 7.6 \\
    \midrule
    \multicolumn{12}{l}{\textit{Panel B: $n = 200$}}\\    
    (BM) & 5.7 & 5.7 & 6.2 & 5.3 & 6.4 & 6.7 & 5.9 & 5.9 & 6.7 & 8.1 & 8.0 \\
    (BB) & 5.2 & 6.6 & 4.7 & 6.9 & 7.3 & 6.3 & 6.9 & 8.8 & 11.8 & 13.5 & 16.3 \\
    (FAR-BM) & 6.4 & 4.9 & 4.5 & 6.6 & 5.8 & 5.0 & 5.4 & 6.8 & 6.9 & 6.7 & 7.6 \\
    (FAR-BB) & 5.0 & 6.4 & 7.1 & 6.2 & 5.2 & 6.9 & 7.9 & 8.6 & 9.1 & 8.5 & 11.3 \\
    (tvBM) & 6.7 & 7.6 & 6.3 & 7.4 & 9.6 & 5.9 & 7.6 & 6.3 & 7.9 & 10.0 & 9.2 \\
    (tvFAR1) & 8.8 & 8.3 & 7.3 & 7.5 & 8.4 & 8.9 & 8.1 & 8.1 & 7.6 & 9.5 & 8.3 \\
    (tvFAR2) & 6.0 & 5.5 & 6.1 & 6.3 & 7.4 & 5.2 & 8.4 & 7.1 & 8.5 & 7.2 & 6.4 \\
    \midrule
    \multicolumn{12}{l}{\textit{Panel C: $n = 500$}}\\    
    (BM) & 5.2 & 5.0 & 5.7 & 5.9 & 5.6 & 7.8 & 8.4 & 8.4 & 8.5 & 10.7 & 11.7 \\
    (BB) & 4.9 & 5.6 & 7.3 & 7.6 & 9.9 & 12.0 & 14.9 & 18.5 & 24.3 & 26.7 & 31.2 \\
    (FAR-BM) & 6.4 & 5.8 & 5.7 & 6.7 & 5.8 & 6.7 & 7.0 & 6.9 & 8.0 & 7.9 & 9.0 \\
    (FAR-BB) & 3.7 & 6.2 & 5.5 & 7.3 & 6.7 & 10.5 & 11.3 & 13.0 & 17.6 & 17.5 & 20.0 \\
    (tvBM) & 6.9 & 7.4 & 6.5 & 7.5 & 7.8 & 8.9 & 7.9 & 10.0 & 11.2 & 13.4 & 12.7 \\
    (tvFAR1) & 6.5 & 7.4 & 6.3 & 8.6 & 6.9 & 7.5 & 7.0 & 9.1 & 9.7 & 10.5 & 11.3 \\
    (tvFAR2) & 5.3 & 6.7 & 6.0 & 6.5 & 8.5 & 6.9 & 6.7 & 7.6 & 8.2 & 7.9 & 8.9 \\
    \bottomrule
    \end{tabular}
\end{table}

\begin{table}[ht]
    \centering
    \small
    \setlength{\tabcolsep}{4pt}
    \caption{Empirical rejection rates (in \%) of the test SN-$b_1$ for an abrupt local change, for different sample sizes $n$, error processes $\eps$ and heights $\delta$.}
    \label{tab:local_sn_b1}
    \begin{tabular}{l|rrrrrrrrrrr}
    \toprule
    $\delta$ & 0.0 & 0.1 & 0.2 & 0.3 & 0.4 & 0.5 & 0.6 & 0.7 & 0.8 & 0.9 & 1.0 \\
    \midrule
    \multicolumn{12}{l}{\textit{Panel A: $n = 100$}}\\    
    (BM) & 3.7 & 5.0 & 4.7 & 4.8 & 6.4 & 6.6 & 6.3 & 7.6 & 8.4 & 8.2 & 11.6 \\
    (BB) & 5.2 & 5.2 & 6.6 & 6.3 & 6.9 & 10.0 & 11.1 & 16.1 & 16.3 & 21.1 & 25.7 \\
    (FAR-BM) & 5.1 & 6.5 & 6.2 & 5.6 & 7.3 & 6.2 & 7.1 & 7.9 & 9.5 & 8.7 & 10.1 \\
    (FAR-BB) & 4.3 & 3.8 & 6.8 & 7.0 & 10.3 & 8.7 & 13.5 & 16.3 & 17.6 & 20.4 & 22.4 \\
    (tvBM) & 6.2 & 6.3 & 7.7 & 6.5 & 7.7 & 8.2 & 8.5 & 8.5 & 9.4 & 12.3 & 12.5 \\
    (tvFAR1) & 7.5 & 7.3 & 9.7 & 6.6 & 9.3 & 10.3 & 10.0 & 10.8 & 10.3 & 11.4 & 10.5 \\
    (tvFAR2) & 4.9 & 6.6 & 6.7 & 7.1 & 8.1 & 6.4 & 8.5 & 7.6 & 8.2 & 9.4 & 11.6 \\
    \midrule
    \multicolumn{12}{l}{\textit{Panel B: $n = 200$}}\\    
    (BM) & 5.5 & 6.8 & 4.8 & 6.5 & 6.9 & 6.5 & 8.7 & 9.2 & 12.0 & 16.1 & 16.6 \\
    (BB) & 4.7 & 5.0 & 6.2 & 9.4 & 11.9 & 17.8 & 21.5 & 25.6 & 32.8 & 41.1 & 44.4 \\
    (FAR-BM) & 5.6 & 4.6 & 5.1 & 5.6 & 5.2 & 6.8 & 6.5 & 8.8 & 10.9 & 12.6 & 12.6 \\
    (FAR-BB) & 5.2 & 6.6 & 6.5 & 8.4 & 8.4 & 14.8 & 19.8 & 22.7 & 27.6 & 33.9 & 41.0 \\
    (tvBM) & 7.1 & 7.2 & 7.7 & 8.2 & 10.3 & 8.3 & 13.6 & 14.3 & 13.9 & 14.1 & 17.6 \\
    (tvFAR1) & 8.2 & 9.0 & 8.0 & 8.0 & 10.4 & 9.2 & 10.6 & 12.1 & 10.2 & 13.8 & 13.7 \\
    (tvFAR2) & 6.2 & 5.6 & 7.0 & 6.7 & 5.9 & 6.2 & 8.2 & 8.6 & 9.7 & 10.2 & 12.7 \\
    \midrule
    \multicolumn{12}{l}{\textit{Panel C: $n = 500$}}\\    
    (BM) & 4.8 & 5.3 & 6.4 & 9.1 & 9.3 & 16.5 & 14.9 & 16.5 & 22.5 & 29.2 & 31.0 \\
    (BB) & 4.9 & 5.9 & 10.4 & 16.2 & 24.9 & 35.0 & 45.0 & 53.8 & 66.0 & 75.3 & 83.1 \\
    (FAR-BM) & 5.7 & 5.5 & 5.8 & 6.6 & 7.1 & 10.7 & 11.0 & 14.3 & 15.0 & 18.5 & 21.6 \\
    (FAR-BB) & 4.9 & 4.8 & 9.0 & 13.8 & 20.5 & 28.0 & 39.0 & 49.8 & 55.1 & 65.6 & 75.1 \\
    (tvBM) & 7.7 & 8.0 & 9.3 & 10.6 & 12.5 & 13.6 & 18.4 & 21.8 & 26.4 & 29.1 & 35.1 \\
    (tvFAR1) & 7.6 & 6.8 & 7.6 & 9.2 & 10.3 & 11.0 & 13.8 & 15.8 & 19.0 & 19.5 & 23.6 \\
    (tvFAR2) & 5.6 & 6.2 & 7.0 & 6.2 & 9.1 & 10.2 & 10.6 & 15.6 & 14.9 & 16.8 & 24.5 \\
    \bottomrule
    \end{tabular}
\end{table}

\begin{table}[ht]
    \centering
    \small
    \setlength{\tabcolsep}{4pt}
    \caption{Empirical rejection rates (in \%) of the test BS-sup for an abrupt local change, for different sample sizes $n$, error processes $\eps$ and heights $\delta$.}
    \label{tab:local_bs_sup}
    \begin{tabular}{l|rrrrrrrrrrr}
    \toprule
    $\delta$ & 0.0 & 0.1 & 0.2 & 0.3 & 0.4 & 0.5 & 0.6 & 0.7 & 0.8 & 0.9 & 1.0 \\
    \midrule
    \multicolumn{12}{l}{\textit{Panel A: $n = 100$}}\\    
    (BM) & 5.7 & 6.4 & 6.3 & 9.3 & 17.1 & 37.0 & 64.8 & 88.1 & 97.2 & 99.7 & 100.0 \\
    (BB) & 5.6 & 8.1 & 18.4 & 48.5 & 81.7 & 97.1 & 99.8 & 100.0 & 100.0 & 100.0 & 100.0 \\
    (FAR-BM) & 10.7 & 11.2 & 12.5 & 12.3 & 15.4 & 25.0 & 41.5 & 61.6 & 80.3 & 92.2 & 98.6 \\
    (FAR-BB) & 8.5 & 10.7 & 16.9 & 41.6 & 74.7 & 93.1 & 98.8 & 100.0 & 100.0 & 100.0 & 100.0 \\
    (tvBM) & 8.6 & 9.8 & 8.1 & 7.7 & 15.0 & 25.5 & 48.1 & 73.5 & 88.6 & 97.0 & 99.3 \\
    (tvFAR1) & 13.3 & 12.5 & 14.8 & 13.4 & 15.5 & 17.9 & 31.4 & 49.2 & 66.5 & 81.9 & 93.1 \\
    (tvFAR2) & 11.9 & 12.4 & 12.8 & 13.7 & 15.9 & 25.2 & 39.9 & 60.7 & 79.3 & 92.2 & 98.1 \\
    \midrule
    \multicolumn{12}{l}{\textit{Panel B: $n = 200$}}\\    
    (BM) & 8.6 & 6.2 & 6.9 & 18.4 & 56.6 & 89.1 & 99.2 & 100.0 & 100.0 & 100.0 & 100.0 \\
    (BB) & 6.6 & 10.6 & 44.9 & 90.0 & 99.7 & 100.0 & 100.0 & 100.0 & 100.0 & 100.0 & 100.0 \\
    (FAR-BM) & 10.5 & 10.1 & 10.5 & 14.1 & 27.2 & 60.9 & 89.2 & 98.7 & 99.7 & 100.0 & 100.0 \\
    (FAR-BB) & 7.3 & 12.2 & 38.4 & 78.3 & 99.0 & 100.0 & 100.0 & 100.0 & 100.0 & 100.0 & 100.0 \\
    (tvBM) & 6.7 & 6.6 & 6.5 & 12.7 & 37.9 & 75.0 & 95.8 & 99.2 & 100.0 & 100.0 & 100.0 \\
    (tvFAR1) & 12.0 & 13.9 & 12.1 & 10.5 & 22.2 & 42.2 & 71.5 & 91.5 & 98.3 & 99.6 & 100.0 \\
    (tvFAR2) & 11.7 & 12.6 & 11.8 & 13.0 & 28.8 & 59.1 & 87.0 & 98.2 & 99.9 & 100.0 & 100.0 \\
    \midrule
    \multicolumn{12}{l}{\textit{Panel C: $n = 500$}}\\    
    (BM) & 6.8 & 5.0 & 21.9 & 82.9 & 99.6 & 100.0 & 100.0 & 100.0 & 100.0 & 100.0 & 100.0 \\
    (BB) & 5.5 & 23.7 & 94.7 & 100.0 & 100.0 & 100.0 & 100.0 & 100.0 & 100.0 & 100.0 & 100.0 \\
    (FAR-BM) & 9.2 & 9.6 & 10.3 & 44.8 & 92.2 & 100.0 & 100.0 & 100.0 & 100.0 & 100.0 & 100.0 \\
    (FAR-BB) & 7.7 & 18.2 & 84.2 & 100.0 & 100.0 & 100.0 & 100.0 & 100.0 & 100.0 & 100.0 & 100.0 \\
    (tvBM) & 7.3 & 8.3 & 13.5 & 68.2 & 98.1 & 100.0 & 100.0 & 100.0 & 100.0 & 100.0 & 100.0 \\
    (tvFAR1) & 8.8 & 8.6 & 8.8 & 28.3 & 77.0 & 98.2 & 100.0 & 100.0 & 100.0 & 100.0 & 100.0 \\
    (tvFAR2) & 8.2 & 10.6 & 13.1 & 39.2 & 88.1 & 99.8 & 100.0 & 100.0 & 100.0 & 100.0 & 100.0 \\
    \bottomrule
    \end{tabular}
\end{table}

\begin{table}[ht]
    \centering
    \small
    \setlength{\tabcolsep}{4pt}
    \caption{Empirical rejection rates (in \%) of the test BS-$L^2$ for an abrupt local change, for different sample sizes $n$, error processes $\eps$ and heights $\delta$.}
    \label{tab:local_bs_l2}
    \begin{tabular}{l|rrrrrrrrrrr}
    \toprule
    $\delta$ & 0.0 & 0.1 & 0.2 & 0.3 & 0.4 & 0.5 & 0.6 & 0.7 & 0.8 & 0.9 & 1.0 \\
    \midrule
    \multicolumn{12}{l}{\textit{Panel A: $n = 100$}}\\    
    (BM) & 5.9 & 7.7 & 7.7 & 8.7 & 7.3 & 10.2 & 11.0 & 13.1 & 10.6 & 12.6 & 15.7 \\
    (BB) & 6.8 & 7.3 & 8.1 & 8.6 & 12.0 & 16.7 & 23.5 & 38.6 & 49.3 & 66.8 & 85.5 \\
    (FAR-BM) & 11.9 & 12.8 & 13.5 & 12.9 & 14.0 & 13.6 & 16.4 & 15.2 & 16.0 & 15.9 & 17.6 \\
    (FAR-BB) & 12.7 & 11.4 & 11.9 & 13.2 & 15.4 & 18.9 & 20.1 & 28.6 & 34.6 & 47.0 & 61.1 \\
    (tvBM) & 10.2 & 11.4 & 9.5 & 8.9 & 11.3 & 11.6 & 9.8 & 12.1 & 12.6 & 13.4 & 14.0 \\
    (tvFAR1) & 15.4 & 14.7 & 15.6 & 14.4 & 16.1 & 14.9 & 16.6 & 18.1 & 17.0 & 17.1 & 18.4 \\
    (tvFAR2) & 13.0 & 13.9 & 14.7 & 15.4 & 14.8 & 17.5 & 16.0 & 17.1 & 18.3 & 18.9 & 22.0 \\
    \midrule
    \multicolumn{12}{l}{\textit{Panel B: $n = 200$}}\\    
    (BM) & 8.8 & 7.0 & 7.7 & 7.2 & 8.4 & 9.1 & 12.4 & 14.2 & 14.7 & 23.2 & 27.7 \\
    (BB) & 6.4 & 7.7 & 9.6 & 13.0 & 19.8 & 32.3 & 57.4 & 82.7 & 96.4 & 99.7 & 99.9 \\
    (FAR-BM) & 11.1 & 10.3 & 11.0 & 13.0 & 12.8 & 12.7 & 15.6 & 16.0 & 17.4 & 20.1 & 22.4 \\
    (FAR-BB) & 9.0 & 11.4 & 12.3 & 13.7 & 17.8 & 26.3 & 37.7 & 54.1 & 77.8 & 92.1 & 99.2 \\
    (tvBM) & 6.4 & 7.7 & 7.1 & 8.3 & 8.6 & 8.6 & 10.5 & 10.1 & 14.9 & 15.9 & 19.2 \\
    (tvFAR1) & 14.0 & 13.6 & 12.4 & 11.3 & 12.8 & 13.4 & 15.1 & 15.1 & 15.4 & 18.7 & 18.6 \\
    (tvFAR2) & 11.5 & 13.0 & 13.5 & 13.6 & 14.8 & 13.3 & 15.3 & 16.7 & 19.0 & 21.6 & 23.2 \\
    \midrule
    \multicolumn{12}{l}{\textit{Panel C: $n = 500$}}\\    
    (BM) & 5.8 & 6.3 & 7.2 & 9.2 & 9.4 & 16.9 & 22.7 & 36.7 & 64.4 & 90.6 & 99.2 \\
    (BB) & 5.8 & 7.4 & 12.5 & 24.8 & 64.4 & 98.4 & 100.0 & 100.0 & 100.0 & 100.0 & 100.0 \\
    (FAR-BM) & 10.5 & 10.2 & 9.0 & 10.7 & 12.9 & 15.4 & 16.5 & 21.8 & 27.1 & 37.7 & 57.0 \\
    (FAR-BB) & 8.3 & 10.0 & 12.0 & 18.0 & 34.0 & 70.2 & 97.4 & 99.9 & 100.0 & 100.0 & 100.0 \\
    (tvBM) & 7.5 & 8.0 & 8.0 & 7.8 & 9.6 & 10.3 & 15.6 & 23.5 & 39.2 & 69.5 & 90.4 \\
    (tvFAR1) & 9.4 & 10.9 & 9.6 & 11.6 & 9.7 & 14.6 & 13.7 & 17.7 & 19.6 & 29.6 & 37.4 \\
    (tvFAR2) & 9.5 & 11.9 & 12.7 & 11.1 & 14.7 & 15.8 & 17.0 & 22.6 & 27.5 & 35.8 & 53.3 \\
    \bottomrule
    \end{tabular}
\end{table}

\begin{table}[ht]
    \centering
    \small
    \setlength{\tabcolsep}{4pt}
    \caption{Empirical rejection rates (in \%) of the test SN-$b_0$ for a gradual change, for different sample sizes $n$, error processes $\eps$ and heights $\delta$.}
    \label{tab:gradual_sn_b0}
    \begin{tabular}{l|rrrrrrrrrrr}
    \toprule
    $\delta$ & 0.0 & 0.1 & 0.2 & 0.3 & 0.4 & 0.5 & 0.6 & 0.7 & 0.8 & 0.9 & 1.0 \\
    \midrule
    \multicolumn{12}{l}{\textit{Panel A: $n = 100$}}\\    
    (BM) & 5.1 & 6.0 & 5.8 & 4.8 & 5.7 & 5.5 & 4.8 & 5.1 & 5.1 & 4.4 & 5.3 \\
    (BB) & 5.2 & 4.9 & 4.8 & 5.0 & 6.4 & 3.9 & 5.3 & 5.8 & 5.3 & 5.9 & 5.0 \\
    (FAR-BM) & 6.1 & 5.5 & 6.2 & 5.0 & 5.5 & 5.3 & 5.8 & 5.1 & 6.1 & 5.6 & 6.0 \\
    (FAR-BB) & 6.9 & 6.2 & 6.2 & 6.0 & 5.0 & 6.8 & 5.1 & 5.8 & 5.4 & 6.4 & 5.5 \\
    (tvBM) & 6.5 & 7.2 & 6.2 & 7.4 & 5.4 & 7.8 & 7.3 & 5.8 & 6.8 & 7.2 & 7.7 \\
    (tvFAR1) & 8.5 & 8.7 & 8.2 & 8.7 & 8.4 & 7.1 & 7.8 & 8.4 & 7.9 & 8.9 & 6.9 \\
    (tvFAR2) & 7.9 & 5.8 & 6.9 & 6.9 & 7.9 & 6.3 & 6.7 & 6.9 & 7.8 & 5.5 & 7.6 \\
    \midrule
    \multicolumn{12}{l}{\textit{Panel B: $n = 200$}}\\    
    (BM) & 5.2 & 5.1 & 6.2 & 4.6 & 5.3 & 5.6 & 4.6 & 6.1 & 5.4 & 4.8 & 5.8 \\
    (BB) & 5.1 & 3.8 & 4.4 & 5.4 & 4.1 & 5.0 & 4.3 & 5.6 & 5.3 & 5.0 & 5.0 \\
    (FAR-BM) & 5.4 & 5.9 & 6.4 & 4.1 & 5.7 & 5.3 & 6.1 & 5.4 & 5.8 & 6.6 & 5.4 \\
    (FAR-BB) & 5.6 & 4.5 & 6.0 & 5.9 & 6.6 & 4.4 & 5.5 & 6.0 & 5.5 & 4.8 & 5.5 \\
    (tvBM) & 7.9 & 7.2 & 8.3 & 7.5 & 7.9 & 7.7 & 8.6 & 7.3 & 6.3 & 7.4 & 6.1 \\
    (tvFAR1) & 7.4 & 7.5 & 8.5 & 7.4 & 7.0 & 7.9 & 7.2 & 6.7 & 7.1 & 6.8 & 7.1 \\
    (tvFAR2) & 6.7 & 8.0 & 7.5 & 5.0 & 8.0 & 5.5 & 5.7 & 6.0 & 4.6 & 5.8 & 5.6 \\
    \midrule
    \multicolumn{12}{l}{\textit{Panel C: $n = 500$}}\\    
    (BM) & 6.4 & 4.7 & 5.2 & 5.2 & 4.7 & 5.6 & 5.4 & 5.3 & 5.6 & 4.7 & 6.1 \\
    (BB) & 5.5 & 6.2 & 2.9 & 6.5 & 4.2 & 5.7 & 5.3 & 4.7 & 5.6 & 4.3 & 5.3 \\
    (FAR-BM) & 5.5 & 5.2 & 5.5 & 5.8 & 6.0 & 5.9 & 4.7 & 4.7 & 5.0 & 5.8 & 4.9 \\
    (FAR-BB) & 5.2 & 5.9 & 7.0 & 6.1 & 5.0 & 5.2 & 5.4 & 4.6 & 5.6 & 5.6 & 6.0 \\
    (tvBM) & 9.0 & 7.7 & 7.6 & 7.0 & 7.4 & 6.2 & 6.1 & 7.3 & 8.1 & 5.8 & 6.8 \\
    (tvFAR1) & 7.5 & 7.5 & 6.8 & 9.3 & 6.7 & 7.9 & 10.1 & 8.1 & 7.3 & 7.5 & 6.8 \\
    (tvFAR2) & 5.1 & 6.1 & 4.9 & 4.9 & 5.6 & 5.5 & 7.4 & 4.9 & 6.2 & 5.5 & 6.1 \\
    \bottomrule
    \end{tabular}
\end{table}

\begin{table}[ht]
    \centering
    \small
    \setlength{\tabcolsep}{4pt}
    \caption{Empirical rejection rates (in \%) of the test SN-$b_1$ for a gradual change, for different sample sizes $n$, error processes $\eps$ and heights $\delta$.}
    \label{tab:gradual_sn_b1}
    \begin{tabular}{l|rrrrrrrrrrr}
    \toprule
    $\delta$ & 0.0 & 0.1 & 0.2 & 0.3 & 0.4 & 0.5 & 0.6 & 0.7 & 0.8 & 0.9 & 1.0 \\
    \midrule
    \multicolumn{12}{l}{\textit{Panel A: $n = 100$}}\\    
    (BM) & 5.4 & 22.9 & 62.2 & 87.8 & 97.1 & 98.4 & 99.6 & 99.9 & 100.0 & 99.9 & 100.0 \\
    (BB) & 6.2 & 61.2 & 97.1 & 99.8 & 100.0 & 100.0 & 100.0 & 100.0 & 100.0 & 100.0 & 100.0 \\
    (FAR-BM) & 5.0 & 16.5 & 46.4 & 72.0 & 92.2 & 97.0 & 98.9 & 99.3 & 99.7 & 99.9 & 100.0 \\
    (FAR-BB) & 4.4 & 55.3 & 94.0 & 99.6 & 100.0 & 100.0 & 100.0 & 100.0 & 100.0 & 100.0 & 100.0 \\
    (tvBM) & 6.1 & 23.0 & 64.1 & 89.8 & 98.5 & 99.0 & 99.8 & 99.8 & 100.0 & 100.0 & 100.0 \\
    (tvFAR1) & 6.6 & 20.5 & 46.5 & 76.8 & 93.0 & 98.2 & 98.9 & 99.3 & 99.8 & 100.0 & 100.0 \\
    (tvFAR2) & 7.1 & 15.7 & 44.7 & 74.1 & 91.5 & 96.7 & 98.5 & 99.4 & 99.6 & 100.0 & 99.8 \\
    \midrule
    \multicolumn{12}{l}{\textit{Panel B: $n = 200$}}\\    
    (BM) & 6.2 & 38.3 & 85.2 & 98.4 & 99.9 & 99.9 & 100.0 & 100.0 & 100.0 & 100.0 & 100.0 \\
    (BB) & 5.1 & 85.6 & 99.4 & 100.0 & 100.0 & 100.0 & 100.0 & 100.0 & 100.0 & 100.0 & 100.0 \\
    (FAR-BM) & 5.8 & 25.8 & 68.3 & 92.7 & 97.8 & 99.8 & 100.0 & 99.9 & 100.0 & 100.0 & 100.0 \\
    (FAR-BB) & 6.4 & 82.4 & 98.8 & 99.9 & 100.0 & 100.0 & 100.0 & 100.0 & 100.0 & 100.0 & 100.0 \\
    (tvBM) & 6.9 & 40.5 & 88.4 & 98.6 & 99.8 & 99.9 & 100.0 & 100.0 & 100.0 & 100.0 & 100.0 \\
    (tvFAR1) & 7.2 & 27.8 & 71.6 & 93.2 & 99.3 & 99.5 & 99.8 & 100.0 & 100.0 & 100.0 & 100.0 \\
    (tvFAR2) & 6.1 & 31.2 & 67.5 & 92.8 & 98.4 & 99.4 & 100.0 & 100.0 & 100.0 & 100.0 & 100.0 \\
    \midrule
    \multicolumn{12}{l}{\textit{Panel C: $n = 500$}}\\    
    (BM) & 6.1 & 70.8 & 98.8 & 99.9 & 100.0 & 100.0 & 100.0 & 100.0 & 100.0 & 100.0 & 100.0 \\
    (BB) & 3.7 & 98.5 & 100.0 & 100.0 & 100.0 & 100.0 & 100.0 & 100.0 & 100.0 & 100.0 & 100.0 \\
    (FAR-BM) & 5.8 & 53.1 & 94.8 & 99.6 & 99.9 & 100.0 & 100.0 & 100.0 & 100.0 & 100.0 & 100.0 \\
    (FAR-BB) & 6.1 & 97.2 & 100.0 & 100.0 & 100.0 & 100.0 & 100.0 & 100.0 & 100.0 & 100.0 & 100.0 \\
    (tvBM) & 6.1 & 72.1 & 99.4 & 99.9 & 100.0 & 100.0 & 100.0 & 100.0 & 100.0 & 100.0 & 100.0 \\
    (tvFAR1) & 6.6 & 55.1 & 95.9 & 99.6 & 100.0 & 100.0 & 100.0 & 100.0 & 100.0 & 100.0 & 100.0 \\
    (tvFAR2) & 5.2 & 52.9 & 94.3 & 99.4 & 100.0 & 99.9 & 100.0 & 100.0 & 100.0 & 100.0 & 100.0 \\
    \bottomrule
    \end{tabular}
\end{table}

\begin{table}[ht]
    \centering
    \small
    \setlength{\tabcolsep}{4pt}
    \caption{Empirical rejection rates (in \%) of the test BS-sup for a gradual change, for different sample sizes $n$, error processes $\eps$ and heights $\delta$.}
    \label{tab:gradual_bs_sup}
    \begin{tabular}{l|rrrrrrrrrrr}
    \toprule
    $\delta$ & 0.0 & 0.1 & 0.2 & 0.3 & 0.4 & 0.5 & 0.6 & 0.7 & 0.8 & 0.9 & 1.0 \\
    \midrule
    \multicolumn{12}{l}{\textit{Panel A: $n = 100$}}\\    
    (BM) & 7.6 & 10.4 & 20.8 & 36.8 & 63.1 & 86.2 & 96.7 & 99.8 & 100.0 & 100.0 & 100.0 \\
    (BB) & 6.2 & 18.1 & 64.1 & 96.6 & 99.9 & 100.0 & 100.0 & 100.0 & 100.0 & 100.0 & 100.0 \\
    (FAR-BM) & 11.6 & 14.2 & 20.4 & 28.2 & 48.4 & 69.9 & 84.4 & 93.2 & 98.6 & 99.5 & 99.9 \\
    (FAR-BB) & 9.5 & 18.7 & 52.6 & 89.0 & 99.3 & 100.0 & 100.0 & 100.0 & 100.0 & 100.0 & 100.0 \\
    (tvBM) & 6.5 & 9.9 & 18.2 & 33.2 & 53.8 & 76.4 & 89.4 & 95.4 & 99.4 & 99.8 & 99.9 \\
    (tvFAR1) & 12.2 & 16.4 & 21.0 & 27.9 & 41.2 & 60.9 & 71.9 & 88.6 & 93.7 & 98.1 & 99.3 \\
    (tvFAR2) & 14.8 & 13.8 & 18.9 & 32.4 & 48.2 & 66.8 & 83.9 & 94.3 & 97.8 & 99.5 & 99.9 \\
    \midrule
    \multicolumn{12}{l}{\textit{Panel B: $n = 200$}}\\    
    (BM) & 6.6 & 11.5 & 30.6 & 67.7 & 95.9 & 99.7 & 100.0 & 100.0 & 100.0 & 100.0 & 100.0 \\
    (BB) & 7.2 & 36.0 & 95.4 & 100.0 & 100.0 & 100.0 & 100.0 & 100.0 & 100.0 & 100.0 & 100.0 \\
    (FAR-BM) & 12.6 & 14.6 & 27.7 & 50.7 & 79.1 & 96.1 & 99.7 & 100.0 & 100.0 & 100.0 & 100.0 \\
    (FAR-BB) & 8.3 & 32.6 & 84.9 & 99.9 & 100.0 & 100.0 & 100.0 & 100.0 & 100.0 & 100.0 & 100.0 \\
    (tvBM) & 7.0 & 12.2 & 28.2 & 55.5 & 87.7 & 98.6 & 99.9 & 100.0 & 100.0 & 100.0 & 100.0 \\
    (tvFAR1) & 10.9 & 12.7 & 23.4 & 44.3 & 68.1 & 88.7 & 97.8 & 99.5 & 100.0 & 100.0 & 100.0 \\
    (tvFAR2) & 11.9 & 18.4 & 26.9 & 49.8 & 77.3 & 95.8 & 99.4 & 100.0 & 100.0 & 100.0 & 100.0 \\
    \midrule
    \multicolumn{12}{l}{\textit{Panel C: $n = 500$}}\\    
    (BM) & 6.5 & 20.0 & 75.1 & 99.8 & 100.0 & 100.0 & 100.0 & 100.0 & 100.0 & 100.0 & 100.0 \\
    (BB) & 5.8 & 77.7 & 100.0 & 100.0 & 100.0 & 100.0 & 100.0 & 100.0 & 100.0 & 100.0 & 100.0 \\
    (FAR-BM) & 9.8 & 18.0 & 49.8 & 94.3 & 100.0 & 100.0 & 100.0 & 100.0 & 100.0 & 100.0 & 100.0 \\
    (FAR-BB) & 8.8 & 63.2 & 100.0 & 100.0 & 100.0 & 100.0 & 100.0 & 100.0 & 100.0 & 100.0 & 100.0 \\
    (tvBM) & 6.1 & 17.0 & 59.9 & 98.1 & 100.0 & 100.0 & 100.0 & 100.0 & 100.0 & 100.0 & 100.0 \\
    (tvFAR1) & 11.0 & 14.6 & 39.1 & 82.5 & 99.0 & 99.9 & 100.0 & 100.0 & 100.0 & 100.0 & 100.0 \\
    (tvFAR2) & 8.9 & 17.9 & 46.7 & 89.6 & 99.9 & 100.0 & 100.0 & 100.0 & 100.0 & 100.0 & 100.0 \\
    \bottomrule
    \end{tabular}
\end{table}

\begin{table}[ht]
    \centering
    \small
    \setlength{\tabcolsep}{4pt}
    \caption{Empirical rejection rates (in \%) of the test BS-$L^2$ for a gradual change, for different sample sizes $n$, error processes $\eps$ and heights $\delta$.}
    \label{tab:gradual_bs_l2}
    \begin{tabular}{l|rrrrrrrrrrr}
    \toprule
    $\delta$ & 0.0 & 0.1 & 0.2 & 0.3 & 0.4 & 0.5 & 0.6 & 0.7 & 0.8 & 0.9 & 1.0 \\
    \midrule
    \multicolumn{12}{l}{\textit{Panel A: $n = 100$}}\\    
    (BM) & 8.4 & 11.0 & 18.3 & 36.2 & 72.6 & 95.1 & 99.7 & 100.0 & 100.0 & 100.0 & 100.0 \\
    (BB) & 7.6 & 15.5 & 73.3 & 99.5 & 100.0 & 100.0 & 100.0 & 100.0 & 100.0 & 100.0 & 100.0 \\
    (FAR-BM) & 13.8 & 14.5 & 19.7 & 25.7 & 46.6 & 73.5 & 91.5 & 98.3 & 99.7 & 100.0 & 100.0 \\
    (FAR-BB) & 12.0 & 16.3 & 51.4 & 94.8 & 100.0 & 100.0 & 100.0 & 100.0 & 100.0 & 100.0 & 100.0 \\
    (tvBM) & 7.1 & 10.4 & 16.5 & 34.3 & 58.9 & 87.0 & 95.8 & 99.5 & 100.0 & 100.0 & 100.0 \\
    (tvFAR1) & 13.5 & 16.3 & 19.8 & 25.3 & 40.1 & 63.8 & 78.6 & 93.3 & 97.4 & 99.9 & 99.9 \\
    (tvFAR2) & 16.4 & 13.9 & 19.6 & 30.9 & 46.5 & 69.6 & 92.8 & 97.7 & 99.7 & 99.9 & 100.0 \\
    \midrule
    \multicolumn{12}{l}{\textit{Panel B: $n = 200$}}\\    
    (BM) & 7.0 & 10.0 & 29.0 & 76.8 & 99.7 & 100.0 & 100.0 & 100.0 & 100.0 & 100.0 & 100.0 \\
    (BB) & 6.6 & 29.9 & 99.3 & 100.0 & 100.0 & 100.0 & 100.0 & 100.0 & 100.0 & 100.0 & 100.0 \\
    (FAR-BM) & 12.4 & 15.1 & 24.9 & 48.3 & 85.3 & 99.2 & 100.0 & 100.0 & 100.0 & 100.0 & 100.0 \\
    (FAR-BB) & 9.4 & 23.9 & 90.6 & 100.0 & 100.0 & 100.0 & 100.0 & 100.0 & 100.0 & 100.0 & 100.0 \\
    (tvBM) & 8.2 & 11.6 & 26.9 & 63.7 & 95.9 & 99.8 & 100.0 & 100.0 & 100.0 & 100.0 & 100.0 \\
    (tvFAR1) & 12.0 & 13.8 & 21.4 & 38.7 & 73.9 & 95.5 & 99.2 & 99.9 & 100.0 & 100.0 & 100.0 \\
    (tvFAR2) & 13.4 & 17.0 & 25.9 & 48.9 & 82.5 & 98.8 & 100.0 & 100.0 & 100.0 & 100.0 & 100.0 \\
    \midrule
    \multicolumn{12}{l}{\textit{Panel C: $n = 500$}}\\    
    (BM) & 6.1 & 16.4 & 85.3 & 100.0 & 100.0 & 100.0 & 100.0 & 100.0 & 100.0 & 100.0 & 100.0 \\
    (BB) & 5.8 & 86.6 & 100.0 & 100.0 & 100.0 & 100.0 & 100.0 & 100.0 & 100.0 & 100.0 & 100.0 \\
    (FAR-BM) & 10.4 & 14.9 & 48.1 & 97.7 & 100.0 & 100.0 & 100.0 & 100.0 & 100.0 & 100.0 & 100.0 \\
    (FAR-BB) & 9.4 & 56.8 & 100.0 & 100.0 & 100.0 & 100.0 & 100.0 & 100.0 & 100.0 & 100.0 & 100.0 \\
    (tvBM) & 6.4 & 15.2 & 68.0 & 99.8 & 100.0 & 100.0 & 100.0 & 100.0 & 100.0 & 100.0 & 100.0 \\
    (tvFAR1) & 11.1 & 13.8 & 36.3 & 88.0 & 100.0 & 100.0 & 100.0 & 100.0 & 100.0 & 100.0 & 100.0 \\
    (tvFAR2) & 8.8 & 16.0 & 42.0 & 96.9 & 100.0 & 100.0 & 100.0 & 100.0 & 100.0 & 100.0 & 100.0 \\
    \bottomrule
    \end{tabular}
\end{table}